\documentclass[11pt, letterpaper]{amsart}
 \usepackage{amsmath, amsthm, amssymb, amsfonts}
 \usepackage[foot]{amsaddr}
\usepackage[utf8]{inputenc}
\usepackage[normalem]{ulem}
\usepackage[colorlinks = true,
            linkcolor = black,
            urlcolor  = blue,
            citecolor = red,
            anchorcolor = red,
            hidelinks]{hyperref}
\usepackage{url}
\usepackage{longtable} 
\usepackage{enumitem}
\usepackage{relsize}
\usepackage{bm}
\usepackage[dvipsnames]{xcolor}
\usepackage{graphicx}
\usepackage{mathabx}
\usepackage{mathdots}
\usepackage{mathtools}
\usepackage{color}
\usepackage{tikz-cd,tikz}
\usepackage{todonotes}
\usepackage{braket}
\usepackage{stmaryrd}
\usepackage{amssymb}
\usepackage{tipa}
\usepackage[margin=1.2 in]{geometry}
\usepackage{appendix}
\usepackage{physics}
\usepackage{mathrsfs}  

\usetikzlibrary{arrows, decorations.markings}
\usetikzlibrary{positioning}
\tikzstyle{vecArrow} = [thick, decoration={markings,mark=at position
   1 with {\arrow[semithick]{open triangle 60}}},
   double distance=1.4pt, shorten >= 5.5pt,
   preaction = {decorate},
   postaction = {draw,line width=1.4pt, white,shorten >= 4.5pt}]
\tikzstyle{innerWhite} = [semithick, white,line width=1.4pt, shorten >= 4.5pt]
\tikzstyle{vecEq} = [thick,
   double distance=1.4pt,
   preaction = {decorate},
   postaction = {draw,line width=1.4pt, white,shorten >= 4.5pt}]
\usetikzlibrary{matrix}
\usetikzlibrary{shapes}
\usetikzlibrary{arrows,decorations.markings, tikzmark}
\usepackage{stmaryrd}

\pgfarrowsdeclare{bad to}{bad to}
{
  \pgfarrowsleftextend{-2\pgflinewidth}
  \pgfarrowsrightextend{\pgflinewidth}
}
{
  \pgfsetlinewidth{0.8\pgflinewidth}
  \pgfsetdash{}{0pt}
  \pgfsetroundcap
  \pgfsetroundjoin
  \pgfpathmoveto{\pgfpoint{-3\pgflinewidth}{4\pgflinewidth}}
  \pgfpathcurveto
  {\pgfpoint{-2.75\pgflinewidth}{2.5\pgflinewidth}}
  {\pgfpoint{0pt}{0.25\pgflinewidth}}
  {\pgfpoint{0.75\pgflinewidth}{0pt}}
  \pgfpathcurveto
  {\pgfpoint{0pt}{-0.25\pgflinewidth}}
  {\pgfpoint{-2.75\pgflinewidth}{-2.5\pgflinewidth}}
  {\pgfpoint{-3\pgflinewidth}{-4\pgflinewidth}}
  \pgfusepathqstroke
}

\newtheoremstyle{user}
{}
{}
{\normalfont}
{}
{\bfseries}
{.}
{2ex}
{\thmname{#1}\thmnumber{ #2}\thmnote{ \textnormal{#3}}}
\theoremstyle{user}
\newtheorem{definition}{Definition}[section]
\newtheorem{assumption}[definition]{Assumption}
\newtheorem{example}[definition]{Example}
\newtheorem{remark}[definition]{Remark}

\theoremstyle{theorem}
\newtheorem{theorem}[definition]{Theorem}
\newtheorem{proposition}[definition]{Proposition}

\newtheorem{corollary}[definition]{Corollary}

\newtheorem{lemma}[definition]{Lemma}

\newcommand{\un}[1]{\underline{#1}}
\newcommand{\ov}[1]{\overline{#1}}
\newcommand\numberthis{\addtocounter{equation}{1}\tag{\theequation}}

\newcommand{\Mod}{\mathsf{Mod}}

\newcommand{\BA}{{\mathbb{A}}}

\newcommand{\CC}{{\mathbb{C}}}
\newcommand{\DD}{{\mathbb{D}}}
\newcommand{\EE}{{\mathbb{E}}}

\newcommand{\GG}{{\mathbb{G}}}
\newcommand{\HH}{{\mathbb{H}}}

\newcommand{\NN}{{\mathbb{N}}}

\newcommand{\PP}{{\mathbb{P}}}
\newcommand{\QQ}{{\mathbb{Q}}}
\newcommand{\RR}{{\mathbb{R}}}
\newcommand{\TT}{{\mathbb{T}}}
\newcommand{\UU}{{\mathbb{U}}}

\newcommand{\ZZ}{{\mathbb{Z}}}

\newcommand{\mA}{{\mathcal A}}

\newcommand{\mE}{{\mathcal E}}
\newcommand{\mF}{{\mathcal F}}

\newcommand{\mM}{{\mathcal M}}

\newcommand{\mO}{{\mathcal O}}
\newcommand{\mP}{{\mathcal P}}
\newcommand{\mQ}{{\mathcal Q}}
\newcommand{\mR}{{\mathcal R}}

\newcommand{\mT}{{\mathcal T}}
\newcommand{\mU}{{\mathcal U}}
\newcommand{\mV}{{\mathcal V}}

\newcommand{\Ff}{{\mathfrak{f}}}

\newcommand{\Fr}{{\mathfrak{r}}}

\newcommand{\crU}{\mathscr{U}}

\DeclareMathOperator{\ad}{ad}

\DeclareMathOperator{\Hilb}{Hilb}

\DeclareMathOperator{\Sym}{Sym}
\DeclareMathOperator{\id}{id}

\newcommand{\Cone}{\mathop{\rm Cone}\nolimits}

\newcommand{\Ext}{\mathcal{E}\textnormal{xt}}

\newcommand{\Bl}{\mathop{\rm Bl}\nolimits}

\newcommand{\Dec}{\text{Dec}}
\newcommand{\Qfr}{^{\ov{n}}_{\infty}Q}
\newcommand{\ovninfty}{{}^{\ov{n}}_{\infty}}
\newcommand{\mufr}{^{\ov{n}}_{\infty}\mu}
\newcommand{\nufr}{^{\ov{n}}_{\infty}\nu}

\newcommand{\End}{{\rm End}}
\newcommand{\Rep}{{\rm Rep}}

\newcommand{\nd}{\mathop{\rm nd}\nolimits}
\newcommand{\pe}{\mathop{\rm pe}\nolimits}
\newcommand{\pa}{\mathop{\rm pa}\nolimits}
\newcommand{\fr}{\mathop{\rm fr}\nolimits}
\newcommand{\tet}{\mathop{\rm t}\nolimits}

\newcommand{\bL}{{\mathsf{L}}}

\newcommand{\Linv}{\mathsf{L}_{\inv}}

\newcommand{\bS}{{\mathsf{S}}}
\newcommand{\bT}{{\mathsf{T}}}

\newcommand{\bR}{\mathsf{R}}

\newcommand{\pt}{{\mathsf{pt}}}

\newcommand{\congpf}{\xymatrix@1@=15pt{\ar[r]^-\sim&}}

\newcommand{\Mf}{\mathfrak{M}}

\newcommand{\MQ}{\mathcal{M}_Q}

\newcommand{\Flag}{\mathrm{Flag}}

\newcommand{\ch}{\mathrm{ch}}

\newcommand{\td}{\mathrm{td}}
\newcommand{\Coh}{\mathrm{Coh}}

\newcommand{\Perf}{\mathrm{Perf}}
\newcommand{\Stab}{\mathrm{Stab}}
\newcommand{\rep}{\mathrm{Rep}}
\newcommand{\Hom}{\mathrm{Hom}}
\newcommand{\uHom}{\mathcal{H}\mathrm{om}}

\newcommand{\vir}{\mathrm{vir}}

\newcommand{\sym}{\textnormal{sym}}
\newcommand{\SSym}{\textnormal{SSym}}

\newcommand{\rk}{\mathrm{rk}}

\newcommand{\rig}{\mathrm{rig}}
\newcommand{\inv}{{\mathrm{wt}_0}}
\newcommand{\inva}{\mathrm{in}}

\usepackage{tikz}
\usepackage{tikz-cd}
\usetikzlibrary{decorations.pathmorphing}
\AtBeginDocument{%
	\def\MR#1{}
}

\renewcommand{\Ext}{\textup{Ext}}

\makeatletter
\newcommand\sqmatrix[2][c]{%
  \fixTABwidth{T}%
  \setbox0=\hbox{$\tabbedCenterstack{#2}$}%
  \setstackgap{L}{\dimexpr\maxTAB@width+\tabbed@gap}%
  \tabbedCenterstack[#1]{#2}%
}
\makeatother

\begin{document}

\baselineskip=16pt
\parskip=5pt

    \title{Universal Virasoro Constraints for Quivers with Relations}

\author{Arkadij Bojko}
\email{abojko@gate.sinica.edu.tw}

\address{Institute of Mathematics, Academia Sinica,
6F, Astronomy-Mathematics Building,
No. 1, Sec. 4, Roosevelt Road,
Taipei 10617, Taiwan}

\date{\today}
\maketitle
\begin{abstract}
Following our reformulation of sheaf-theoretic Virasoro constraints with applications to curves and surfaces joint with Lim--Moreira, I describe in the present work the quiver analog. After phrasing a universal approach to Virasoro constraints for moduli of quiver-representations, I prove them for any finite quiver with relations, with frozen vertices, but without cycles. I use partial flag varieties which are special cases of moduli spaces of framed representations as a guiding example throughout. These results are applied to give an independent proof of Virasoro constraints for all Gieseker semistable sheaves on $\PP^2$ and $\PP^1 \times \PP^1$ by using derived equivalences to quivers with relations. Combined with an existing universality argument for Virasoro constraints on Hilbert schemes of points on surfaces, this leads to the proof of this rank 1 case for any $S$ which is independent of the previous results in Gromov--Witten theory.
\end{abstract}
\setcounter{tocdepth}{1} 
\tableofcontents

\section{Introduction}
Witten in his influential work \cite{witten} conjectured that the descendant integrals on $\ov{\mM}_{g,n}$ are determined by a family of partial differential equations called the KdV hierarchy together with the string equation. This came to be known as Witten's conjecture and it was proved in the celebrated work of Kontsevich \cite{kontsevich}. Its legendary status is supported by the many different proofs that followed, e.g. in the works of Okounkov--Pandharipande \cite{OPwitten} and Mirzakhani \cite{mirzakhani}. 

The first appearance of the Virasoro algebra in connection with Witten's conjecture traces back to \cite{DVV} where it was recast in terms of differential operators $L_k$ for $k\geq -1$ satisfying Virasoro commutation relations. This perspective garnered further popularity when \cite{ehx} proposed how to extend it to $\ov{\mM}_{g,n}(X)$ for some varieties $X$. This was followed by a further generalization by Katz (presented in \cite{EJX}) dealing with off-diagonal Hodge cohomology. The methods developed by Givental in \cite{givental} were applied in \cite{flag, grass, iritani, teleman} to prove Virasoro constraints in many cases. An entirely orthogonal approach was used by Okounkov--Pandharipande \cite{OP} to address smooth curves.

The correspondence \cite{GW/PT} relating descendent Gromov--Witten (GW) invariants and descendent Pandharipande--Thomas (PT) invariants of projective toric 3-folds naturally led to the question of compatible Virasoro constraints on the PT side. This has been answered affirmatively in the stationary setting in \cite{moop} and motivated further investigations of Virasoro constraints for moduli of sheaves in \cite{moreira, bree}. The proofs addressing Hilbert schemes for surfaces found in \cite{moop, moreira} relied on the dimensional reduction from PT invariants of toric 3-folds and a universality argument.

The problem of motivating sheaf-theoretic Virasoro constraints and giving them meaning was resolved in \cite{BML} where we showed that they can be derived from a natural conformal element on the homology of a large stack containing all moduli schemes of sheaves. This homology was given a vertex algebra structure by Joyce \cite{Jo17} which was used to express wall-crossing formulae in \cite{GJT, Jo21}. We used this conceptual wording of Virasoro constraints compatible with wall-crossing to prove them in many cases including (semi)stable higher rank torsion-free sheaves on surfaces with only diagonal Hodge cohomology groups.

One of the most striking observations resulting from our approach was its apparent generality. As there was nothing special about working with sheaves, one could hope to introduce and prove Virasoro constraints for moduli of objects in any nice additive category with a well-behaved deformation-obstruction theory. The present work does so for representations of quivers with relations and frozen vertices.\footnote{The current version is an update influenced  by i) some comments by previous referees ii) the follow-up work by Lim--Moreira \cite{MW} which used some of the ideas presented here iii) a fruitful email exchange with P. Bousseau.} 

I describe a useful application of the above general framework to sheaves on surfaces. It uses derived equivalences of the varieties $\PP^2$ and $\PP^1\times \PP^1$ to quivers with relations recalled in §\ref{sec:P2}. Using Virasoro constraints for quivers, I prove them for all Gieseker semistable sheaves on the above two surfaces\footnote{And on $\Bl_{\pt}\PP^2$ following the same proof.}. Here Gieseker stability is used in the sense of \cite[Definition 1.2.4]{HL}, so it includes rank 0 sheaves. Due to the universality of Virasoro constraints for $\Hilb^n(S)$ which first appeared in \cite{moreira} in a more restrictive form and was extended to full generality by Joyce in \cite{Jo23}, it is sufficient to prove the Hilbert scheme constraints for $S=\PP^2,\PP^1\times \PP^1$ to establish them for any surface $S$. The previously available proof relied on results in GW theory through \cite{moop} so the present work severs this connection and thus establishes sheaf-theoretic Virasoro constraints for surfaces as an independent fundamental phenomenon. 

Lastly, Virasoro constraints for semistable one-dimensional sheaves on $\PP^2$ proved here are used to study relations on generators of the cohomology of their moduli stacks in \cite{KPMW}.

\subsection{Results about intersection theory of classical moduli schemes}
If $Q$ is a quiver, and $\mR$ is a two-sided ideal of relations as in Definition \ref{def:homorela}, then the representations of $(Q,\mR)$ can have many interpretations even beyond their algebraic origin.
\begin{enumerate}
    \item Any finite-dimensional algebra $A$ over $\CC$ can be modelled by some $(Q,\mR)$. This is because the category $A$-mod of finite dimensional $A$-left modules is equivalent to the category $\rep(Q,\mR)$ of finite dimensional representations for some $(Q,\mR)$ (see \cite[§4.1]{benson}). Studying moduli schemes of semistable A-modules corresponds therefore to working with semistable representations.
    \item Stable framed quiver representations from Definition \ref{def:framedrep1} for an acyclic quiver $Q$ have a particularly nice description. Starting for example from $Q=A_{l-1}$ as depicted in \eqref{Eq:equiv},  one can identify any length $l$ partial flag variety with a moduli scheme of stable framed representations (see Example \ref{ex:flag}). More generally, any moduli scheme of framed representations of an acyclic quiver can be constructed from a point as an iterated Grassmann bundle by \cite[Theorem 4.10]{reineke}. Any statement about the intersection theory of such a moduli scheme can thus be phrased in terms of identities for symmetric polynomials. 
    \item If $X$ is a smooth projective variety that admits a strong full exceptional collection, then it is derived equivalent to a quiver with relations whose vertices are the exceptional objects of the collection. This was shown by Bondal in \cite[Theorem 6.2]{bondal}, and it applies to any del Pezzo or toric surface $S$ by \cite{Orlov, Kawamata}. For these two classes of surfaces, one should be able to use the derived equivalences to express $\Hilb^n(S)$ or more generally Gieseker semistable sheaves in terms of Bridgeland stability conditions for quivers. Due to the universality of Virasoro constraints for $\Hilb^n(S)$, it is sufficient to treat the two cases $S=\PP^2,\PP^1\times \PP^1$. I do so in §\ref{sec:P2} together with $S=\Bl_{\pt}\PP^2$ culminating in Theorem \ref{thm:P2indep}.\footnote{The case $S=\Bl_{\pt}\PP^2$ was added because in the case of dimension 1 semistable sheaves this is a strictly speaking new result, and I was already using all the references necessary to prove it anway. It does not play a major role in the present work, so I will not state it in the main theorems in the introduction.}
\end{enumerate}
The perspectives offered by each one of these points are studied in the main body of this work. Weaker versions of the results proved here are stated in the next theorem -- the full statements allow strictly semistable objects. The definitions of Virasoro constraints that are being addressed will be recalled in the next subsection.
\begin{theorem}
\label{thm:first}
\begin{enumerate}[label = (\Roman*)]
    \item For a finite dimensional algebra $A$ over $\CC$, fix a Bridgeland stability condition $\sigma$ constructed on the heart $A$-mod of $D^b(A)$ by using \cite[Proposition 5.3]{bridgeland}. The moduli scheme $M^{\sigma}_{\alpha}$ of $\sigma$-semistable left-modules with class $\alpha\in K^0(A)$ satisfies Virasoro constraints when there are no strictly semistables if there exists such a stability condition $\sigma'$ with $M^{\sigma'}_{\beta}$ satisfying Virasoro constraints for all $\beta$. One may take $A$ to be the path algebra of $(Q,\mR)$ for an acyclic quiver $Q$. In this case, the condition on the existence of $\sigma'$ is satisfied, and the slope stability from \eqref{eq:mustability} is an example of a Bridgeland stability $\sigma$.
\item Let $(Q, \mR)$ be a quiver with relations, then for any dimension vector $\ov{d}$ and any framing vector $\ov{n}$, the moduli scheme of stable framed representation $M^{\ov{n}}_{Q,\ov{d}}$ from Definition \ref{def:framedrep1} satisfies Virasoro constraints. By Example \ref{ex:flag}, this includes Virasoro constraints for partial flag varieties. 
\item Let $S=\PP^2$ or $\PP^1\times \PP^1$. The moduli spaces $M^{p}_{r,d,n}$ of Gieseker stable sheaves on $S$ with Chern character $(r,d,n)\in H^\bullet(S)$ satisfy Virasoro constraints whenever there are no strictly semistables.\footnote{} 
\end{enumerate}
\end{theorem}
The first of these results is an application of Proposition \ref{prop:bridgeland} which generalizes Theorem \ref{thm:main} from $\mu$-stabilities to Bridgeland stabilities. It is worth noting that the proof of wall-crossing for such Bridgeland stability conditions of a quiver requires a separate check that Joyce's wall-crossing assumptions from \cite[§5.1, §5.2]{Jo21} hold. This takes up the majority of the proof of Proposition \ref{prop:bridgeland}.
The result about quiver representations is then transformed into a statement about finite dimensional algebras through the appropriate Morita equivalences. 

Part (II) follows from Theorem \ref{thm:main} by applying the alternative description of framed quiver representations from Definition \ref{def:framedrep2}. In the case of the $A_{n}$ quiver, one obtains Virasoro constraints for partial flag varieties which can be formulated in terms of integrals of Chern characters of universal quotient bundles (see Example \ref{ex:flag}). 

Finally, part (III) is stated in Theorem \ref{thm:P2indep} and makes use of Proposition \ref{prop:bridgeland}. A more general version of all three statements holds because I prove them even in the presence of strictly semistable objects using Theorem \ref{thm:quivvirWC}. I will recall this general formulation and its geometric nature before stating the mentioned theorem. Recently, the authors of \cite{MW} proved (III) but without strictly semistables following my approach that originally only dealt with positive rank. 

\subsection{What are Virasoro constraints for quivers?}
\label{sec:whatare}
The traditional formulation of Virasoro constraints for a moduli scheme $M$ is in terms of \textit{descendents} and their integrals with respect to the (virtual) fundamental class $[M]$. This makes Virasoro constraints a property of $[M]$ rather than the moduli schemes itself as different choices of obstruction theories may lead to different sets of conditions. I ignored this point in the statement of Theorem \ref{thm:first} because the virtual fundamental classes are naturally determined by the data specified there.

In §\ref{sec:background}, I recall the basics behind working with quivers. For the purposes of the introduction, I work with a $Q$ that has no cycles (see the discussion in point (3) in §\ref{sec:future}) but will partially remove this restriction later. Additionally, fix a two-sided ideal of relations $\mR$ with a homogeneous set of generators $R$ as in Definition \ref{def:homorela}. In \cite{BML}, we distinguished between sheaves $G$ and pairs of sheaves $V\to F$ with the term $V$ fixed. This is covered in the present work by introducing frozen vertices denoted by $\diamond$. I always assume that no relations begin at a frozen vertex.

When studying representations of $Q$ recalled in §\ref{sec:background}, the vector spaces at each frozen vertex do not vary. This affects the moduli schemes of representations and their obstruction theory in the same way as it did for pairs of sheaves. Note that the Grothendieck group $K^0(Q,\mR)$ of $\Rep(Q,\mR)$ is isomorphic to $\ZZ^{V}$ where $V$ is the set of vertices of $Q$. The dimension vector $\ov{d}\in \ZZ^V$ is thus used to determine the type of a representation of $(Q,\mR)$ as it varies in its moduli scheme. For framed representations, one additionally needs to specify the framing vector $\ov{n}\in \ZZ^V$ which determines the dimension of the framing at each vertex as in Definition \ref{def:framedrep1}. I always assume that $\ov{d},\ov{n}\neq 0$. Additionally, I will take it as given that $\mR$ and $R$ are included and known, and I will not mention them from now on.

Suppose that $\sigma$ is a stability condition on $\rep(Q)$. When there are no strictly $\sigma$-semistables, the moduli scheme $M^{\sigma}_{Q,\ov{d}}$ of $\sigma$-stable representations with dimension vector $\ov{d}$ is often known to be projective thanks to the work of King \cite{King}. For a chosen $R$, there is a fixed perfect obstruction theory from Example \ref{ex:Tvir} (for a wrong choice of $R$, this may not be an interesting obstruction theory as explained in Remark \ref{rem:wrongobstruction}). The resulting (virtual) fundamental class 
$$
[M^{\sigma}_{Q,\ov{d}}]\in H_{\bullet}(M^{\sigma}_{Q,\ov{d}})
$$
can serve to define integration over $M^{\sigma}_{Q,\ov{d}}$ of cohomology classes. Denoting by $\mU_v$ the universal vector space at the vertex $v$ of $Q$, perhaps the most natural option is to integrate polynomials in the descendent classes $$\tau_i(v)=\ch_i(\mU_v)\,.$$
The descendent algebra $\TT^Q_{\ov{d}}$ generated freely by $\tau_i(v)$ for $i>0$ is a representation of the positive half-branch of the Virasoro algebra. More explicitly, there is a set of differential operators 
$$
\bL_k = \bR_k + \bT_k \quad \text{for}\quad k\geq -1
$$
acting on $\TT^Q_{\ov{d}}$ that satisfy the commutator relations 
$$
[\bL_n,\bL_m] = (m-n)\bL_{n+m} \,.
$$
The reason that no additional term of the form $(n-n^3)/12 \delta_{n+m,0}C$ as in Definition \ref{Def: conformal vertex algebra} appears, is because it would not contribute anyway as long as I work with $\bL_k, k\geq -1$. One would need to additionally determine the negative half-branch which is what I do in Theorem \ref{thm:twovirasoro}.

As can be seen from Definition \ref{def:virop}, the term $\bR_k$ is identical to the one from \cite[§2.3]{BML}. The connection between $\bT_k$ and the Chern character of the virtual tangent bundle $T^{\vir}$ was noticed in \cite[Example 2.6]{BML}. It was stated there and in \cite[Remark 4.13]{BML} that the similarity is a general phenomenon that can be used to determine $\bT_k$.  After computing $\ch(T^{\vir})$ in \eqref{eq:Tvir} this leads to the definition of $\bT_k$ in the second part of Definition \ref{def:virop}.

The formulation of Virasoro constraints in terms of special universal representations $\ov{\UU}$ called $v$-normalized for a vertex $v$ is still applicable, and it is stated in Definition \ref{def:Skvirasoro}. As before, one introduces an additional set of operators $\bS^{v}_k$ and requires that 
\begin{equation}
\label{eq:universalVir}
\int_{[M^{\sigma}_{Q,\ov{d}}]}\xi_{\ov{\UU}} \big( (\bL_k + \bS^v_k)(\tau)\big)  = 0 \quad \textnormal{for each}\quad k\geq 0, \tau\in \TT^Q_{\ov{d}} \,.
\end{equation}
Here  $\xi_{\ov{\UU}}$ denotes the realization of the a priori formal variables $\tau_i(v)$ as Chern characters of the universal vector space $\mU_v$ at the vertex $v$.

Alternatively, one can state Virasoro constraints in terms of the single weight zero operator $\bL_{\inv}$ from Definition \ref{def:weight0} as the vanishing
 $$
\int_{[M^{\sigma}_{Q,\ov{d}}]}\xi\big(\Linv(\tau)\big)=0\quad \textnormal{for any} \quad  \tau\in \TT^Q_{\ov{d}}
  $$
  which is independent of the choice of a universal representation $\ov{\UU}$ so it is omitted from the notation.
  
If one instead works with the framed representations of Definition \ref{def:framedrep1} and their moduli spaces $M^{\ov{n}}_{Q,\ov{d}}$ for fixed $\ov{n}$, $\ov{d}$, then the Virasoro constraints need to be modified. The new operators are
$$
\bL^{\ov{n}}_k = \bR_k+ \bT^{\ov{n}}_k
$$
where $\bR_k$ remains the same, but the $\bT$-operator is altered by subtracting $k!\tau_k(\ov{n})$. For the framed Virasoro constraints, one replaces the operators $\bL_k$ appearing in \eqref{eq:universalVir} by the new ones. This is similar to the modification in \cite[§2.6]{BML} that gave pair Virasoro constraints. This time, however, it is less ad-hoc as explained in the next subsection. There I explain how to extract $\bT^{\ov{n}}_k$ from $\bT_k$, and I show that the resulting constraints are subsumed by the same general theory.
  \subsection{Unifying pair and sheaf Virasoro constraints}
  \label{sec:pairissheaf}
  Here I describe an important observation that however requires some knowledge from \cite{BML} to be fully appreciated. Without prior familiarity with the topic, I recommended skipping it for now, and just keeping in mind the slogan ``pair Virasoro constraints $\subset$ sheaf Virasoro constraints".

To obtain the correct operator $\bT_k$ when there are no frozen vertices, one starts from the virtual tangent bundle of the stack of representations $\mathfrak{M}_{Q,\ov{d}}$ from \eqref{eq:MAGL} rather than its rigidification or $M^{\sigma}_{Q,\ov{d}}$. What changes when frozen vertices are included is discussed in detail below.
  A notable difference when compared to \cite{BML} is the lack of a special type of Virasoro constraints when frozen vertices are present. In \cite[§2.6]{BML}, they were called pair Virasoro constraints, and their absence is discussed in detail in Remark \ref{rem:nopairVir}, Remark \ref{rem:frozenisrigid}, Remark \ref{rem:rigidvsnonrigid}, and Remark \ref{rem:Lierigid}. Stated briefly, it comes down to the equivalences
  \begin{equation}
  \label{eq:frozenrigid}
  \begin{tikzcd}[column sep = huge]
      \boxed{\begin{array}{c}\geq 1 \ \textnormal{frozen}\\
\textnormal{vertices}\end{array}}\arrow[r, Leftrightarrow,"\textnormal{Def. \ref{def:onefrozen}}"]& \boxed{\begin{array}{c}1 \ \textnormal{frozen}\\
    \textnormal{vertex}\end{array}}\arrow[r, Leftrightarrow, "\textnormal{Rem. \ref{rem:nopairVir}}"]& \boxed{\begin{array}{c}
    \textnormal{rigidifying}\\
    \textnormal{the stack}\\
    \textnormal{of representations}\\
    \textnormal{of} \ Q
     \end{array}}
  \end{tikzcd}\,.
  \end{equation}
The diagram states that adding frozen vertices rigidifies the moduli stack of quiver representations, and therefore this leads to adding a copy of $\mO$ to $T^{\vir}$. Through the correspondence with Virasoro constraints, this extra copy was hidden in \cite[Conjecture 2.18]{BML}. It was already pointed out that one should start from the unrigidified stack instead as this is compatible with Joyce's vertex algebra from \cite{Jo17} that offers an intuitive explanation for the presence of Virasoro constraints. This modification is explained in Remark \ref{rem:nopairVir}, and it corrects the pair Virasoro constraints from \cite[§2.6]{BML} allowing me to unify the two cases.

Next, I will use this point of view to derive framed Virasoro constraints directly. Starting from a quiver $Q$, I will construct a new quiver $\ovninfty Q$ like in \eqref{eq:framedflag} with an extra vertex $\stackrel{\infty}{\diamond}$ and the number of edges going out of it to all the original vertices of $Q$ determined by $\ov{n}$.  It plays multiple roles throughout this article, and in each one of them, only the dimension vectors of the form $(d_{\infty},\ov{d}), d_{\infty} = 0,1$ appear. By \cite[Proposition 3.3]{reineke}, its $\mu^{\fr}$-stable representations from Definition \ref{def:framedrep2} are precisely the framed representations with the framing vector $\ov{n}$. Using this perspective, I explain in Lemma \ref{lem:framedVir} how framed Virasoro constraints can be derived from the universal ones in \eqref{eq:universalVir}. Using the above correction for frozen vertices, we see explicitly the cancellation of $\bS^\infty_k$ by $\delta_k$ that undoes the rigidification. This needed to be done superficially in \cite{BML}.
\subsection{The vertex algebra approach and general results}
\label{sec:how}
The conclusion that a virtual fundamental class satisfies Virasoro constraints if and only if it is a physical state in the vertex algebra of Joyce is the core new idea of \cite{BML}. This together with Joyce's wall-crossing \cite{Jo21} which preserves physical states is what led to the proofs of multiple new cases of Virasoro constraints. I follow a similar path of arguments here with some deviations when it comes to constructing explicit conformal elements in §\ref{sec:VAframing}. Not separating between sheaf and pair Virasoro also uniformizes multiple arguments in §\ref{sec:WC}. I now discuss the main aspects of the proof and its conclusions but leave out some of the steps that were already presented in length in \cite[§1.4 - §1.6]{BML}. In the process, the quiver $Q$ is assumed to have no frozen vertices as they are added by hand later.

Let $\MQ$ be the (higher) stack of all perfect complexes of representations of $Q$, then by \cite{Jo17} its homology 
$$
V_{Q,\bullet}  = \widehat{H}_{\bullet}(\mM_Q)
$$
is a vertex algebra determined by the data in Definition \ref{def: VOAconstruction} (the reader should ignore the extra dashes appearing in this definition for now). Using the Euler form $\chi: K^0(Q)\times K^0(Q)\to \ZZ$ from Definition \ref{def:eulerform}, the additional $\widehat{(-)}$ denotes a shift of the homological degree by $2\chi(\ov{d},\ov{d})$ on each connected component $\mM_{\ov{d}}$ of objects with class $\ov{d}\in K^0(Q)$.  I will often neglect the subscript $Q$ when the quiver is understood. By the results in \cite{Jo17} recalled in Theorem \ref{thm:geometricvaconstruction}, one knows that $V_{\bullet}$ is the lattice vertex algebra (see Example \ref{ex:latticeVA}) for the lattice $(\Lambda,\chi_{\sym})$ where I used $\Lambda=K^0(Q)$, and $\chi_{\sym}$ is the symmetrization of $\chi$. When Assumption \ref{ass:homotopyref} holds, this follows after establishing in \eqref{eq:xidagger} the isomorphism
$$
\xi^{\dagger}:\widehat{H}_{\bullet}(\mM_{\ov{d}}) \stackrel{\sim}{\longrightarrow} e^{\ov{d}}\otimes \TT_{\Lambda}
$$
where $\TT^Q_{\ov{d}}$ is the dual algebra of $\TT_{\Lambda}$ in the sense of §\ref{sec:gradedalg}. This assumption can be removed while forgetting $\mR$ as explained in Remark \ref{rem:forgetR}.
The translation operator $T: V_{\bullet}\to V_{\bullet + 2}$ can then be interpreted as the dual of $\bR_{-1}$, and one defines the quotient 
 $\widecheck{V}_{\bullet} = V_{\bullet + 2}/T(V_{\bullet})$.

Recall that one of the ingredients in constructing the vertex algebra structure is the $\Ext$-complex from \eqref{eq:extcomplex} on $\mM_Q\times \MQ$ that restricts along the diagonal to $T^{\vir}[-1]$. In this subsection, I will avoid mentioning the issues related to §\ref{sec:pairissheaf} and the need to distinguish between rigidified and non-rigidified vertex algebras. This discussion will be postponed to §\ref{sec:rigidvsnonrigid} which forces me to assume for now that the symmetrized Euler pairing $\ovninfty \chi_{\sym}$ on the lattice $\ovninfty \Lambda = K^0(\ovninfty Q)$ of the quiver $\Qfr$ is non-degenerate. Even in simple examples, this assumption can be violated, so I use it purely to simplify the exposition. The vertex algebra for $\Qfr$ will be denoted by $\ovninfty V_{\bullet}$.

In §\ref{sec:conformalel}, I recall the definition of a conformal element $\omega$ in a vertex algebra and that for a non-degenerate lattice, its lattice vertex algebra contains a natural choice of $\omega$. This then leads to the conformal element $\ovninfty \omega\in\ovninfty V_{\bullet}$ which gives rise to the above Virasoro constraints for (virtual) fundamental classes $[M^{\sigma}_{Q,\ov{d}}]$. We explained the sheaf version of this already in \cite[§1.4, §1.5]{BML} so I recommend consulting this reference for more details, while I give a brief summary here.

There is an open embedding $M^{\sigma}_{Q,\ov{d}}\hookrightarrow \mM^{\rig}_{\ov{d}}$ where $(-)^{\rig}$ denotes the rigidification of a stack. Pushing forward the (virtual) fundamental class defines
\begin{equation}
\label{eq:virpush}
\ov{[M^{\sigma}_{Q,\ov{d}}]}\in \widecheck{H}_{0}(\mM^{\rig}_{\ov{d}}) \subset \widecheck{V}_0
\end{equation}
when using the notation \eqref{eq:VcheckHcheck}. The quotient $\widecheck{V}_0$ carries naturally the structure of a Lie algebra and the wall-crossing of the classes $\ov{[M^{\sigma}_{Q,\ov{d}}]}$ relating them for different values of $\sigma$ is expressed in terms of it (see Theorem \ref{thm:WCformulae}). The partial lift $\widehat{\ad}$ of the Lie bracket $[-,-]$ that is heavily used in \cite{BML} and recalled in \eqref{eq:hatad} is an interesting observation of its own, but after removing the separation between pair and sheaf Virasoro constraints, there is no longer any need for it when proving Virasoro constraints. It does, however, motivate the definition of $\bL_{\inv}$. This is why I will not make use of it in the introduction, but it appears in the main body of the paper. 

 As recalled in Definition \ref{Def: conformal vertex algebra}, one can extract from a conformal vector $\omega$ a representation of a Virasoro algebra on the vertex algebra. For any quiver $Q$ with non-degenerate pairing $\chi_{\sym}$, the conformal charge of the Virasoro algebra is given by $|V|$.  In the case of $\ovninfty Q$, I will label the corresponding Virasoro operators by $\ovninfty L_k$. Denote by $P_i$ the subspaces of physical states in $\ovninfty V_{\bullet}$ of conformal weight $i\in \ZZ$. Their elements $a$ satisfy
$$
\ovninfty L_k(a)= i\delta_k a \quad \text{for}\quad k\geq 0\,.
$$
Suppose that $a$ also satisfies either of the following two cases:
\begin{enumerate}
    \item $d_{\infty}=0$ and $$a\in e^{\ov{d}}\otimes \TT_{\Lambda}\subset e^{(0,\ov{d})}\otimes \TT_{\ovninfty \Lambda}\,,$$
    \item $\tau_1(\infty)\cap a = 0$ for $a\in e^{(1,\ov{d})}\otimes \TT_{\ovninfty \Lambda}$.
\end{enumerate}
  Then adding the symbol $(-)^{\dagger}$ to denote the dual with respect to the pairing between the dual algebras
\begin{equation}
\label{eq:onvinfpairing}
\langle -,-\rangle: \TT^{\Qfr}_{(d_{\infty},\ov{d})}\times e^{(d_{\infty},\ov{d})}\otimes \TT_{\ovninfty \Lambda}\longrightarrow \CC\,,\end{equation}
  it follows from Theorem \ref{thm:twovirasoro} and Lemma \ref{lem:dashnodash} that $$\ovninfty L_k(a) = (\ovninfty \bL_k)^{\dagger}(a)\quad \text{for}\quad k\geq -1\,.$$
 This is used to conclude for $a\in \ovninfty V_{\bullet}$ satisfying either one of the above two assumptions that 
$$
a\in P_1 \quad \iff \quad \int_{a}\bL_{\inv}(\tau)=0\quad \text{for} \quad  \tau\in \TT^{\Qfr}_{(d_{\infty},\ov{d})}
$$
where the integral is a more suggestive notation for the pairing \eqref{eq:onvinfpairing}. 

I will denote the projection of $a$ to the quotient $\widecheck{P}_0 = P_1/T(P_0)$ by $\ov{a}$. Most importantly, $\widecheck{P}_0$ is a strict Lie subalgebra of $\widecheck{V}_0$ so Virasoro constraints restrict the domain where (virtual) fundamental classes can wall-cross. The class $\ov{a}$ is in general replaced by 
$$
\ov{[M^{\sigma}_{(d_{\infty},\ov{d})}]}\in \ovninfty \widecheck{V}_{0} \quad \text{for} \quad (d_{\infty},\ov{d})\in K^0(\Qfr), \ d_{\infty}=0,1
$$
which are the invariant classes defined in \cite[Theorem 5.8]{GJT} and \cite[§6.4, §9.1]{Jo21} generalizing the (virtual) fundamental classes. When $d_{\infty}=0$, one introduces the moduli space $M^{\sigma_P}_{(1,\ov{d})}$ where $\sigma_P$ is a variant of Joyce--Song stability  constructed out of $\sigma$ for $\ovninfty Q$. There are no strictly $\sigma_P$-semistable representations, so one may consider $\Big[M^{\sigma_P}_{(1,\ov{d})}\Big]$ and push it forward along the natural map to the stack $\Pi: M^{\sigma_P}_{(1,\ov{d})}\to \mM^\sigma_{\ov{d}} \subset \mM_Q$ after caping with $c_{\rk}(T_{\Pi})$\footnote{I used the standard notation where $c_{\rk}(-)$ denotes the top Chern class and $T_{\Pi}$ is the relative tanget bundle.}. One then corrects the resulting homology class by subtracting contributions of semistable strict subrepresentations of objects in $\mM^\sigma_{\ov{d}}$.

In Proposition \ref{prop:bridgeland}, I extend Joyce's wall-crossing theorem recalled in Theorem \ref{thm:WCformulae} to Bridgeland stability. It relates the above classes for different $\sigma$ in terms of the Lie bracket on $\ovninfty\widecheck{V}_0$ and allows me to prove the main general theorem.
\begin{theorem}[Proposition \ref{cor:PMimpliesadPM}, Proposition \ref{prop:bridgeland}]
\label{thm:quivvirWC}
    Virasoro constraints are preserved under wall-crossing. In other words, if $\ov{[M^{\sigma_1}_{(c_{\infty},\ov{c})}]}$ satisfy Virasoro constraints on the right-hand side of \eqref{eq:quiverWC} or \eqref{eq:pairWC}, then so do $\ov{[M^{\sigma_2}_{(d_{\infty},\ov{d})}]}$ on the left whenever $d_{\infty}=0,1$. This holds for any pair of Bridgeland stability conditions $\sigma_1,\sigma_2$ constructed on the heart $\rep(Q)$.
\end{theorem}

By selecting $\sigma_1$ to be the increasing $\mu$-stability condition of \cite[Definition 5.5, Proposition 5.6]{GJT}, I reduce the proof of Virasoro constraints to point-like moduli spaces where they trivially hold. This approach is reminiscent of the proof of \cite[Theorem A (1)]{BML}, where we deduced Virasoro constraints for vector bundles on curves by wall-crossing into 0-dimensional sheaves. Thus, I proved
\begin{corollary}[Theorem \ref{thm:main}, Proposition \ref{prop:bridgeland}]
    The classes $\ov{[M^{\sigma}_{(d_{\infty},\ov{d})}]}^{\inva}$ satisfy Virasoro constraints whenever $d_{\infty} = 0,1$ and $\sigma$ is a Bridgeland stability condition on $\Rep(Q)$.
\end{corollary}
Parts (I) and (II) of Theorem \ref{thm:first} are immediate consequences of the corollary, while part (III) requires some extra work that I elucidate in the §\ref{sec:introsheaves}. Interestingly, once Virasoro constraints for the classes $[M^{\sigma}_{Q,\ov{d}}]$ are recast in terms of the vertex algebra $V_{Q,\bullet}$ as above, one can state them in an almost identical form as one does for GW Virasoro constraints. I present this in Remark \ref{rem:relatingtoGW}.
\subsection{Vertex algebras with non-degenerate pairings}
\label{sec:rigidvsnonrigid}
Here I will explain how to remedy the situation when $\ovninfty\chi_{\sym}$ is singular. As I apply the perspective outlined in \ref{sec:pairissheaf}, it is advisable to get acquainted with \cite{BML} or other parts of the present work before delving deeper into this subsection.

I will focus on quivers constructed out of $Q$ by adding frozen vertices, so it will be important to distinguish between their rigidified and the non-rigidified $T^\vir$ and the corresponding $\Ext$ complexes. This is done in Definition \ref{def: VOAconstruction} by adding the extra $(-)'$ when in the rigidified situation. One particular example that plays a role in §\ref{sec:VAframing} is the quiver $Q^{\fr}$ which adds a frozen vertex $(v)$ and additional edge for each vertex $v$ of $Q$ 
$$
\begin{tikzcd}
    \overset{(v)}{\diamond}\arrow[r]&\overset{v}{\bullet}
\end{tikzcd}\,.
$$
The symmetrization $\chi^{\fr}_{\sym}$ of its Euler form is non-degenerate as shown in Lemma \ref{lem:Vfrnondegenerate}. Moreover, there exist maps transforming representations of one quiver to another along the following squiggly arrows
$$
\begin{tikzcd}
Q\arrow[r,rightsquigarrow]&\Qfr\arrow[r,rightsquigarrow]&Q^{\fr}
\end{tikzcd}\,.
$$
This can be translated to the homology of the associated stacks to induce inclusions of vertex algebras. A priori, there are two ways of doing this: 1) following the approach with pair vertex algebras from \cite[Definition 4.4]{BML}, 2) adding the rigidification correction from §\ref{sec:pairissheaf}. I represent both in the following diagram where I use $\ovninfty V_{\bullet}$ to denote the vertex algebra associated to $\Qfr$, $\ovninfty V'_{\bullet}$ its rigidified version, and $V^{\fr}_{\bullet}$ the rigidified vertex algebra of $Q^{\fr}$.
$$
\begin{tikzcd}[column sep = 0ex]
    &\arrow[dl, hook']V_{Q,\bullet}\arrow[dr, hookrightarrow]&\\
    \ovninfty V_{\bullet}\arrow[d, hook']&&\ovninfty V'_{\bullet}\arrow[d, hookrightarrow]\\
  \qquad\quad V^{\nd}_{\bullet}\ni \omega^{\nd}  &&\omega^{\fr}\in V^{\fr}_{\bullet}\qquad
\end{tikzcd}
$$
The vertex algebra $V^{\nd}_{\bullet}$ is an abstract lattice vertex algebra associated with a non-degenerate lattice $(\Lambda^{\nd}, \chi^{\nd}_{\sym})$ containing $(\ovninfty \Lambda,\ovninfty \chi_{\sym})$. We thus have two different conformal elements $\omega^{\nd}$ and $\omega^{\fr}$.

The purpose of keeping both points of view is to establish a unified approach to ``sheaf" and ``pair" Virasoro constraints through $\omega^{\nd}$ while preserving the existence of an explicit geometric conformal element $\omega^{\fr}$ (see Remark \ref{rem:Lierigid} for a more thorough account of this issue). The lie brackets on $\ovninfty \widecheck{V}_{\bullet}$ of the form used in Theorem \ref{thm:WCformulae} are not affected by including the dash $(-)'$ so both $(-)^{\nd}$ and $(-)^{\fr}$ perspectives can be used.

Both $\omega^{\nd}$ and $\omega^{\fr}$ lead to representations of Virasoro algebras
\begin{align*}
L^{\nd}_k:& V^{\nd}_{\bullet}\to V^{\nd}_{\bullet-k}\,,\qquad 
L^{\fr}_k: V^{\fr}_{\bullet}\to V^{\fr}_{\bullet-k}\,,
\end{align*}
respectively. The restriction of $L^{\nd}_k$ to $\ovninfty V_{\bullet}$ determines a unique set of operators $\ovninfty L_k$ that were previously obtained assuming non-degeneracy of $\ovninfty \chi_{\sym}$. The restriction of $L^{\fr}_k$ is denoted by $\ovninfty L'_k$ to keep in mind that they came from a rigidified vertex algebra. When working with $\ovninfty L'_k$, I mark the outcome of all previous constructions by an extra dash $(-)'$.

 I can now explain how the interpretation described in \cite[§4.4]{BML} can be recovered from the present unified approach if one instead works with $V^{\fr}_{\bullet}$ and $\ovninfty L'_k$. Lemma \ref{lem:dashnodash} states that 
$$
a\in P_1\quad \iff \quad a\in P'_{1-d_{\infty}}
$$
whenever $a$ satisfies one of the assumptions in §\ref{sec:how}. This degree shift is the reason why we needed to distinguish between pair and sheaf Virasoro constraints previously, and it forced us to use the partial lift $\widehat{\text{ad}}$ in \cite[(56)]{BML}.
\subsection{Semistable sheaves on surfaces}
\label{sec:introsheaves}
Virasoro constraints for moduli of sheaves were formulated in \cite[Definition 2.11, Conjecture 2.12, Definition 2.13, Conjecture 2.15]{BML}. They are very similar in nature to the Virasoro constraints for quivers because both of them are determined by the vertex algebra that interacts with $T^{\vir}$. In particular, the two can be identified under potential derived equivalences.  This point is explained more rigorously in Remark \ref{rem:derequiv}. It is then used in §\ref{sec:P2} to prove Virasoro constraints for semistable sheaves on $S =\PP^2,\PP^1\times \PP^1$ (and $S=\Bl_{\pt}\PP^2$). More specifically, letting $r$ be the rank, $d$ the first Chern class, and $n$ the second Chern character of a sheaf on $S$, I study the projective moduli spaces $M = M^{p}_{r,d,n}$ of Gieseker stable sheaves $F$ with $\ch(F) = (r,d,n)$ or more generally the enumerative invariants $[M^{p}_{r,d,n}]^\inva$ that are defined by Joyce in \cite[§7.7, §9.1]{Jo21} in a similar way as $[M^\sigma_{\ov{d}}]$ were -- I reviewed this before the statement of Theorem \ref{thm:quivvirWC}. 

Recall that if $\mF$ is the universal sheaf on $S \times M$, then the virtual tangent bundle on $M$ is given by 
$$
T^{\vir} = \tau^{\geq 1}\big(\text{R}\uHom_{M}(\mF,\mF)\big)[1]\,.
$$
It leads to a virtual fundamental class $[M]^{\vir}$ that can be pushed forward to $\widecheck{V}_{S,\bullet}$ from \cite[Lemma 4.10]{BML} which replaces $\widecheck{V}_{Q,\bullet}$. When $(r,d,n) = (1,0,-n)$, one recovers the Hilbert scheme $M=\Hilb^n(S)$ and its fundamental class. Even in the two simple cases $S=\PP^2$ and $S=\PP^1\times \PP^1$, the only known proof of Virasoro constraints for $[\Hilb^n(S)]$ was given in \cite[Theorem 20]{moop}, and it is based on a long sequence of arguments going back to Kontsevich as explained at the beginning of the introduction. Because of how natural sheaf-theoretic Virasoro constraints are, one expected that there had to be an independent way of proving them for surfaces. I achieve this goal here by addressing these two basic surfaces which are simultaneously the main stepping stones towards a general statement for any $S$ due to the universality of Virasoro constraints proved in \cite{Jo23} extending the results of \cite[§3]{moreira}.

The main idea is to use the strong full exceptional collections $\mE(k)$ of $\PP^2$ and $\mE(i,j)$ of $\PP^1\times\PP^1$ appearing in Example \ref{ex:excP2} and inducing derived equivalences
$$
 \Gamma(k): D^b(\PP^2)\stackrel{\sim}{\longrightarrow}D^b(P_2)\,,\qquad  \Gamma(i,j): D^b(\PP^1\times\PP^1)\stackrel{\sim}{\longrightarrow}D^b(P_1\times P_1)
$$ for the quivers $P^2,P_1\times P_1$ illustrated there. Now that I proved Virasoro constraints for Bridgeland semistable objects in the hearts $\rep(Q)$, I only need to understand how $\mM^p_{r,d,n}$ translate under the derived equivalences. 

This question was originally studied in \cite[§5--§7]{ABCH} in the case of Mumford stability for torsion-free sheaves without strictly semistables, but because it was based on the large volume limit of stability conditions introduced by Bridgeland in \cite{bridgelandK3}, it applies directly to Gieseker semistable sheaves. The argument for $\PP^1\times\PP^1$ and $\Bl_{\pt}(\PP^2)$ is the same and was carried out in \cite{EMiles}. I briefly summarize the argument:
\begin{enumerate}
    \item By \cite[§6]{Macrilecture} there is a Bridgeland stability condition $\sigma_{s_0,t_0}$ for $t_0\gg 0$ that is equivalent to twisted Gieseker stability when $r>0$. This argument works also for $r=0$ thanks to \cite[Lemma 7.3]{Wolf}.
    \item Using Bertram's nested wall theorem \cite{ABCH, Macioca13}, one can vary $(s,t)$ along a unique wall $W$ containing $(s_0,t_0)$ such that the stacks $\mM^{s,t}_{r,d,n}$ of $\sigma_{s,t}$-semistable objects remain unchanged.
    \item For sufficiently small $t$, the stability condition $\sigma_{s,t}$ can be constructed on hearts generated by full strong exceptional collections of $\PP^2$ by \cite[§7]{ABCH}. For $\PP^1\times \PP^1$ (and $\Bl_{\pt}\PP^2)$, this statement appears in \cite[§4.2, §5.3 and §5.4]{EMiles}.
    \item Using Proposition \ref{prop:bridgeland}, I show that 
    $$
    [M^p_{r,d,n}]^{\inva} = [M^{\sigma}_{\ov{d}}]\in \widecheck{P}_0
    $$
for some appropriate dimension vector $\ov{d}$ and Bridgeland stability $\sigma$ on the heart of quiver representations. 

Including strictly semistables proves to be rather challenging, because I need to do more than just compare moduli stacks of semistable objects for a fixed topological data. In Lemma \ref{lem:filing}, I am  in fact required to identify the stacks for all potential subobject. This is necessary to prove that the procedure of \cite[§9.1]{Jo21} for defining invariants is independet of whether one is working with quivers of with surfaces. 
\end{enumerate}
 This is how Theorem \ref{thm:P2indep} is proved after recalling that being an element of $\widecheck{P}_0$ is equivalent to Virasoro constraints being satisfied. In particular, this gives a self-contained new proof of \cite[Conjecture 1.4]{bree} in the case of $\PP^2$, $\PP^1\times\PP^1$, and $\Bl_{\pt}\PP^2$ covering also strictly semistable sheaves just as in \cite{BML}. The result for 1-dimensional semistable sheaves on the above three surfaces is strictly speaking new because in \cite[Theorem A (3)]{BML} the claim for any $S$ depends on an additional assumption. This assumption will be addressed in my future work. Following the appearance of an older version of this work, the authors of \cite{MW} used a simpler form of my approach to recover the above results in the stable case only. 
 
\subsection{Future directions}
\label{sec:future}
There are some interesting questions that I did not have time to address here, and that are worth investigating. I order them by increasing difficulty starting with what I expect to be the most straightforward one. 

\begin{enumerate}[wide=0pt]
    \item Recently, Bu studied in \cite{Bu23} wall-crossing for homological invariants in self-dual categories. If $Q$ is a self-dual quiver in the sense of \cite[Definition 7.1]{Bu23}, then one can construct a category of its self-dual representations. His work then goes on to construct fundamental classes of stable self-dual representations for self-dual stability conditions. They are then extended to invariant classes analogous to the ones in \cite{GJT} that satisfy wall-crossing formulae \cite[Theorem 7.9]{Bu23}. One can most likely adapt the general approach to Virasoro constraints through wall-crossing via vertex algebras to the self-dual setting. The only technical hurdle lies in checking that the operation $\heartsuit$ appearing in \cite[Theorem 4.17, Theorem 5.18]{Bu23} preserves Virasoro constraints just like the Lie bracket did. This seems plausible because a weaker version of locality \cite[(4.29)]{Bu23} still holds for Bu's twisted vertex algebra modules. Tracing through all the arguments on the vertex algebra side, one can pinpoint locality as the vital condition holding everything together. 
    \item Because both toric and del Pezzo surfaces admit full exceptional collections, a similar method to the one used for $\PP^2$, $\PP^1\times \PP^1$, and $\Bl_{\pt}\PP^2$ could potentially be applied to give another proof of Virasoro constraints for Gieseker semistable sheaves in this generality. As can be seen from the first two pages of \cite{rekuski}, there is an extensive body of literature that studies stability manifolds of quivers and surfaces. In fact, many of them focus on geometric and algebraic stability conditions whose origin is usually described by their name. 
    \item At the beginning of §\ref{sec:whatare}, I made the assumption that $Q$ has no cycles which I partially get rid of later on. I did so for two reasons:
    \begin{enumerate}[label = (\roman*)]
        \item I wanted to work with projective moduli spaces of semistable objects which may be violated once cycles are present, but sometimes relations can be used to correct this. It boils down to whether the moduli scheme of semisimple representations for a given $\ov{d}$ is projective because then one can still apply the proof of \cite[Proposition 4.3]{King}. This last condition is always satisfied if the resulting path algebra is finite-dimensional which allows me to state Theorem \ref{thm:first} (I) in its full generality. 
        
        \item Increasing stability conditions used in the proof of Theorem \ref{thm:main} no longer make sense. Still, one can find in cases where cycles are cancelled out by relations example-specific $\mu$-stabilities that again give rise to point-like moduli schemes. One can incorporate such examples into the statement of Theorem \ref{thm:main}, but I did not have an immediate application in mind.
    \end{enumerate}
    If there are cycles that force the moduli spaces to be non-compact, then one can still ask for projective fixed point loci with respect to the torus rescaling the cycles. One thus lands in the realm of equivariant Virasoro constraints. Just like in the case of the equivariant Segre--Verlinde correspondence studied in \cite{BH}, the cohomological degree is no longer constrained to $0$ but lies in the interval $[-1.\infty)$ once an additional admissibility condition is checked. Therefore, one will need a whole family of Virasoro constraints for all non-zero degrees similar to the higher degree Segre--Verlinde correspondence observed in \cite[Theorem 1.11]{BH}.
\end{enumerate}
\section*{Acknowledgements}
  I would like to thank R. Pandharipande for encouraging the completion of this project. This work originated from the collaboration with W. Lim and M. Moreira. Finally, I greatly appreciate the discussions I had on this topic with P. Bielmans, P. Bousseau, T. Bridgeland, A. Jacovskis, D. Joyce, H. Liu, E. Macrì and F. Rota.

  I was supported by ERC-2017-AdG-786580-MACI. This project has
received funding from the European Research Council (ERC) under the European
Union Horizon 2020 research and innovation program (grant agreement No 786580). During the final phase of writing this work, I stayed at Academia Sinica.

  \section*{Notation and Conventions}
\begin{center}
\begin{tabular}{p{3.5cm} p{10.5cm}}
$\delta_n$& Shorter notation for the Kronecker symbol $\delta_{n,0}$.
\\
$\NN_0$ & The natural numbers.\\
$V$& The set of vertices of a quiver $Q$.\\
$V_{\bullet}$& The underlying vector space of a vertex algebra.\\
    $v,w$& Elements of $V$ sometimes viewed as generators of $K^0(Q)$.\\
$a,b,c$& Elements of $V_\bullet$.
    \\
    $\Lambda,\Lambda_{\CC}$& The lattice $\Lambda=K^0(Q)=\ZZ^{V}$ and the associated complex vector space $\Lambda_{\CC} = \Lambda\otimes_{\ZZ}\CC$.\\
    $\Lambda_+$&The subset $(\NN_0)^V\backslash \{0\}$ of $\Lambda$.\\
    $u,x,y$& Elements of $\Lambda_{\CC}$.\\
    $\deg(-)$ & Degree for any graded vector space. 
    \\
     $L_n, T_n, R_n,$& Virasoro operators acting on homology and vertex algebras.\\
    $\mathsf{L}_n, \mathsf{T}_n,\mathsf{R}_n,$& Dual operator notation on cohomology and the descendent algebra.\\
    $\mathsf{L}_{\inv}$& The weight zero Virasoro operator on the descendent algebra.\\
    $\mathfrak{M}_Q$& The Artin stack of representations of a quiver $Q$.\\
    $\MQ$& The higher stack of complexes of representations of a quiver $Q$.
        \end{tabular}
\end{center}
I will always work over the field $\CC$ so I will most often leave it out of the notation. For example, I will just write vector space instead of a $\CC$-vector space. Similarly, tensor products over $\CC$ will be denoted just by $(-)\otimes (-)$.
\section{Definition of Virasoro constraints for quivers with relations}
\label{sec: virasoro constraints}
\subsection{Graded algebras}
\label{sec:gradedalg}
Leading up to the definition of Virasoro constraints in this setting, I recall some basics about commutative graded algebras based on \cite[§2.1]{BML}. Compared to loc. cit., the degree of a commutative graded unital algebra $D_\bullet$ takes values only in $2\ZZ$. We, therefore, change the previous convention by dividing $\deg(v)\in 2\ZZ$ by 2 and making the present algebras commutative $\ZZ$-graded. Henceforth, I only use graded to mean $\ZZ$-graded and all algebras will be unital. 

The convention will always be that if $C_{\bullet}$ is a graded vector space, then the associated commutative graded algebra is the symmetric algebra $D_{\bullet} = \Sym[C_\bullet]$. If instead, one starts from $C^\bullet$ which denotes the graded dual of $C_\bullet$, then the associated dual graded algebra $D^\bullet$ is the completion of $\Sym[C^\bullet]$ with respect to the degree. Explicitly, this means 
\[D^\bullet\coloneqq \Sym\llbracket C^\bullet\rrbracket=\prod_{i\geq 0}\Sym[C^\bullet]^i\]
where $\Sym[C^\bullet]^i$ denotes the degree $i$ part of $\Sym[C^\bullet]$ with the degree induced by the one on $C^\bullet$.

For the mutually dual vector spaces $C_{\bullet}$ and $C^\bullet$, I denote the natural pairing by  $\langle-,-\rangle\colon C^\bullet\times C_\bullet\to \CC$. This induces a cap product 
$$
\cap\colon C^\bullet\times D_\bullet\longrightarrow D_\bullet
$$
as a derivation restricting to $\langle-,-\rangle$ on $C^\bullet\times C_{\bullet}$.  To obtain 
$$
\cap\colon D^\bullet\times D_\bullet\longrightarrow D_\bullet\,,
$$
one simply requires $(\mu\nu)\cap u =\mu \cap(\nu\cap u)$ making $D_\bullet$ into a left $D^\bullet$ module. Fixing a basis $B$ of $C_{\bullet}$ allows me to write 
$$
\nu \cap (-) = \sum_{v\in B}\langle \nu,v\rangle\, \frac{\partial}{\partial v}\,.
$$
After composing the cap-product with the algebra homomorphism $D_{\bullet}\to \CC$ induced by $C_\bullet\to \CC$, one is left with the pairing 
$$
\langle-,-\rangle : D^{\bullet}\times D_{\bullet}\longrightarrow \CC\,.
$$
Starting from a $\CC$-linear $f: D^\bullet_1\to D^\bullet_2$, this will allow me to define the dual map $$f^{\dagger}: D_{2,\bullet}\to D_{1,\bullet}\,.$$
\subsection{Background on quivers with relations}
\label{sec:background}
Let $Q=(V,E)$ be a connected quiver with the set of vertices $V$ and edges $E$. For each $e\in E$, I denote by $t(e),h(e)\in V$ the vertex at the tail, respectively the head of $e$. It will be important later on to allow some vertices to be \textit{frozen} which leads to a setting similar to pairs of sheaves and their moduli schemes. The subset of frozen vertices will be denoted by $F\subset V$ and unlike the usual vertices labelled by $\bullet$, I will use $\diamond$ to represent the frozen ones. Additionally, I require that each frozen vertex is a source of the quiver (no arrows are going into it) which will allow us to replace the general scenario of any number of frozen vertices with a quiver with just one vertex (see Definition \ref{def:onefrozen}).

The simplest example to keep in mind is the $A_{l-1}$-quiver
\begin{equation}
\label{Eq:equiv}
\begin{tikzcd}
\overset{l-1}{\bullet}\arrow[r]&\overset{l-2}{\bullet}\arrow[r]&\overset{l-3}{\bullet}\arrow[r]&\cdots\arrow[r]&\overset{1}{\bullet}\,,
\end{tikzcd}
\end{equation}
for some $l>1$. After adding a frozen vertex $\overset{\infty}{\diamond}$ with multiple edges going towards each of the original vertices, I will be able to recover partial flag-varieties as the moduli schemes of representations of the quiver in \eqref{eq:framedflag}. 

\begin{definition}
\label{def:homorela} Let $\mA_Q$ be the path algebra of $Q$, $\mA_{v,w}$ its linear subspace spanned by paths starting at $v$ and ending at $w$, and $\mA^{(k)}_Q$ be the two-sided ideal generated by paths of length $k$. The relations of $Q$ are determined by choosing a two-sided ideal $\mR \subset \mA^{(2)}_Q$. A set of generators $R$ of $\mR$ is said to be \textit{homogeneous} if for each $r\in R$, one can choose $t(r),h(r)\in V$ such that $r\in \mA_{t(r),h(r)}$. Such $r$ are called \textit{homogeneous generators}.

 A \textit{quiver with relations} is given by the data $(Q,\mR)$. If additionally, the vertices determined by the set $F$ are frozen, I indicate it by writing $(Q,F,\mR)$ and assume that $t(r)\notin F$ for any homogeneous $r\in\mR$.
\end{definition}

A representation $(\ov{U},\ov{f})$ of $(Q, F, \mR)$ consists of 
\begin{enumerate}
    \item a vector space $U_v$ for each $v\in V $ together with a fixed isomorphism $U_v\cong \CC^{d_v}$ for some $d_v\geq 0$ if $v\in F$,
    \item a collection of morphisms $f_e: U_{t(e)}\to U_{h(e)}$ for each $e\in E$, which induce a representation of $\mA_{\mQ}/\mR$ on $U=\bigoplus_{v\in V}U_v$ mapping $e$ to $f_e$. This means that the morphisms $f_e$ satisfy all the relations in $\mR$. 
\end{enumerate}

In the following, I assume that there are no relations for $Q$ and include them later in the Remark \ref{rem:relations}. The \textit{dimension vector} $\ov{d} = \ov{\dim}(\ov{U}, \ov{f})$ is defined by $d_v=\dim(U_v)$. Unless specified otherwise, I will always fix $\ov{d}\in (\NN_0)^V\backslash \{0\}=\Lambda_+$ with 
$$
\sum_{v\in F}d_v >0\,.
$$

One can choose $U_v=\CC^{d_v}$ for every $v\in F$, and one can choose such an identification for any $v\in V\backslash F$ up to some isomorphism of representations. One constructs the space
$$
A_{Q,\ov{d}}  = \bigoplus_{e\in E}\Hom\big(U_{t(e)},U_{h(e)}\big)
$$
parametrizing all representations of $Q$ with dimension vector $\ov{d}$. It carries the trivial rank $d_v$ vector bundles $\un{U_v}$ for each $v\in V$ with universal morphisms $$\un{\Ff_e}: \un{U_{t(e)}}\longrightarrow \un{U_{h(e)}}\,.$$ 

The symmetry group acting $A_{Q,\ov{d}}$ is
$$
GL_{\ov{d}} = \prod_{v\in V\backslash F}GL(U_v)\,,
$$
and it acts on $A_{Q,\ov{d}}$ by conjugation of $f_e$ for all $e$. This action lifts to $\un{U_v}$ via the canonical action of $GL(U_v)$ on the identical fibres. This makes $\un{\mathfrak{f}_e}$ into a morphism of $GL_{\ov{d}} \,$-equivariant vector bundles. Therefore,  one obtains the induced morphisms 
$$
\Ff_{e}: \mU_{t(e)}\longrightarrow \mU_{h(e)}
$$
between the universal vector bundles on the stack 
\begin{equation}
\label{eq:MAGL}
\mathfrak{M}_{Q,\ov{d}}=\big[A_{Q,\ov{d}}/GL_{\ov{d}}\big]\,.
\end{equation}
\begin{remark}
\label{rem:relations}
    To include relations, one needs to replace the original definition of  $A_{Q,\ov{d}}$ by the closed $GL_{\ov{d}} \,$-invariant subset $A_{Q,\ov{d}}\subset \bigoplus_{e\in E}\Hom\big(U_{t(e)},U_{h(e)}\big)$ consisting of representations $(\ov{U}, \ov{f})$ satisfying the relations. In general, this subset can be described as the vanishing locus of a $GL_{\ov{d}}\,$-invariant section of a $GL_{\ov{d}}\,$-equivariant vector bundle on the latter affine space. This still gives the description of $\mathfrak{M}_{Q,\ov{d}}$ as the quotient $\big[A_{Q,\ov{d}}/GL_{\ov{d}}\big]$ which is a closed substack of the moduli stack of representations without relations. 
\end{remark}
\subsection{Semistable representations and their moduli schemes
}
\label{sec:semrepmod}
The goal is to study the intersection theory of moduli schemes of semistable quiver representation, and I focus on \textit{slope stability} at first which is defined in terms of a map 
\begin{equation}
\label{eq:mustability}
\mu: K^0(Q)\longrightarrow \RR\,,\qquad \mu(\ov{d}) = \frac{\sum_{v\in V}\mu_v d_v}{\sum_v d_v}
\end{equation}
for some  $\ov{\mu}\in \RR^{V}$. A representation $\ov{U}$ with dimension vector $\ov{d}$ is said to be (semi)stable if for every proper subrepresentation $\ov{W}\subset \ov{U}$ with dimension vector $\ov{c}$ we have
$$
\mu(\ov{c})(\leq)<\mu(\ov{d})\,.
$$
It is important to note here that for $\ov{W}$ to be called a subrepresentation, it must satisfy $c_{v}=d_v$ or $c_v=0$ whenever $v\in F$. 

When there are no frozen vertices, King constructed in \cite[§3 and §4]{King} the coarse moduli schemes $M^{\mu}_{Q,\ov{d}}$ of $\mu$-semistable representations with dimension vector $\ov{d}$, and he proved in \cite[Proposition 4.3]{King} that they are projective if $\mA_Q$ is finite-dimensional. I will from now on only consider $Q$ such that this holds.

I will also need projectivity of the moduli schemes with included frozen vertices. I denote the moduli schemes by $M^{\mu}_{Q,\ov{d}}$ without specifying the set $F$. By Definition \ref{def:onefrozen}, I can always replace any number of frozen vertices by a single one labeled $\infty$ with $d_{\infty}=1$. The $\mu$-stability is preserved under this operation, and the resulting moduli schemes can be constructed without distinguishing whether $\infty$ is frozen or not because freezing a dimension 1 vertex is exactly equivalent to rigidifying the moduli stack $\mathfrak{M}_{Q,\ov{d}}$. Applying \cite{King} when $\mA_Q$ is finite-dimensional, one concludes projectivity also in this setting.

 Note that by Definition \ref{def:framedrep2}, framed representations studied by Nakajima in \cite{Nakajimaquiver}, can also be obtained from working with a single frozen vertex. Consequently, this gives an alternative proof of projectivity of moduli schemes of framed representations which was originally shown in \cite[Theorem 3.5]{Nakajimaquiver}. Concerning relations, one only needs to note that the moduli schemes of semistable representations of $(Q,F,\mR)$ are Zariski closed subsets of $M^{\mu}_{Q,\ov{d}}$ for the quiver without relations so their projectivity is implied by the result for the latter space.

In the case that there are no strictly semistable representations, one may ask whether $M^{\mu}_{Q,\ov{d}}$ is fine and thus admits a universal representation. As I will need it later on, I rephrase the proof of a result from \cite[§5]{King} giving a positive answer to this question under additional restrictions. The current formulation is compatible with the geometric construction of vertex algebras in §\ref{sec:recallVA} based on \cite{Jo17} and introduces some notation that will be used there.
\begin{proposition}[{\cite[Proposition 5.3]{King}}]
\label{prop:unirep}
Let $F=\emptyset$, $\ov{d}$ be such that $d_v$ are relatively prime integers and there are no strictly $\mu$-semistables $(\ov{U},\ov{f})$ with $\ov{\dim}(\ov{U},\ov{f}) = \ov{d}$, then the moduli scheme $M^{\mu}_{Q,\ov{d}}$ admits a universal representation with universal morphisms
$$
\Ff_e: \UU_{t(e)}\longrightarrow  \UU_{h(e)}\quad \textnormal{for each} \quad e\in E
$$
between the universal vector spaces $\{\UU_{v}\}_{v\in V}$. If instead there is at least one frozen vertex with $d_v\neq 0$, then by the connectedness of $Q$ there is a canonical choice of a universal representation $(\ov{\UU},\ov{\Ff})$.
\end{proposition}
\begin{proof}
Let us first consider a quiver $Q$ with no relations and no frozen vertices. The action of $GL_{\ov{d}}$ on $A_{Q,\ov{d}}$ factors through the action of
$$
PGL_{\ov{d}} = GL_{\ov{d}}/\Delta
$$
where $\Delta$ is the one dimensional subtorus consisting of elements $(c\id_{U_v})_{v\in V}\in GL_{\ov{d}}$ for $c\in \GG_m$. The scheme $M^\mu_{Q,\ov{d}}$ is an open subscheme of the \textit{rigidification} $\mathfrak{M}^{\rig}_{Q,\ov{d}} = [A_{Q,\ov{d}}/PGL_{\ov{d}}]$. 

The canonical action of $GL_{\ov{d}}$ on $\un{U_v}$ does not factor through $PGL_{\ov{d}}$.  More precisely, the induced action
\begin{equation}
\rho:[*/\Delta]\times \mathfrak{M}_{Q,\ov{d}}\longrightarrow \mathfrak{M}_{Q,\ov{d}}\,,
\end{equation}
which rescales the automorphisms of each representation has weight $1$ on each $\mU_{v}$. That is
$$
\rho^*\,\mU_{v} = \mV\,\boxtimes \,\mU_{v}
$$
for the universal line bundle $\mV$ on $[*/\Delta]$. The solution is to twist the universal vector space $\mU_v$ by a trivial line bundle with a non-trivial action. In the original phrasing, this corresponds to changing the action on $\un{U_v}$ while preserving the action on $A_{Q,\ov{d}}$. The new universal vector space will have the form 
$$
\tilde{\mU}_v = \prod_{w\in V}\textnormal{det}^{-\lambda_w}(\mU_w) \otimes \mU_v
$$
for some $\lambda_w\in \ZZ$ and its pullback under the $[*/\Delta]$ action becomes 
$$
\rho^*(\tilde{\mU}_v) = \mV^{1-\sum_{w\in V}\lambda_w d_w }\boxtimes \tilde{\mU}_v\,.
$$
In conclusion, one needs to find integers $\lambda_w$ such that $\sum_{w\in V}\lambda_w d_w = 1$ which is equivalent to $d_v$ being relatively prime. Then $\tilde{\mU}_v$ descend to $\mathfrak{M}^{\rig}_{Q,\ov{d}}$ from  $\mathfrak{M}_{Q,\ov{d}}$, and I denote its restriction to $M_{Q,\ov{d}}$ by $\UU_{v}$. The same argument goes through if one allows relations in $Q$ and uses Remark \ref{rem:relations}. Then the construction of $\tilde{\mU}_v$ descending to $\UU_{v}$ is the same.

In the presence of frozen vertices with $d_v\neq 0$, the stabilizer subgroup of $GL_{\ov{d}}$ for each stable representation is trivial. Thus one can restrict $\mU_{v}$ to the moduli scheme to obtain the canonical $\UU_{v}$.
\end{proof}

Independently of whether $M^{\mu}_{Q,\ov{d}}$ admits a universal representation, it always carries a (virtual) fundamental class whenever there are no strictly semistables. The additional need to use virtual classes enters only when introducing relations, but I will not distinguish between the two cases. I will always write 
\begin{enumerate}
    \item $[M^\mu_{Q,\ov{d}}]$ for the (virtual) fundamental class in $H_*(M^\mu_{Q,\ov{d}})$,
    \item the result of a cohomology class $\gamma\in H^*(M^\mu_{Q,\ov{d}})$ acting on $[M^\mu_{Q,\ov{d}}]$ as 
    $$
    \int_{M^\mu_{Q,\ov{d}}}\gamma\,.
    $$
\end{enumerate}
\subsection{Descendents for quivers}
In \cite{BML}, we followed the approach to defining Virasoro operators found in the existing literature. This required us to introduce the \textit{formal algebra of descendents} as an auxiliary structure. The situation is simpler when working with quivers so I will work with the $\ov{d}$-dependent algebra $\TT^Q_{\ov{d}}$ directly:
\begin{definition}\label{def: descendentalgebra}
Let $\textnormal{T}^Q_{\ov{d}}$ denote the infinite-dimensional vector space over $\CC$ with a basis given by the collection of symbols
\[\tau_i(v)\quad\textup{ for }\quad i>0,\, v\in V\]
called \textit{vertex descendents}\footnote{I chose a different notation from \cite{BML} here because I want to distinguish between the geometric variables $\ch_i(\gamma)$ defined for each cohomology class $\gamma\in H^\bullet(X)$ and $\tau_i(v)$ of a more algebraic origin.}.  To keep track of $i$ we introduces the $\ZZ$-grading $\deg\tau_i(v) = i$ for $v\in V$. 

The above definition clearly does not depend on $\ov{d}$, but I keep track of the dimension vector as I will also work with 
$$
\tau_0(v) = d_v
$$
considered as elements of the unital algebra $\TT^Q_{\ov{d}}$. 
\end{definition}
\begin{remark}
\label{rem:TQiscohM}
In the case that $F=\emptyset$, I will recall in §\ref{sec:JoyceVA} that $\TT^Q_{\ov{d}}$ is naturally isomorphic to the cohomology of the moduli stack $\mM_{\ov{d}}$ of all perfect complexes of representations with dimension vector $\ov{d}$. Therefore $\tau_i(v)$ for $i>0$ form a set of natural generators of $H^{\bullet}(\mM_{\ov{d}})$.
\end{remark}
For the next definition, I assume for now that either one of the situations of Proposition \ref{prop:unirep} holds. This is so that there is a universal representation $(\ov{\UU},\ov{\Ff})$ on $M^{\mu}_{Q,\ov{d}}$. I will remove this restriction later on.
\begin{definition}
   Fix a universal representation $(\ov{\UU},\ov{\Ff})$ on $M^{\mu}_{Q,\ov{d}}$, then the \textit{realization morphism} $\xi_{\ov{\UU}}: \TT^Q_{\ov{d}}\to H^{\bullet}(M^{\mu }_{Q,\ov{d}})$ is a morphism of algebras defined on the generators by 
    $$
    \xi_{\ov{\UU}}\big(\tau_i(v)\big) = \ch_i(\UU_v) \quad\textnormal{for}\quad i>0,v\in V\,.
    $$
Unless it is strictly necessary, I will not specify the choice of the universal representation $\ov{\UU}$ from now on thus neglecting to write it in the subscript of $\xi_{\ov{\UU}}$.
\end{definition}

In \cite[Example 2.6, Remark 4.13]{BML}, we compared the form of the Virasoro constraints for moduli schemes of sheaves to the Chern character of the virtual tangent bundle expressed in terms of descendents. Presently, there is a similar correspondence which is what allows me to guess the correct Virasoro constraints in §\ref{sec:vircondef}. The following notations are introduced to allow a compact way to state this observation.
\begin{definition}
\label{def:eulerform}
For $(Q,\mR)$, I always fix a set $R$ of homogeneous generators of the ideal $\mR$. Note that if one includes non-trivial $F$, then based on the assumptions there are no relations that simultaneously start and end at a frozen vertex.
Define the \textit{adjacency matrix} $A^Q$ and the \textit{relation matrix} $S^Q$ with entries
\begin{align*}
A^Q_{v,w} &= \# \{e\in E\colon t(e) = v, h(e) = w\}\,,\\
S^Q_{v,w} &= \# \{r\in R\colon t(r) = v, h(r) = w\}
\end{align*}
labelled by $(v,w)\in V\times V$. Note that in the literature, one often means the adjacency matrix $A^{\Gamma}$ of the underlying undirected graph $\Gamma$ which is related to $A^Q$ by $A^{\Gamma} = A^Q+(A^Q)^{T}$. 
The \textit{naive Euler form} (which I from now on just call Euler form despite it being a misnomer by Remark \ref{rem:wrongobstruction}) of $(Q,F,\mR)$ is the pairing $\chi: \ZZ^{V}\times \ZZ^{V}\to \ZZ$ defined by 
\begin{equation}
\label{eq:eulerform}
\chi\big(\ov{c},\ov{d}\big) =\sum_{v\in V\backslash F}c_v\cdot d_v -\sum_{e\in E}c_{t(e)}\cdot d_{h(e)} + \sum_{r\in R}c_{t(r)}\cdot c_{h(r)}\,.
\end{equation}
In terms of the standard pairing $\langle \ov{c},\ov{d}\rangle=\sum_{v}c_v\cdot d_v $, this can be written as 
$$
\chi\big(\ov{c},\ov{d}\big) = \langle \ov{c},(\pi_{V\backslash F}-A^Q + S^{Q})\cdot \ov{d}\rangle\,,
$$
where $\pi_{V\backslash F}: \ZZ^V\to \ZZ^{V\backslash F}$ is the projection.
To draw a parallel between the Euler forms for quivers and varieties, I thus introduce the notation $\td(Q) = \pi_{V\backslash F}-A^Q + S^Q$.\footnote{Sometimes varieties can be derived equivalent to quivers which in particular identifies the Euler forms on both sides. It can be easily checked that in the case $D^b(\PP^1)\cong D^b(\bullet\rightrightarrows \bullet)$ the two Todd classes are not identified so this should be only viewed as notation.}
\end{definition}
\begin{remark}
    \label{rem:wrongobstruction}
    The reason for calling $\chi$ the naive Euler form is that it is not defined by 
    $$
    \chi(\ov{c},\ov{d}) =\sum_{i\geq 0}\Ext^i(\ov{U},\ov{W})
    $$
    where $\ov{\dim}(\ov{U}) =\ov{c}$ and $\ov{\dim}(\ov{W}) =\ov{d}$. Firstly, the homological dimension of a quiver with relations need not be restricted by 2, and secondly, the homogeneous generating set $R$ would need to be minimal so that one does not include redundant relations into the last sum in \eqref{eq:eulerform}. See also the comment in \cite[Method 6.22]{Jo21}.

The form $\chi$ is introduced in \cite[(6.5)]{Jo21} as a consequence of the obstruction theory of $M^{\mu}_{Q,\ov{d}}$ chosen there. It is recalled in \eqref{ex:Tvir} and is the obstruction theory of the \un{derived} vanishing locus enhancing the construction mentioned \ref{rem:relations}.  In this sense, the theory is truly meaningful only when the above two requirements on $Q$ and $R$ are satisfied.
\end{remark}
\begin{example}
\label{ex:Tvir}
 The virtual tangent bundle $T^\vir$ of $M^{\mu}_{Q,\ov{d}}$ used in \cite[§6.4.2]{Jo21} is constructed from the complex
$$
\bigoplus_{v\in V\backslash F} \UU_v^*\otimes \UU_v\stackrel{\mathfrak{f}_E}{\longrightarrow} \bigoplus_{e\in E} \UU^*_{t(e)}\otimes \UU_{h(e)}\stackrel{\mathfrak{s}_R}{\longrightarrow} \bigoplus_{r\in R} \UU^*_{t(r)}\otimes \UU_{h(r)}
$$
in degrees $[-1,1]$ by adding a copy of $\mO_{M_{\ov{d}}}$ in degree 0 if $F=\emptyset$. The morphism $\mathfrak{f}_E$ is the difference between pre-compositions and post-compositions by the universal morphisms 
$$\mathfrak{f}_E = \big( \circ \mathfrak{f}_{e} - \mathfrak{f}_{e}  \circ \big)_{e\in E}$$ 
with the domains of the morphisms $\circ \mathfrak{f}_e$ and $\mathfrak{f}_e \circ$ being $\UU^*_{h(e)}\otimes \UU_{h(e)}$ and $\UU^*_{t(e)}\otimes \UU_{t(e)}$, respectively. The map $\Fr$ has been spelled out in \cite[Definition 6.9]{Jo21}, and I will not recall it here.

Since $v\in \ZZ^V$ can stand for the generator of $\ZZ \cdot v$, one may consider its linear combinations, especially $\td(Q)\cdot v$. I use the notation 
$$
\tau_{i}\tau_j(\sum_{v,w\in V}c_{v,w}\,v\boxtimes w) = \sum_{v,w\in V}c_{v,w}\,\tau_i(v)\cdot \tau_j(w)
$$
for any formal expression $\sum_{v,w}c_{v,w}\,v\boxtimes w$ with rational coefficients $c_{v,w}$. After setting 
$$
\Delta_*\td(Q) = \sum_{v\in V}\td(Q)\cdot v\boxtimes v\,,
$$
and using the shorthand notation $\delta_n = \delta_{n,0}$, I can write the Chern character of $T^\vir$ in the compact form
\begin{equation}
\label{eq:Tvir}
\ch\big(T^\vir\big) = -\sum_{i,j\geq 0}(-1)^i\tau_i\tau_j\big(\Delta_*\td(Q)\big) + \delta_{|F|}\,.
\end{equation}
\end{example}
\subsection{Universal formulation of Virasoro operators for quivers}
\label{sec:vircondef}
Just as in the geometric case, Virasoro constraints for quivers are expressed mainly in terms of two operators $\bR_k, \bT_k: \TT^Q_{\ov{d}}\to \TT^Q_{\ov{d}}$ which add up to 
$$
\bL_k = \bR_k + \bT_k\,.
$$
The term $\bR_k$ does not change compared to \cite[§2.3]{BML} while $\bT_k$ can be derived from the form of \eqref{eq:Tvir} as explained already in \cite[Example 2.6]{BML}. They are defined for any $k\geq -1$ as follows:
\begin{definition}
\label{def:virop}
    \begin{enumerate}
        \item The operator $\bR_k$ is a derivation on $\TT^Q_{\ov{d}}$ acting on the generators by
        $$
        \bR_k\big(\tau_i(v)\big) = \prod_{j=0}^k (i+j)\,\tau_{i+k}(v)\quad \textnormal{for}\quad i>0, v\in V
        $$
        where the product is equal to 1 if $k=-1$,
        \item and $\bT_k$ acts by multiplication with 
        $$
      \sum_{\begin{subarray}{c} i+j=k \\ i,j \geq 0 \end{subarray}}i!j!\tau_i\tau_j\big(\Delta_*\td(Q)\big)+\delta_{k}(1-\delta_{|F|})\,.
        $$
    \end{enumerate}
\end{definition}
When $F=\emptyset$, the quiver Virasoro operators take identical form as the Virasoro operators for sheaf moduli schemes in \cite[§2.3]{BML} if one replaced the vertices $v$ with cohomology classes $\gamma$, $\td(Q)$ with $\td(X)$ for some variety $X$ and assumed $H^{p,q}(X) =0$ whenever $p\neq q$.
\begin{remark}
\label{rem:nopairVir}
As before, I distinguished two cases akin to sheaf and pair Virasoro constraints, but we will see that they are one and the same:
\begin{enumerate}[label=\alph*)]
    \item When $F= \emptyset$, the extra $\delta_{|F|}$ term in \eqref{eq:Tvir} coming from rigidification is omitted. This is equivalent to working with the virtual tangent bundle of the stack $\Mf_{Q,\ov{d}}$ and is the analog of what we called sheaf Virasoro constraints previously.
    \item When frozen vertices are present, the stack that I denoted by $\Mf_{Q,\ov{d}}$ is already the rigidified stack so one should look at the virtual tangent bundle of $[*/\GG_m]\times \mM_{Q,\ov{d}}$ instead to determine the $\bT_k$ operator. Using the projection $\pi_2$ to the second factor, the K-theory class of the virtual tangent bundle of $[*/\GG_m]\times \mM_{Q,\ov{d}}$ is $\pi_2^* T^{\vir} - 1$. This leads to the additional $+1$ term in the definition of  $\bT_0$. This behavior already appeared as pair Virasoro constraints in \cite[Conjecture 2.18]{BML}, but the correction was hidden within the formula itself.
\end{enumerate}
\end{remark}
In the rest of this work, I will continue the trend of unifying the two points of view that were previously separated. To be able to relate \cite{BML} to the present results, I will use dashed notation to denote the former conventions. In practice, this means the following: 
\begin{equation}
\label{eq:dashednot}
\bL'_k = \bR_k + \bT'_k = \bR_k +\bT_k - \delta_k(1-\delta_{|F|}) = \bL_k - \delta_k(1-\delta_{|F|})\,.
\end{equation}

\subsection{Weight zero Virasoro constraints for quivers}
\label{sec:framedw0vir}
The algebra of \textit{weight-zero descendents} is defined by
$$
\TT^Q_{\inv,\ov{d}} = \ker(\bR_{-1})\,.
$$
Just as in \cite{BML}, $\TT^Q_{\inv{,\ov{d}}}$ plays the role of the cohomology of the rigidified moduli stack of $\mM_{\ov{d}}$ (see also Remark \ref{rem:TQiscohM}). This is where the name comes from because its elements are precisely the classes in $\TT^Q_{\ov{d}}$ labelled in \cite[Definition 2.5]{Bo21} as weight zero with respect to the $[*/\Delta]$-action. An alternative but equivalent way to think about it was presented in \cite[§2.4]{BML} where we introduced it as the subalgebra of descendent classes $\tau$ in $\TT^Q_{\ov{d}}$ whose realization $\xi_{\ov{\UU}}(\tau)$ is independent of the choice $\ov{\UU}$.

To make Virasoro constraints also independent of choices, we introduced the \textit{weight-zero} Virasoro operator in \cite[Definition 2.11]{BML}.
\begin{definition}[{\cite[Definition 2.11]{BML}}]
\label{def:weight0}
   Define the weight-zero Virasoro operator by 
   $$
\Linv=\sum_{j\geq -1}\frac{(-1)^j}{(j+1)!}\bL_j \bR_{-1}^{j+1}: \TT^Q_{\ov{d}}\to \TT^Q_{\inv,\ov{d}}\,.
  $$
  Because integrals of the form
  $$
\int_{M^{\mu}_{Q,\ov{d}}}\xi_{\ov{\UU}}(\tau)\quad \textnormal{for any}\quad \tau \in \TT^Q_{\inv,\ov{d}}
  $$
  are independent of $\ov{\UU}$, I will omit specifying the choice of the universal representation altogether. I then say that $[M^{\mu}_{Q,\ov{d}}]$ satisfies Virasoro constraints if 
  $$
\int_{M^{\mu}_{Q,\ov{d}}}\xi\big(\Linv(\tau)\big)=0\quad \textnormal{for any} \quad  \tau\in \TT^Q_{\ov{d}}\,.
  $$
\end{definition}
To connect it to the more familiar formulation using countably many operators instead of just a single one, I recall that the weight-zero Virasoro constraints from the above definition are equivalent to including an auxiliary operator $\bS_k$. 
\begin{definition}
\label{def:Skvirasoro}
Fix a vertex $v\in V$ and assume that $(\ov{\UU},\ov{\Ff})$ is a \textit{$v$-normalized universal representation}, which means that it satisfies $$\xi_{\ov{\UU}} \big( \tau_1(v) \big) = 0\,.$$ 
Assuming that $d_v\neq 0$, the operator $\bS^{v}_k$ is given by $$
    \bS^v_k=-\frac{(k+1)!}{d_v}\bR_{-1}\big( \tau_{k+1}(v)\cdot -\big)\quad \textnormal{for}\quad k\geq -1\,.
    $$
    I will always choose $v\in F$ if possible in which case $\xi_{\ov{\UU}} \big( S^v_k(\tau)\big)$ is automatically 0 for all $k\geq 0$. 
Because the proof of \cite[Proposition 2.16]{BML} is formal, the weight zero Virasoro constraints from Definition \ref{def:weight0} are equivalent to the vanishing
    \begin{equation}
    \label{eq:SkVirasoro}
\int_{M^{\mu}_{Q,\ov{d}}}\xi_{\ov{\UU}} \big( (\bL_k + \bS^v_k)(\tau)\big)  = 0 \quad \textnormal{for each}\quad k\geq 0, \tau\in \TT^Q_{\ov{d}} \,.
    \end{equation}
     If it holds, I will say that $[M^{\mu}_{Q,\ov{d}}]$ satisfies the \textit{$v$-normalized Virasoro constraints}.
\end{definition}
Let us discuss a situation when $S^v_k$ becomes just multiplication by $\delta_{k}$. 
\begin{lemma}
\label{lem:Skvanishing}
    Suppose that $Q$ has no frozen vertices and $\ov{d}$ is its dimension vector such that $d_v = 1$ for a fixed $v\in V$, then there is a $v$-normalized universal representation $\ov{\UU}$ on $M^{\mu}_{Q,\ov{d}}$ such that $\xi_{\ov{\UU}}\big(\tau_i(v)\big) = 0$ for $i>0$ so $$\xi_{\ov{\UU}} \big( S^v_k(\tau)\big) = \delta_k \ \xi_{\ov{\UU}} ( \tau)\quad \textnormal{for} \quad k\geq 0$$  and any $\tau\in \TT^{Q}_{\ov{d}}$.
\end{lemma}
\begin{proof}
This is just a special case of the proof of Proposition \ref{prop:unirep}. One may simply take $\lambda_w = \delta_{v,w}$ which leads to the universal vector spaces $\tilde{\mU}_w = {\mU_{v}}^{-1}\otimes \mU_{w}$ on $\mathfrak{M}_{Q,\ov{d}}$. As a consequence of this, $\UU_{v} = \mO$ so the conclusion of the Lemma holds. 
\end{proof}
\begin{remark}
\label{rem:frozenisrigid}
   We saw that the presence of $F\neq \emptyset$ rigidifies the representations of $Q$. In fact, all the frozen vertices can be collected into a single one labelled by $\infty$ with $d_{\infty}=1$ as in Definition \ref{def:onefrozen} without changing the moduli schemes of (semi)stable representations. Fixing $U_{\infty} =\CC$, is precisely what removes a copy of $\GG_m$ from the automorphisms of $(\ov{U},\ov{f})$. This way, one may also think of the lack of the $S_k$ operators in the pair Virasoro constraints \cite[Conjecure 2.18]{BML} as a consequence of the above lemma.
\end{remark}

\subsection{Framed Virasoro constraints}
\label{sec:framedvir}
I begin by recalling the usual definition of framed quiver representations and formulating Virasoro constraints for their moduli schemes. I then discuss the well-known approach to expressing framed representations in terms of ordinary ones which allows me to derive the framed Virasoro constraints from \eqref{eq:SkVirasoro}.

\begin{definition}
\label{def:framedrep1}
 Let $Q$ be a quiver with relations and $F=\emptyset$, then a \textit{framed representation} of $Q$ with the dimension vector $\ov{d}$ and the \textit{framing vector} $\ov{n} = (n_v)_{v\in V}\in \Lambda_+$\footnote{The framing vector is always required to be non-zero.} is the data of a representation $(\ov{U},\ov{f})$ of $Q$ together with a collection of \textit{framing morphisms} $\phi_v\in \Hom(\CC^{n_v}, U_v)$ for all $v\in V$.   A framed representation is called \textit{stable} if there is no subrepresentation of $(\ov{U},\ov{f})$ through which the maps $\phi_v$ factor simultaneously. 

 The moduli scheme of stable framed representations  
      $
      M^{\ov{n}}_{Q,\ov{d}}
      $ was constructed by Nakajima in \cite[Theorem 3.5]{Nakajimaquiver},
      and it carries a canonical universal framed representation $(\ov{\UU},\ov{\Ff}, \ov{\Phi})$ where 
      $$
      \Phi_v: \un{\CC^{n_v}}\longrightarrow \UU_v
      $$
      is the universal framing morphism. The construction of the universal objects is analogous to the case of quivers with frozen vertices.
\end{definition}

Now, I transition directly to the definition of Virasoro constraints on $M^{\ov{n}}_{Q,\ov{d}}$ and demonstrate its implication in the case of flag varieties. 
\begin{definition}
\label{def:framedVir}
    The descendent algebra of a quiver $Q$ with a fixed framing vector $\ov{n}$ and the dimension vector $\ov{d}$ is given by $\TT^Q_{\ov{d}}$ from Definition \ref{def: descendentalgebra}. The Virasoro operators are still written as a sum 
    $$
    \bL^{\ov{n}}_k = \bR_k+\bT_k^{\ov{n}}
    $$
    where $\bR_k$ remains unchanged, and
    $$\bT_k^{\ov{n}} = \bT_k - k!\tau_k(\ov{n})+\delta_{k}\,.$$
    They lead to \textit{framed Virasoro constraints} defined by 
    \begin{equation}
        \label{eq:framedVC}
        \int_{M^{\ov{n}}_{Q,\ov{d}}}\xi_{\ov{\UU}} \big( \bL^{\ov{n}}_k(\tau)\big) = 0 \quad \textnormal{for each}\quad k\geq 0, \tau\in \TT^Q_{\ov{d}}\,.
    \end{equation}
\end{definition}
  \begin{example}
  \label{ex:flag}
  Consider the quiver $Q$ described in \eqref{Eq:equiv}. Fix a dimension vector $\ov{d}$ together with $d_l$ satisfying 
    $$
    d_l>d_{l-1}>\cdots >d_2>d_1
    $$
    and consider the framed representations of $Q$ for the framing vector $\ov{n}$ given by $n_{l-1} = d_l$ and $n_i = 0$ otherwise. Clearly, the condition of stability given in Definition \ref{def:framedrep1} enforces 
    \begin{enumerate}
        \item that the morphism $\phi_{l-1}: \CC^{d_l}\to U_{l-1}$ is surjective.
        \item surjectivity of the maps $V_{i}\to V_{i-1}$ for all $i=2,\ldots, l-1$. 
    \end{enumerate}
   One can use this to identify $M^{\ov{n}}_{Q,\ov{d}}$ with the partial flag variety $\Flag(\ov{d})$ of length $l$ parametrizing sequences of surjective morphisms of vector spaces
   $$
   \begin{tikzcd}
       \CC^{d_l}\arrow[r, twoheadrightarrow]&E_{l-1}\arrow[r,twoheadrightarrow]&E_{l-2}\arrow[r, twoheadrightarrow]&\cdots \arrow[r, twoheadrightarrow]&E_2\arrow[r, twoheadrightarrow]&E_1
   \end{tikzcd}
   $$
   where $\dim(E_i) = d_i$.
   Under this comparison, the universal quotients $\EE_i$ on $\Flag(\ov{d})$ are identified with the universal vector spaces $\UU_i$ for all $i=1,\ldots, l-1$. Therefore, the descendent classes can be alternatively expressed as $\tau_k(i) = \ch_k(\EE_i)$ for $i=1,\ldots, l-1$. This amounts to an explicit description of the realization map $\xi$, operators $\bR_k,\bT^{\ov{n}}_k$, and therefore Virasoro constraints on $\Flag(\ov{d})$. 
\end{example}

   By \cite[Proposition 3.3]{reineke}, one can express stable framed representations using just the definitions from §\ref{sec:background} and §\ref{sec:semrepmod} which recovers the framed Virasoro constraints from \eqref{eq:SkVirasoro}. The natural approach is to use the \textit{quiver framed at $\infty$} by an additional vertex as it is also more suitable for wall-crossing. Consider a quiver $Q$ (potentially with relations) together with a framing vector $\ov{n}$, and construct a quiver $^{\ov{n}}_{\infty}Q$ by adding a frozen vertex $\infty$ with $n_v$ arrows going from $\infty$ to each $v\in V$:
\begin{equation}
\label{eq:framedflag}
\begin{tikzcd}[column sep = large]
&&\overset{\infty}{\diamond}\arrow[dll, bend right = 30, "\times n_{l-1} "]\arrow[dl, bend right = 15, "\times n_{l-2}"]\arrow[d,  "\times n_{l-3}"]\arrow[drr, bend left = 30, "\times n_{1}"]&&\\
\overset{l-1}{\bullet}\arrow[r]&\overset{l-2}{\bullet}\arrow[r]&\overset{l-3}{\bullet}\arrow[r]&\cdots\arrow[r]&\overset{1}{\bullet}
\end{tikzcd}
\end{equation}
Explicitly, the sets of vertices and edges are described as
$$
^{\ov{n}}_{\infty}V = V\sqcup \{\infty\}\,,\qquad ^{\ov{n}}_{\infty}E = E\sqcup \bigsqcup_{v\in V}[n_v]
$$
where $[n_v] = \{1,\ldots, n_v\}$ and $t(e) = \infty, h(e) = v$ for any $e\in [n_v]$. The ideal of relations $\mR$ remains the same, and I continue using the same set of its generators $R$. Note that $\mA_{\ovninfty Q}$ is finite-dimensional if and only if $\mA_Q$ is. 
     \begin{definition}
     \label{def:framedrep2}
       A stable framed representation of $Q$ with dimension vector $\ov{d}$ and framing vector $\ov{n}$ is a stable representation of $^{\ov{n}}_{\infty}Q$ with the dimension vector $\ov{d}_{\infty} = (1,\ov{d})$  with respect to the slope stability
       $$
       \mu^{\fr}(c_{\infty},\ov{c})=\frac{c_{\infty}}{\sum_{v\in V} c_v + c_{\infty}}\,.
       $$ 
I will write \begin{equation}
      \label{eq:sumofphi}
      \phi^{\infty}_v: U_{\infty}^{\oplus n_v}\to U_{v}\end{equation} for the sum of all morphisms from $\infty$ to $v$.
      
      Note that $U_{\infty}$ can be identified with $\CC$ because $\infty$ is frozen. This maps every stable framed representation in the sense of the current definition to a representation of $Q$ with morphisms $\phi_v=\phi^{\infty}_v: \CC^{n_v}\to U_v$. It is not difficult to show  (see Proposition \cite[Proposition 3.3]{reineke} which deals with a reversed version of both definitions) that the two notions of stability introduced here and in Definition \ref{def:framedrep1} coincide, and thus one finds that
      $$
      M^{\mu^{\fr}}_{^{\ov{n}}_{\infty}Q,\ov{d}} \cong M^{\ov{n}}_{Q,\ov{d}}\,.
      $$
   There is a canonical universal representation on $M^{\mu^{\fr}}_{^{\ov{n}}_{\infty}Q,\ov{d}}$ with $\UU_{\infty}\cong \mO$ which can be identified with the one on $M^{\ov{n}}_{Q,\ov{d}}$ under the above isomorphism. Hence, I will not distinguish between the two definitions and will use only $M^{\ov{n}}_{Q,\ov{d}}$ to denote the resulting moduli scheme.
     \end{definition}
The next lemma follows immediately from the above comparison. 
\begin{lemma}
\label{lem:framedVir}
    The framed Virasoro constraints for $M^{\ov{n}}_{Q,\ov{d}}$ defined in \eqref{eq:framedVC} are equivalent to the $\infty$-normalized Virasoro constraints in \eqref{eq:SkVirasoro} under the identification in Definition \ref{def:framedrep2}.
\end{lemma}
\begin{proof}
Consider the morphism of unital algebras $\zeta: \DD^{\ovninfty{Q}}_{(1,\ov{d})}\to \DD^Q_{\ov{d}}$ defined by $\zeta\big(\tau_k(\infty)\big) = 0$ for $k>0$  and acting by identity on the rest of the generators. Then $\xi\circ\zeta = \xi$, and $\bR_k\circ\zeta = \zeta\circ \bR_k$ for all $k\geq 0$. This implies for any $k\geq 0$ and $\tau\in\DD^{\ovninfty{Q}}_{(1,\ov{d})}$ that 
$$
\xi\big((\bL_k+\bS^\infty_k)(\tau)\big) =\xi\big((\bR_k+\zeta(\bT_k) + \delta_k)\zeta(\tau)\big)\,.
$$
Observing that $\zeta(T_k) = T^{\ov{n}}_k$ completes the proof.
\end{proof}
Lastly, I discuss how to express all semistable representations of a quiver with frozen vertices in terms of a quiver framed at $\infty$. Therefore, when I do wall-crossing in §\ref{sec:reductiontop}, I will only consider quivers $Q$ with one or no frozen vertices.
\begin{definition}
\label{def:onefrozen}
    Let $(Q,F,\mR)$ be such that $F\neq \emptyset$, then define the subquiver $Q^m$ of $Q$ consisting of the vertices in $V\backslash F$ and the edges between them. If a dimension vector $\ov{d}$ of $Q$ is given, I will always write $\ov{d}^m$ for its restriction to $Q ^m$. Let us fix the dimension $d_v$ for all $v\in F$ and construct a new quiver $^F_{\infty}Q$ of the form \eqref{eq:framedflag} by adding a single vertex $\infty$ to $Q^m$. For each edge $e$ in $Q$ with $t(e)\in F$, add $d_{t(e)}$-many edges going from $\infty$ to $h(e)$. Note that all relations are contained in $Q^m$ by the 
    assumption on $\mR$ in Definition \ref{def:homorela}.
    
Consider the following quiver $Q$:
    $$
\begin{tikzcd}
&\overset{d_{(l-1)''}}{\diamond}\arrow[d]&&[-30pt]\overset{d_{(l-3)'}}{\diamond}\arrow[dr]&[-30pt]\overset{d_{(l-3)''}}{\diamond}\arrow[d]&\cdots&\arrow[d]\overset{d_{1'}}{\diamond}&\\
 \overset{d_{(l-1)'}}{\diamond}\arrow[r]&\underset{l-1}{\bullet}\arrow[r]&\underset{l-2}{\bullet}\arrow[rr]&&\underset{l-3}{\bullet}\arrow[r]&\cdots\arrow[r]&\underset{1}{\bullet}
\end{tikzcd}
$$
where I used dimensions to label the frozen vertices instead. Then the associated quiver $^F_\infty Q$ takes the form
$$
\begin{tikzcd}[column sep = large]
&&\overset{\infty}{\diamond}\arrow[dll, bend right = 30, "\times \big(d_{(l-1)'}\\ \ +d_{(l-1)''}\big)"']\arrow[d,  "\begin{subarray}{c} \times \big(d_{(l-3)'}\\ \ +d_{(l-3)''}\big)\end{subarray}"']\arrow[drr, bend left = 30, "\times d_{1'}"]&&\\
\overset{l-1}{\bullet}\arrow[r]&\overset{l-2}{\bullet}\arrow[r]&\overset{l-3}{\bullet}\arrow[r]&\cdots\arrow[r]&\overset{1}{\bullet}
\end{tikzcd}\,.
$$
Choose a stability parameter vector $\ov{\mu}\in \RR^Q$, then it defines a stability condition for the representations of $^F_{\infty}Q$ in terms of $^F_\infty\ov{\mu}\in \RR^{^F_{\infty}Q}$. The values of $^F_\infty\ov{\mu}$ are identical to $\ov{\mu}$ on the vertices of $Q^m$ and 
$$
\mu_{\infty} = \sum_{v\in F}d_v\mu_v\,.
$$
Continuing to use the notation from \eqref{eq:sumofphi}, I introduce for fixed $(d_v)_{v\in F}$ a map taking representations of $^F_{\infty}Q$ with the dimension vector $$(d_{\infty}, \ov{d}^m)\quad\textnormal{where}\quad d_{\infty} = \sum_{v\in F}d_v$$
to representations of $Q$ with the dimension vector $\ov{d}$. For each $v$, it acts by first constructing $\phi^{\infty}_v:\CC^{\oplus n_v}=U^{\oplus n_v}_{\infty}\to U_v$ where
$$n_v=\sum_{\begin{subarray}{c} e\in E:\\ t(e)\in F, h(e) =v\end{subarray}}d_{t(e)}\,.$$
I then partition $\CC^{\oplus n_v}$ into $\CC^{d_{t(e)}}$ for all $e$ appearing in the above sum in a fixed order.

This construction also identifies the two respective notions of slope stability and goes through for families so I described an isomorphism
\begin{equation}
\label{eq:isfroztofram}
M^{\mu}_{Q,\ov{d}}\cong M^{^F_{\infty}\mu}_{^F_{\infty}Q,(1,\ov{d}^m)}\,.
\end{equation}
Finally, observe that the virtual tangent bundles are identified under this isomorphism and the pullback along the above map of representations induces a morphism between the descendent algebras of $Q$ and $^F_{\infty}Q$ that commutes with the Virasoro operators. Therefore, Virasoro constraints on either side of the isomorphism \eqref{eq:isfroztofram} are equivalent to each other.
\end{definition}
\section{Vertex algebras from quiver moduli}
I recall here briefly the definition of vertex algebras and the only example of them that I will use -- the lattice vertex algebras. Then I set some notation for working with Joyce's geometric construction of lattice vertex algebras while recalling the main theorem. Most of this section summarizes parts of \cite[§3, §4]{BML} in the simpler case of quivers. The two main departures from loc. cit. are:
\begin{enumerate}
    \item I distinguish between the rigidified and non-rigidified vertex algebras for quivers with frozen vertices, while we used only the analog of the former before. The non-rigidified version is related to the special form of the $\bT_k$ operators in Definition \ref{def:virop} and the discussion below it.
    \item In §\ref{sec:VAframing}, I construct the framing vertex algebra $V^{\fr}_{\bullet}$ which is motivated by Definition \ref{def:framedrep1}. It is meant to make the symmetrized Euler pairing of the quiver non-degenerate, a role that was previously played by the vertex algebra $V^{\pa}_{\bullet}$ from \cite[Definition 4.4]{BML}.
\end{enumerate}
\subsection{Lattice vertex algebras}
\label{sec:recallVA}
Continuing to use my grading convention which leads to $\ZZ$-graded commutative objects, I recall that a vertex algebra on a graded vector space consists of 
\begin{enumerate}
\item a \textit{vacuum vector} $\ket{0}\in V_0$, 
    \item a linear \textit{translation operator}
    $T\colon V_\bullet\to V_{\bullet+1}$,
    \item and a \textit{state-field correspondence} which is a degree 0 linear map
    $$
    Y\colon V_\bullet\longrightarrow \End(V_\bullet)\llbracket z,z^{-1} \rrbracket\,,
    $$
    denoted by 
   \begin{equation}
   \label{eq:Yaz}
   Y(a,z)\coloneqq \sum_{n\in\ZZ}{a_{(n)}}z^{-n-1}\,,
    \end{equation}
   where $a_{(n)}:V_\bullet\rightarrow V_{\bullet+\deg(a)-n-1}$ and $\deg (z) =-1$.
\end{enumerate}
The conditions needed to be satisfied by this collection of data can be found for example in \cite[§4]{Borcherds}\footnote{In this reference, one should first assume that the vector space $V$ only admits even parity.}, \cite[(1.3.3)]{Ka98}, or \cite[Definition 3.1.1]{LLVA}. Let $a,b,c\in V_{\bullet}$, then two of the conditions are 
\begin{align}
\label{eq:skew}
a_{(n)}b &= \sum_{i\geq 0}(-1)^{|a||b|+i+n+1}\frac{T^i}{i!} b_{(n+i)}a\,,\\
\big(a_{(m)}b\big)\raisebox{0.2 pt}{\ensuremath{_{(n)}}} c
&=\sum_{i\geq 0}(-1)^i \binom{m}{i}\Big[a_{(m-i)}\big(b_{(n+i)}c\big) -(-1)^{|a||b|+m}b_{(m+n-i)}\big(a_{(i)}c\big)\Big]\,.\label{eq:jacobi}
\end{align}
I refer to \cite{Borcherds}, \cite[Chapters 3 and 4]{Ka98}, \cite[§3]{LLVA} and \cite[§3.1]{BML} for the basic properties of vertex algebras with the last reference collecting precisely what is necessary for the present work.
\begin{example}
\label{ex:latticeVA}
 The only explicit examples of vertex algebras I will be working with here are the \textit{lattice vertex algebras} as discussed in \cite{Borcherds}, \cite[§5.4, §5.5]{Ka98}, \cite[§6.4,§6.5]{LLVA}, and \cite[§3.2]{BML}.  Starting from a free abelian group and a not necessarily symmetric pairing  $q:\Lambda\times \Lambda\to \ZZ$, define 
 $$
 \mathcal{Q}(\alpha,\beta) = q(\alpha,\beta) + q(\beta,\alpha)\quad  \textnormal{for}\quad \alpha,\beta\in \Lambda\,, \qquad \Lambda_{\CC} = \Lambda\otimes_{\ZZ}\CC\,.
 $$
 The underlying vector space of the lattice vertex algebra associated with the lattice $(\Lambda,\mQ)$ is defined by 
$$
V_{\bullet} = \CC[\Lambda]\otimes \TT_{\Lambda}\,, \quad \TT_{\Lambda} = \SSym\big[T_{\Lambda}\big]\,,\qquad T_{\Lambda} = \bigoplus_{k>0}\Lambda_{\CC}\cdot t^{-k}
$$
with the generators $x_{k} = x\cdot t^{-k}\in \TT_{\Lambda}$ and $e^\alpha\in \CC[\Lambda]$ whenever $x\in \Lambda_{\CC}$ and $\alpha\in \Lambda$. Note that $V_{\bullet}$ carries a natural algebra structure since it is a tensor product of a group algebra and a polynomial algebra. However, the multiplication will seem ad hoc from the geometric perspective I describe later on. 

 After fixing the grading 
$$
\deg \Big(e^{\alpha}\otimes x_{k_1}^1\cdots x_{k_n}^n\Big) = \sum_{i=1}^n k_i  
 + q(\alpha,\alpha)\quad \textnormal{for}\quad x^1,\ldots,x^n \in \Lambda_{\CC}\,,
$$
one endows $V_{\bullet}$ with the following vertex algebra structure:
\begin{enumerate}
    \item One sets $|0\rangle\coloneqq e^0\otimes 1\in V_0$ for the vacuum vector.
    \item The translation operator is determined as the derivation on the algebra $V_{\bullet}$ satisfying 
\begin{equation}
\label{eq:Tabstract}
T(x_{k}) = kx_{k+1}\,,\qquad T(e^{\alpha}) = e^\alpha\otimes \alpha_{1}\,. 
\end{equation}
\item Let us start by defining 
\begin{equation}
\label{eq:fields1}
 Y(x_{1},z) = \sum_{k\geq 0}{x_{k+1}\cdot} \ z^{k} +\sum_{k\geq 1}\sum_{b\in B} k \mQ(x,b)\frac{\partial}{\partial b_{k}}z^{-k-1} + \mQ(x,\beta)z^{-1}
\end{equation}
 where $B\subset \Lambda$ is a basis of the lattice and $x_{k+1}\cdot$ denotes the multiplication by $x_{k+1}$. This is how the action of $x_{(k)}$ on $e^\beta\otimes \TT_{\Lambda}$ for all $k\in \ZZ$ is defined. For example, one sees that $x_{(-k)} =  x_k\cdot \ $ for $k\geq 1$.
 
The rest of the state-field correspondence is described by 
\begin{align*}
 Y(e^{\alpha},z) = (-1)^{q(\alpha, \beta)}z^{\mQ(\alpha,\beta)}e^{\alpha} \exp\Big[-\sum_{ k<0}\frac{\alpha_{(k)}}{k}z^{-k}\Big]\exp\Big[-\sum_{k>0}\frac{\alpha_{(k)}}{k}z^{-k}\Big]\,,
 \end{align*}
  in combination with the reconstruction-like theorem
  \begin{multline}\label{Eq: reconstruction}     
Y\big(e^\alpha\otimes x^1_{k_1+1}\cdots x^n_{k_n+1},z\big)\\
= \frac{1}{k_1!k_2!\cdots k_n !}:Y(e^\alpha,z)\frac{d^{k_1}}{(dz)^{k_1}}Y(x^1_{1},z)\cdots \frac{d^{k_n}}{(dz)^{k_n}}Y(x^n_{1},z): \,.
\end{multline}
\end{enumerate}
This collection of data determines the vertex algebra structure on $V_{\bullet}$ uniquely.
\end{example}
\subsection{Conformal element}
\label{sec:conformalel}
The main novelty of \cite{BML} is to phrase Virasoro constraints that appeared for the first time in \cite{moop} in terms of a choice of a conformal element in the lattice vertex algebra constructed by Joyce. The present work is simpler, as there is no (anti-)fermionic part of the vertex algebra, and there is less freedom in choosing a conformal element. I now recall the definition of a conformal element and its form in Example \ref{ex:latticeVA}.

\begin{definition}\label{Def: conformal vertex algebra}
 A \textit{conformal element} $\omega$ of a vertex algebra $V_\bullet$ is an element of $V_4$ such that its associated operators $L_n=\omega_{(n+1)}$ satisfy
 \begin{enumerate}
     \item the Virasoro commutation relations
     \[\big[L_n,L_m\big] = (n-m)L_{n+m} +\frac{n^3-n}{12}\delta_{n+m,0}\cdot C
     \,,\]
     where $C\in \CC$ is a constant called the \textit{conformal charge} of $\omega$,
     \item $L_{-1}  = T\,,$
     \item and $L_0$ is diagonalizable. 
 \end{enumerate}
A vertex algebra $V_\bullet$ together with a conformal element $\omega$ is called a \textit{vertex operator algebra}.
\end{definition}
Suppose now that $V_\bullet$ is a lattice vertex algebra for a non-degenerate lattice $(\Lambda,\mQ)$, then there is a simple construction of a conformal element which was already observed by Borcherds \cite[§5]{Borcherds}. Fix a basis $B$ of the lattice $\Lambda$ and its dual $\hat{B}$ with respect to $\mQ$, then 
$$
\omega= \frac{1}{2}\sum_{b\in B}  \hat{b}_{(-1)}b_{(-1)}\ket{0}
$$
is a conformal element with the conformal charge $C = \rk(\Lambda)$. It is then well-known and easy enough to check that $$L_0(v) = \deg(v)\cdot v$$
for a pure degree element $v\in V_{\bullet}$.

\begin{remark} In \cite[Example 5.5b]{Ka98}, Kac mentions that in some cases there is a way of constructing alternative conformal structures relying on the boson-fermion correspondence in \cite[§5.2]{Ka98}. If one follows the rest of the present paper, the new conformal element would lead to new Virasoro operators. The operators would no longer have such a clear relation to the form of $\ch(T^\vir)$ in \eqref{eq:Tvir} so I will not be interested in the resulting constraints here.
\end{remark}

\subsection{Joyce's vertex operator algebra for quivers}
\label{sec:JoyceVA}

The primary mechanism that does most of the heavy lifting in proving Virasoro constraints is Joyce's geometric vertex algebra construction \cite{Jo17} and the wall-crossing that it allows to formulate (see \cite{GJT, Jo21}). I will briefly summarize the main ingredients in the case of quivers and set some notation necessary for working with it based on its explicit description in \cite{Jo17}.

In this section and §\ref{sec:WC}, I will need three types of quivers -- $Q$ with no frozen vertices, the framed quiver $\Qfr$ associated with it, and a further quiver defined in §\ref{sec:VAframing} with as many frozen vertices as there are vertices in $Q$. For this reason, I first specify the collection of data required to construct the vertex algebra for any quiver with any $F\subset V$, but I modify this data later on to suit the specific cases and to make the definition compatible with \cite{Jo21}.

\begin{definition}
\label{def: VOAconstruction}
\begin{enumerate}
    \item I work with the higher moduli stack $\MQ$ of perfect complexes of representations of $(Q,\mR)$ constructed by Toën-Vaquié \cite{TV07} which admits a universal object $(\ov{\mathscr{U}},\ov{\varphi})$. Introducing $F\neq\emptyset$ does not change the definition of the stack $\mM_Q$ but rather the other ingredients below. I will also need the direct sum map
$\Sigma\colon\MQ\times \MQ\to \MQ
\,,$
such that 
$$
\Sigma^\ast\ \ov{\mathscr{U}} = \ov{\crU} \boxplus \ov{\crU}\,,
$$
and an action $\rho\colon [*/\GG_m]\times\MQ\to \MQ$ determined by
$$
\rho^\ast\ \ov{\crU} = \mV\boxtimes \ov{\crU}
$$
for the universal line bundle $\mV$ on $[*/\GG_m]$.
   \item The piece of data needed to obtain a vertex algebra that interacts the most with Virasoro constraints is the complex
    \begin{equation}
    \label{eq:extcomplex}
    \Ext'_Q =
\bigoplus_{v\in V\backslash F} \crU_v^\vee\boxtimes \crU_v\stackrel{\varphi_E}{\longrightarrow} \bigoplus_{e\in E} \crU^\vee_{t(e)}\boxtimes \crU_{h(e)}\stackrel{\varsigma_R}{\longrightarrow} \bigoplus_{r\in R} \crU^\vee_{t(r)}\boxtimes \crU_{h(r)}
    \end{equation}
   on $\mM_Q\times \mM_Q$ in degrees $[-1,1]$.  Note that in some cases as explained in Remark \ref{rem:wrongobstruction}, this restricts at each $\CC$-point $\big([\ov{U}^\bullet], [\ov{W}^\bullet]\big)\in \MQ\times \MQ$ to $\Ext^\bullet\big(\ov{U}^\bullet, \ov{W^\bullet}\big)$. Denoting by  $\sigma\colon \MQ\times \MQ\to \MQ\times \MQ$ the map swapping the two factors, one constructs the symmetrized complex
    $$
    \Theta'_Q = (\Ext'_Q)^\vee\oplus \sigma^\ast\Ext'_Q
    $$
   satisfying
    $
   \sigma^* \Theta'_Q\cong (\Theta'_Q)^\vee.
    $
    \item For any $\ov{d}\in K^0(Q)\simeq \pi_0(\mathcal{M}_Q)$, I denote the corresponding connected component by $\mM_{\ov{d}}\subset\MQ$. Any restriction of an object living on $\MQ$ to $\mM_{\ov{d}}$ will be labelled by adjoining the subscript $(-)_{\ov{d}}$. This allows one to write
    $$
   \chi'_{\sym}(\ov{c},\ov{d})\coloneqq \rk\big(\Theta'_{Q,\ov{c},\ov{d}}\big) = \chi(\ov{c},\ov{d}) + \chi(\ov{d},\ov{c})
    $$
    for $\chi\colon K^0(Q)\times K^0(Q)\to \ZZ$ the Euler form \eqref{eq:eulerform}.
 \end{enumerate}
 When $F=\emptyset$, I will omit writing the dash $(-)'$ leading to just $\Ext_Q$, $\Theta_Q, \chi_{\sym}$ instead.
\end{definition}
\begin{remark}
\label{rem:rigidvsnonrigid}
   Note that in the presence of $F\neq \emptyset$, the above definition will eventually lead to an analog of a vertex algebra of pairs as in \cite[§8.2.1]{Jo21}. I will need to introduce corrections $\Ext_Q, \Theta_Q$ of $\Ext'_Q, \Theta'_Q$ to make up for the endomorphisms at frozen vertices. This will in the end recover the term $\delta_k$ from $\bT_k$ in Definition \ref{def:virop}, and it is necessary for making a connection with Joyce's \cite[(8.28)]{Jo21} which describes the $\Ext$-complex for the stack of pairs of sheaves. This is natural because without including these corrections, one would be describing the obstruction theory of the rigidified stack. But this is where the Lie algebra, not the vertex algebra should be constructed (see §\ref{sec:geominterpLie}). 
  Still, it turns out that finding conformal elements is much easier in the case when the rigidified $\Ext'_Q$ complexes are used. I will therefore work with two vertex algebras -- $V_{Q,\bullet}$ compatible with the pair vertex algebra of Joyce and thus used to formulate wall-crossing, and $V'_{Q,\bullet}$ defined in terms of $\Ext'$ and meant to accommodate a conformal element $\omega$. I will also show that restricted to the current applications,  the Lie brackets induced by both of the vertex algebras coincide.
\end{remark}
Let us begin by discussing the vertex algebra $V'_{Q,\bullet}$ the role of which was explained in the above remark. I will focus on the modifications necessary to remove the dash later.

Joyce's vertex algebra lives on the homology of $\mM_Q$ so I first give an overview of the basic definitions:
\begin{enumerate}
\item The (co)homology is defined by first acting with the topological realization functor $(-)^{\textnormal{t}}$ from \cite[§3.4]{Blanc} which gives a topological space $\MQ^{\tet}$ and only then taking (co)homology. In other words, $H^\bullet(\MQ) = H^\bullet(\MQ^{\textnormal{t}})$.
\item For a perfect complex $\mE$ on $\MQ$, one can construct its associated K-theory class using the classifying map $\mE: \MQ\to \Perf_{\CC}$ of the same name. Note that $(\Perf_{\CC})^{\tet} = BU\times \ZZ$ by \cite[§4.2]{Blanc}. So, acting with $(-)^{\tet}$ leads to $\mE^{\textnormal{t}}:\MQ^{\textnormal{t}}\to BU\times\ZZ$, and one can pull-back the universal K-theory class $\mathfrak{U}$ from $BU\times \ZZ$. I will use the notation $\mE^{\tet}=(\mE^{\textnormal{t}})^*(\mathfrak{U})$.
\item Lastly, Chern classes of $\mE$ are defined by $c(\mE) = c(\mE^{\textnormal{t}})$, and so are Chern characters.
\end{enumerate}

The underlying vector space of $V'_{Q,\bullet}$ is obtained by shifting the grading on homology of $\MQ$:
$$
V'_{Q,\bullet}= \widehat{H}'_{\bullet}(\mM_Q) =   \bigoplus_{\ov{d}\in K^0(Q)}\widehat{H}'_{\bullet}(\mM_{\ov{d}})\,,\qquad \widehat{H}'_{\bullet}(\mM_{\ov{d}}) =H_{\bullet -2\chi(\ov{d},\ov{d})}(\mM_{\ov{d}})\,.
$$
Joyce \cite[§5.2.2, §5.3.2]{Jo17}, \cite[§6.2]{Jo21} gave its precise description so I will summarize the main points.
I begin with an assumption that generalizes the homogeneity assumption in \cite[Definition 6.6]{Jo21} (note that it is different from the one in Definition \ref{def:homorela}).
\begin{assumption}
\label{ass:homotopyref}
    Assume that the quiver with relations $(Q,\mR)$ admits a (homogeneous) generating set $R$ with a set of edges $E_R\subset E$  satisfying the following: for a fixed $r\in R$ expressed as the linear combination of paths 
    $
    r = \sum a_{\gamma}\gamma
    $, there exists a \textit{contracting number} $c_r>0$ such that the number of edges from $E_R$ contained in each $\gamma$ is $c_r$. I will call such $E_R$ a \textit{contracting set}. In the presence of cycles, a fixed edge can appear with higher multiplicity. In this case, each appearance contributes to the value of $c_r$.
\end{assumption}
\begin{example}
 A relation $r$ is homogeneous in terms of \cite[Definition 6.2]{Jo21} if it is a non-zero linear combination of paths of the same length. If $R$ consists only of such $r$, then one can choose $E_R=E$. Consider the other extreme when there is only a single relation $R=\{r\}$, then one can choose the first edge along each path $\gamma$ of $r$ independent of its length. A mix between the two served as a motivation for introducing the above definition which is meant to remove the restrictive condition on the length of paths $\gamma$ for $r$ to be homogeneous.     
\end{example}

Consider now the inclusion 
$$
 \prod_{v\in V} \Perf_{\CC}\hookrightarrow \MQ
$$
obtained by mapping the collection of complexes $(U^\bullet_v)_{v\in V}$ to the representation complex with zero maps for all edges between them. The point of Assumption \ref{ass:homotopyref} is that one can construct an $\BA^1$-homotopy inverse of this map as
$$
(\ov{U}^\bullet, \ov{f})\mapsto (\ov{U}^\bullet, 0)\,.
$$
One can construct an $\BA^1$-homotopy between the above two representations by first rescaling 
$$
\BA^1\ni\lambda \mapsto \lambda f_e \quad \text{for}\quad e\in R\,,
$$
and then doing so for the rest of the edges. Each relation $r\in R$ is preserved by the scaling in the first step because it is proportional to $\lambda^{c_r}$ when applied to $\lambda f_e$.

Moreover, the pullback of the universal complexes $\crU_v$ returns universal complexes on each factor $\Perf_{\CC}$ denoted simply by $\crU$ when working with a fixed copy of $\Perf_{\CC}$. 
This gives an identification
\begin{equation}
\label{eq:xi}
\xi: \TT^Q_{\ov{d}} \stackrel{\sim}{\longrightarrow} H^\bullet(\mM_{\ov{d}}) \,, \qquad \tau_i(v)\mapsto \ch_i(\crU_v)
\end{equation}
whenever $$\ov{d}\in \Lambda = K^0(Q)\,.$$ To see that this is an isomorphism, note that the connected components of $\prod_{v\in V}\Perf_{\CC}$ are labelled by dimension vectors $\ov{d}\in \Lambda$ and $$H^{\bullet}(\Perf_{\CC}) \cong H^{\bullet}(BU\times \ZZ)\cong \CC[\ZZ]\otimes \SSym\llbracket\ch_i(\mathfrak{U}), i>0\rrbracket\,.$$ 

Finally, I construct the isomorphisms
\begin{equation}
\label{eq:xidagger}
\xi^\dagger:\widehat{H}'_{\bullet}(\mM_{\ov{d}})\stackrel{\sim}{\longrightarrow} e^{\ov{d}}\otimes \TT_{\Lambda}\,,\qquad \ov{d}\in \Lambda
\end{equation}
which is graded if one uses the shifted grading from Example \ref{ex:latticeVA} for the right-hand side.

To do so, I fix the pairing 
\begin{equation}
\label{eq:pairingTT}
\langle -,-\rangle : T^Q_{\ov{d}}\times T_{\Lambda}\longrightarrow \CC\,,\qquad \langle \tau_k(v),x_{-j}\rangle = \frac{x_v}{(i-1)!}\delta_{k,j}
\end{equation}
where $x=\sum_{v\in V}x_v v$. By §\ref{sec:gradedalg}, this induces a cap product $\cap: \TT^Q_{\ov{d}}\times \TT_{\Lambda}\to \TT_{\Lambda}$ that can be written as 
\begin{equation}
\label{eq:tauder}
\tau_k(v)\cap =\frac{1}{(k-1)!}\frac{\partial}{\partial v_{-k}}\,. 
\end{equation}
 The isomorphism $\xi^{\dagger}$ is then defined as the dual of $\xi$.  Using $\cap$ to also denote the usual topological cap product between cohomology and homology, this means that the following diagram commutes
\begin{equation}
\label{Eq: homologyisomorphism}
\begin{tikzcd}
\TT^Q_{\ov{d}}\times\TT_{\Lambda} \arrow[d,"{\xi\times \xi^{-\dagger}}"', "\sim"]\arrow[r,"\cap"]&\TT_{\Lambda}\arrow[d,"{\xi^{-\dagger}}"]\\
H^\bullet(\mM_{\ov{d}})\times H_\bullet(\mM_{\ov{d}})\arrow[r,"\cap"]&H_\bullet(\mM_{\ov{d}})
\end{tikzcd}\,.
\end{equation}
\begin{remark}
\label{rem:forgetR}
    An even simpler way to circumvent restrictions on the set $R$ is to forget it altogether and work with the embedding $\iota_R$ of $\mM_Q$ into the stack of $Q$ without relations. This way, one still obtains the morphism \eqref{eq:xi}, but it may no longer be invertible. Still, it is sufficient for defining $\xi^{\dagger}$ which can be identified with $(\iota_R)_*$. One can modify the vertex algebra for $Q$ without relations by adding to the K-theory class $\Ext'_Q$ the term
$$
\sum_{r\in R} \crU_{t(r)}^\vee\boxtimes \crU_{h(r)}
$$
by hand. This way, $\xi^{\dagger}$ becomes a morphism of vertex algebras which will induce a morphism of Lie algebras in §\ref{sec:Lie}. Using the methods following this remark, one can then prove Virasoro constraints independent of what $R$ looks like. The main drawback of this approach is that one is not studying the classes $[M^{\mu}_{Q,\ov{d}}]$ in the homology of $\mM_Q$ but rather its ambient stack that forgets $R\subset\mR$. So, unless Assumption \ref{ass:homotopyref} is satisfied, one might lose some information when pushing forward along $\iota_R$. Doing so, however, is compatible with the realizations of the descendent classes. For this reason, I will forget about Assumption \ref{ass:homotopyref} from now on.
\end{remark}
The next result was originally only proved for quivers with $F=\emptyset$ but the extension to any $F$ is straightforward.
\begin{theorem}[{\cite[Theorem 3.12, Theorem 5.19]{Jo17},}]\label{thm:geometricvaconstruction}
Suppose that the Assumption \ref{ass:homotopyref} holds, then the vertex algebra structure on $V'_{Q,\bullet}$ defined by 
\begin{enumerate}
    \item using the inclusion $0\colon*\to \MQ$ of a point corresponding to the zero object to construct
     $$
    \ket{0}=0_*(*)\in H_0(\MQ)\,,
    $$
    \item taking $t\in H_2\big([*/\GG_m]\big)$ to be the dual of $c_1(\mV)\in H^2\big([*/\GG_m]\big)$ and setting 
    \begin{equation}
    \label{eq:Translation}
    T(a) = \rho_\ast\big(t\boxtimes a\big)\quad \textnormal{for}\quad a\in V'_{Q,\bullet}\,,
    \end{equation}
    \item introducing the state-field correspondence by the formula
    \begin{equation}\label{eq: joycefields}
    Y(a,z)b = (-1)^{\chi(\ov{c},\ov{d})}z^{\chi'_\sym(\ov{c},\ov{d})}\Sigma_\ast\Big[(e^{zT}\boxtimes \textnormal{id})\big(c_{z^{-1}}(\Theta'_Q)\cap (a\boxtimes b)\big)\Big]
    \end{equation}
for any $a\in \widehat{H}'_\bullet(\mathcal{M}_{\ov{c}})$ and $b\in \widehat{H}'_\bullet(\mathcal{M}_{\ov{d}})$.
\end{enumerate}
is the lattice vertex algebra from Example \ref{ex:latticeVA} for $\Lambda = K^0(Q), \mathcal{Q} =\chi'_{\sym}$ under the isomorphism $\xi^\dagger:\widehat{H}'_{\bullet}(\MQ)\xrightarrow{\sim}\CC[\Lambda]\otimes \TT_{\Lambda}$ defined in \eqref{eq:xidagger}. If $\chi'_{\sym}$ is non-degenerate, one can construct a dual basis $\hat{V}=\{\hat{v}\colon v\in V\}$ of $V$ in $\Lambda_{\CC}$ such that 
$$
\chi'_{\sym}(\hat{v},w) =\delta_{v,w} \quad  \textnormal{for}\quad v,w\in V\,.
$$
Then $V'_{Q,\bullet}$ can be promoted to a vertex operator algebra with the conformal element
\begin{equation}
\label{eq:omega}
\omega = \sum_{v\in V}\hat{v}_{(-1)}v_{(-1)}\ket{0}
\end{equation}
of conformal weight $C=|V|$.
\end{theorem}
When $F=\emptyset$, I again omit the dash by writing $V_{Q,\bullet}$ or when $Q$ is understood just $V_{\bullet}$ instead of $V'_{Q,\bullet}$. When frozen vertices are present and $\omega$ exists, I will call it the \textit{rigidified conformal element} for the reasons explained in Remark \ref{rem:rigidvsnonrigid}.

For later purposes, I recall that combining the field equation in \eqref{eq:fields1} with \eqref{eq:tauder} leads to 
\begin{equation}
\label{eq:fieldastau}
 Y(x_{1},z) = \sum_{k\geq 0}x_{k+1}\cdot 
 z^{k} +\sum_{k\geq 1}\sum_{v\in V} k !\chi'_{\sym}(x,v) \tau_k(v) z^{-k-1} + \chi'_{\sym}(x,\beta)z^{-1}
\end{equation}
where I omitted writing the cap product after $\tau_k(v)$ and will continue to do so.
 \subsection{Framing vertex algebra}
 \label{sec:VAframing}
 Let us from now on assume that $Q$ has no frozen vertices as I will add these by hand. Thus the quiver is given by 
  $$(Q,\mR)\,,\quad Q = (V,E)\,,\quad F=\emptyset \,.$$
 The non-degeneracy of the pairing $\chi_{\sym}$ is necessary to construct the conformal element. However, it can be checked immediately that it does not hold even in the simple example of Kronecker quiver with two arrows
$$
Q=\begin{tikzcd}
\overset{1}{\bullet}\arrow[r, shift left = 0.2em]\arrow[r, shift right = 0.2em]&\overset{2}{\bullet}
\end{tikzcd}\,.
$$
 Following the idea presented by \cite[§5]{Borcherds}, we have resolved a similar issue in \cite{BML} by embedding $\Lambda$ into a larger lattice $\Lambda^{\pa}$ this time with non-degenerate symmetric pairing $\chi^{\pa}_{\sym}$. The restriction of the resulting Virasoro operators to $V_{\bullet}$ still denoted by $L_k$ is independent of the choice of $(\Lambda^{\pa}, \chi^{\pa}_{\sym})$ so we picked one with a suitable geometric interpretation. In the present work, a natural choice is motivated by Definition \ref{def:framedrep1}.
 
 Before I introduce the ostensibly non-degenerate pairing, I will go over the non-rigidified version of $V'_{\Qfr,\bullet}$. 
 \begin{definition}
 
     Fix a dimension vector $\ov{n}$, and consider the quiver $\Qfr$ from \eqref{eq:framedflag}. For it, I define
     \begin{enumerate}
    \item  The non-rigidified $\Ext$-complex
     \begin{equation}
     \label{eq:ext}
     \Ext_{\Qfr} = \bigoplus_{v\in \ovninfty V} \crU_v^\vee\boxtimes \crU_v\longrightarrow \bigoplus_{e\in \ovninfty E} \crU^\vee_{t(e)}\boxtimes \crU_{h(e)}\longrightarrow \bigoplus_{r\in R} \crU^\vee_{t(r)}\boxtimes \crU_{h(r)}
     \end{equation}
     which is just the cone of 
     $$
     \crU^\vee_{\infty}\boxtimes \crU_{\infty}\longrightarrow \Ext'_{\Qfr}\,.
    $$
    Its symmetrization
    $$
    \Theta_{\Qfr} = (\Ext_{\Qfr})^\vee \oplus \sigma^* \Ext_{\Qfr}
    $$
    replaces $\Theta'_{\Qfr}$.
    \item Recall that one writes the dimension vector of $\Qfr$ as $$(d_{\infty}, \ov{d})=\ov{d}_{\infty}\in \ovninfty \Lambda = K^0(\Qfr) \,.$$ Denoting by $\ovninfty\chi' , \ovninfty\chi'_{\sym}$ the pairings from definitions \ref{def:eulerform} and \ref{def: VOAconstruction}, I define their corrections
    \begin{align*}
    \label{eq:comparingforms}
    \ovninfty \chi\big(\ov{c}_{\infty},\ov{d}_{\infty}\big) &= \rk\big(\Ext_{\Qfr,\ov{c}_{\infty},\ov{d}_{\infty}}\big) = \ovninfty \chi'\big(\ov{c}_{\infty},\ov{d}_{\infty}\big) + c_{\infty}\cdot d_{\infty}\,,\\
    \ovninfty \chi_{\sym}\big(\ov{c}_{\infty},\ov{d}_{\infty}\big)& = \rk\big(\Theta_{\Qfr,\ov{c}_{\infty}, \ov{d}_{\infty}}\big) = \ovninfty\chi'_{\sym}\big(\ov{c}_{\infty},\ov{d}_{\infty}\big) + 2c_{\infty}\cdot d_{\infty}\,.
    \numberthis
    \end{align*}
     \end{enumerate}
Replacing the original dashed ingredients in Theorem \ref{thm:geometricvaconstruction} with the ones from the current definition results in a new vertex algebra with the underlying graded vector space 
$$
\ovninfty V_{\bullet} =\widehat{H}_{\bullet}(\mM_{\Qfr}) = \bigoplus_{\ov{d}_{\infty}\in \ovninfty \Lambda}\widehat{H}_{\bullet}(\mM_{\ov{d}_{\infty}}) = H_{\bullet-2 \ \ovninfty\chi(\ov{d}_{\infty},\ov{d}_{\infty})}(\mM_{\ov{d}_{\infty}})\,.
$$
 \end{definition}
Note that this in general does not solve the problem with non-degeneracy. In fact, it can make it worse as can be seen by framing the $A_2$ quiver with the framing vector $\ov{n}=(1,1)$. Instead, I construct a new quiver with an extra set of frozen vertices $(V)$ which is just equal to $V$ but I use $(-)$ to distinguish the new copy.
I also add an edge between each new vertex and its original copy leading to the data: $$(Q^{\fr}, F^{\fr},\mR)\,,\quad Q^{\fr}=\big((V)\sqcup V, E\sqcup V\big)\,,\quad F^{\fr} = (V)\,.$$
 For the Kronecker quiver and the example from \eqref{Eq:equiv}, this leads to 
$$
\begin{tikzcd}
\overset{(1)}{\diamond}\arrow[d]&\arrow[d]\overset{(2)}{\diamond}\\
\overset{1}{\bullet}\arrow[r, shift left = 0.2em]\arrow[r, shift right = 0.2em]&\overset{2}{\bullet}
\end{tikzcd}\,,
$$
$$
\begin{tikzcd}
\overset{(l-1)}{\diamond}\arrow[d]&\overset{(l-2)}{\diamond}\arrow[d]&\overset{(l-3)}{\diamond}\arrow[d]&\cdots&\arrow[d]\overset{(1)}{\diamond}\\
\underset{l-1}{\bullet}\arrow[r]&\underset{l-2}{\bullet}\arrow[r]&\underset{l-3}{\bullet}\arrow[r]&\cdots\arrow[r]&\underset{1}{\bullet}
\end{tikzcd}\,.
$$
I will use $\chi^{\fr}$ to denote the Euler form from Definition \ref{def:eulerform} for the quiver $Q^{\fr}$ and its symmetrization will be labelled by $\chi^{\fr}_{\sym}$ without specifying the dash. This is because the rigidified version is the more natural one in this case so I will omit the dash in all notations related to $Q^{\fr}$.
\begin{lemma}
\label{lem:Vfrnondegenerate}
 Consider the framed quiver $^{\ov{n}}_{\infty}Q$ for some framing vector $\ov{n}$, then there exist natural morphisms 
 $$
 \begin{tikzcd}
 \MQ\arrow[r,"^{\ov{n}}_{\infty}\iota"]& \mM_{^{\ov{n}}_{\infty}Q}\arrow[r,"\iota^{\fr}"]&\mM_{Q^{\fr}}
 \end{tikzcd}
 $$
 inducing inclusions  of lattice vertex algebras:
 \begin{equation}
 \label{eq:inclusionsofVA}
 \begin{tikzcd}
 \widehat{H}_{\bullet}(\MQ) = V_{\bullet} \arrow[r,hookrightarrow,"^{\ov{n}}_{\infty}\iota_*"]&\widehat{H}'_{\bullet}(\mM_{^{\ov{n}}_{\infty}Q})= {}^{\ov{n}}_{\infty}V'_{\bullet}\arrow[r,hookrightarrow,"\iota^{\fr}_*"]&\widehat{H}_{\bullet}(\mM_{Q^{\fr}})=V^{\fr}_{\bullet}
 \end{tikzcd}
 \end{equation}
 and 
 $$
 \begin{tikzcd}
     \widehat{H}_{\bullet}(\MQ) = V_{\bullet} \arrow[r,hookrightarrow,"^{\ov{n}}_{\infty}\iota_*"]&\widehat{H}_{\bullet}(\mM_{^{\ov{n}}_{\infty}Q})={}^{\ov{n}}_{\infty}V_{\bullet}\,.
 \end{tikzcd}
 $$
The lattice vertex algebra $V^{\fr}_{\bullet}$ is associated to the lattice $\Lambda^{\fr}=K^0(Q^{\fr})$ with the non-degenerate pairing $\chi^{\fr}_{\sym}$. Therefore, $V^{\fr}_{\bullet}$ is a vertex operator algebra with the conformal element denoted by $\omega^{\fr}$. It is constructed by applying Theorem \ref{thm:geometricvaconstruction}, and I denote the associated Virasoro operators by $L^{\fr}_k$. 
\end{lemma} 
\begin{proof}
    The morphism $^{\ov{n}}_{\infty}\iota:\MQ\to \mM_{^{\ov{n}}_\infty Q}$ is induced by mapping each complex of representations of $Q$ to itself with $0$ at $\infty$. The map $\iota^{\fr}: \mM_{^{\ov{n}}_\infty Q}\to \mM_{Q^{\fr}}$ originates from acting on complexes of representation of $^{\ov{n}}_\infty Q$ with a functor that preserves the representation of $Q$ and constructs for each edge $\overset{(v)}{\diamond}\longrightarrow \overset{v}{\bullet}$ the morphisms $\phi^{\infty}_v:U^\bullet_{(v)}=(U^{\bullet}_{\infty})^{\oplus n_v}\to U^\bullet_{v}$ from \eqref{eq:sumofphi} applied to complexes of vector spaces. 
    That $^{\ov{n}}_{\infty}\iota_*$ leads to a morphism of vertex algebras in both cases is immediate so I only discuss the second inclusion $\iota^{\fr}_*$. Using that $\iota^{\fr}$ was induced by a linear exact functor, I only need to compare the K-theory classes 
    $$\big((\iota^{\fr})^*\Ext_{Q^{\fr}}\big)^{\textnormal{t}} = (  \Ext'_{^{\ov{n}}_{\infty}Q})^{\textnormal{t}}$$
    to apply \cite[Definition 2.11, Theorem 2.12]{GJT} with their $\sigma,\xi,F$ set as $\sigma =\iota^{\fr}$, and $\xi,F =0$.     
     Both on $\mM_{^{\ov{n}}_{\infty}Q}\times \mM_{^{\ov{n}}_{\infty}Q}$ and $\mM_{Q^{\fr}}\times \mM_{Q^{\fr}}$, I denote by $\Ext_Q$ the part of $\Ext'_{^{\ov{n}}_{\infty}Q}$ (respectively of $\Ext_{Q^{\fr}}$) stemming from $Q$. Recalling the definition from \eqref{eq:extcomplex}, one can relate the complexes via
    \begin{align*}
    \Ext'_{^{\ov{n}}_{\infty}Q} &= \Cone\Big(\Ext_Q\longrightarrow \bigoplus_{v\in V}\big(\crU_{\infty}^\vee\otimes \crU_v\big)^{\oplus n_v}\Big)\,,\\ \quad \Ext_{Q^{\fr}} &= \Cone\Big(\Ext_Q\longrightarrow \bigoplus_{v\in V}\crU_{(v)}^\vee\otimes \crU_v\Big)\,.
    \end{align*}
    Restricting $\crU_{(v)}$ along $\iota^{\fr}$ results in $(\crU^\bullet_{\infty})^{\oplus n_v}$. This identifies the two K-theory classes and concludes that $\iota^{\fr}_*$ is a morphism of vertex algebras. 

   The final statement I need to check is the non-degeneracy of $\chi^{\fr}_{\sym}$. Recall now from Definition \ref{def:eulerform} that one can write 
   $$\chi^{\fr}(-,-) = \langle -, \td(Q^{\fr})\cdot (-)\rangle\,,\qquad \td(Q^{\fr})=\pi_{\big((V)\sqcup V\big)\backslash (V)}-A^{Q^{\fr}}+ S^{Q^{\fr}}\,.$$
   Thus to show non-degeneracy of $\chi^{\fr}_{\sym}$, I need to prove that $\td(Q^{\fr})+\td(Q^{\fr})^{\intercal}$ is an invertible matrix. It has the following form
   $$
\td(Q^{\fr})+\td(Q^{\fr})^{\intercal} = \renewcommand{\arraystretch}{1.8} \begin{pmatrix}
    0&-\id_{\ZZ^{|V|}}\\
   - \id_{\ZZ^{|V|}}&\td(Q) + \td(Q)^{\intercal}
\end{pmatrix}
   $$
   which is clearly invertible.
\end{proof}
\begin{remark}
\label{rem:Lierigid}
One approach now is to work with $(V^{\fr}_\bullet,\omega^{\fr})$ while restricting to $^{\ov{n}}_{\infty}V'_{\bullet}$ or $V_{\bullet}$ when necessary. I will show in Lemma \ref{lem:sameLie} that the Lie brackets induced by $^{\ov{n}}_{\infty}V'_{\bullet}$ and $^{\ov{n}}_{\infty}V_{\bullet}$ coincide if at least one of the classes being acted on comes from $V_{\bullet}$. As Joyce's wall-crossing is phrased in terms of $^{\ov{n}}_{\infty}V_{\bullet}$, one is still able to use the results for $(V^{\fr}_\bullet,\omega^{\fr})$ under this condition like we did in \cite{BML}. In doing so, one is forced to distinguish between Virasoro constraints with and without frozen vertices which seems ad hoc because the general theory should be applicable to any nice additive category with perfect obstruction theories on its moduli.

   One can circumvent this entire issue by working with an abstract non-degenerate lattice $(\Lambda^{\nd},\chi^{\nd}_{\sym})$ containing $\ovninfty\Lambda$ and the pairing  $\chi^{\nd}_{\sym}$ restricting to $\ovninfty \chi_{\sym}$. The downside in doing so is the lack of an explicit conformal element and no knowledge about the conformal charge. To be true to one of the main goals of finding geometric interpretations of the representation theoretic structures provided by the vertex algebra, I combine both points of view. In conclusion, I work with two conformal elements -- $\omega^{\fr}$ with conformal charge $2|\Lambda|$ and the inexplicit $\omega^{\nd}$ that does not differentiate between the pair and sheaf Virasoro constraints of \cite[§2.5 and §2.6]{BML}. \end{remark}

\section{Wall-crossing of Virasoro constraints of quivers}
\label{sec:WC}
I will rephrase Virasoro constraints defined in §\ref{sec:framedw0vir}, §\ref{sec:framedvir} as the condition that the (virtual) fundamental classes of quiver-representations are physical states in the Lie algebra constructed from the geometric vertex algebra. This implies compatibility of the Virasoro constraints with wall-crossing which I use to prove the constraints in large generality by reducing to point moduli spaces where the check is trivial. In Remark \ref{rem:relatingtoGW}, I introduce some notation that is meant to draw a parallel with the Virasoro constraints on the Gromov--Witten side.
\subsection{Lie algebra of physical states}
\label{sec:Lie}
Let $V_{\bullet}$ for now be any lattice vertex algebra for a non-degenerate lattice $(\Lambda,\mQ)$. By the standard construction of Borcherds, it induces a graded Lie algebra 
$$
\widecheck{V}_{\bullet} = V_{\bullet+2}\,/\,T(V_{\bullet}),\quad [\overline{a},\overline{b}]=\overline{a_{(0)}b}
$$ where $\ov{(-)}$ this time denotes the associated class in the quotient. When working with a fixed $\alpha\in \Lambda$, I will write 
$$
e^{\alpha}\otimes \widecheck{\TT}_{\Lambda} = e^{\alpha}\otimes \TT_{\Lambda}/ T\big(e^{\alpha}\otimes \TT_{\Lambda}\big)\,.
$$
One intriguing consequence of Virasoro constraints is that Wall-crossing eventually takes place in a smaller Lie subalgebra of \textit{physical states}. One way to define it is in terms of the \textit{partial lift} of the Lie bracket that was introduced in \cite[Lemma 3.12]{BML}. As the lift is no longer a bilinear map from a single vector space to itself, it is more natural to think of it as a linear morphism:
\begin{equation}
\label{eq:hatad}
\widehat{\ad} : V_{\bullet}\longrightarrow  \Hom\big(\widecheck{V}_{\bullet}, V_{\bullet}\big)\,,\qquad \widehat{\ad}(b)(\ov{a}) = -a_{(0)}b\,.
\end{equation}
Composing $\widehat{\ad}(b)$ with the projection $\ov{(-)}: V_{\bullet+2}\to \widecheck{V}_{\bullet}$ is the usual derivation $\ad(\ov{b})$.

I now recall the two alternative ways of defining the Lie algebra of physical states.
\begin{definition}[{\cite[§5]{Borcherds}, \cite[Corollary 4.10]{Ka98}, \cite[Proposition 3.13]{BML}}]
\label{Def: physicalstates}
\sloppy Let $\big(V_{\bullet},\omega\big)$ be an operator lattice vertex algebra. The space of physical states of conformal weight $i\in \ZZ$ is defined as 
$$P_i= \big\{a\in V_{\bullet}\,\big|\, L_0(a)= i\cdot a,\ L_n(a)=0\ \textnormal{for all}\ n\geq 1\big\}\subset V_i\,.
$$
The quotient $\widecheck{P}_0 = P_1/T(P_0)$ is naturally a Lie subalgebra of $\widecheck{V}_{\bullet}$ with a well-defined restriction
$$
\widehat{\ad}: P_i \longrightarrow \Hom\big(\widecheck{P}_0,  P_i\big)\,.
 $$
\end{definition}
The following can be later used to show that an element of $\widecheck{P}_0$ satisfies the weight zero Virasoro constraints from Definition \ref{def:weight0}. It also explains where the form of the operator $\bL_{\inv}$ comes from. 
\begin{proposition}[{\cite[Definition 3.14, Lemma 3.15, Proposition 3.16]{BML}}]
\label{prop:combinedBLM}
Define  
$$
\widecheck{K}_0 = \{\ov{a}\in \widecheck{V}_0 \ | \ \widehat{\ad}(\omega)(\ov{a})=0\}\,,
$$
then the condition $\ov{a}\in \widecheck{K}_0$ can be equivalently written in terms of the Virasoro operators of $\omega$ as
$$
\sum_{j\geq -1}\frac{(-1)^j}{(j+1)!}T^{j+1}\circ L_j (a) = 0\,.
$$
Suppose additionally that $\ov{a}\in e^{\alpha}\otimes \widecheck{\TT}_{\Lambda}$ for $\alpha\neq 0$, then $\ov{a}\in \widecheck{K}_0$ if and only if $\ov{a}\in \widecheck{P}_0$. 
\end{proposition}
In the next two sections, I will relate the Lie algebra $\widecheck{P}_0$ to Virasoro constraints by showing that the dual of the $(-)^{\fr}$ version of $\bL_k$ under the pairing between homology and cohomology of $\mM_{Q^{\fr}}$ equals to $L^{\fr}_k +\delta_k$.
\subsection{Dual Virasoro algebras}
I will denote the Virasoro operators from Definition \ref{def:virop} for the quiver $Q^{\fr}$ by 
$$
\bL^{\fr}_k = \bR^{\fr}_k+\bT^{\fr}_k:\TT^{Q^{\fr}}_{\alpha}\longrightarrow \TT^{Q^{\fr}}_{\alpha}\,.
$$
As $\TT^{Q^{\fr}}_{\alpha}$ is dual to $e^{\alpha}\otimes \TT_{\Lambda^{ \fr}}$ by the pairing $\langle -,-\rangle$ in \eqref{eq:pairingTT}, one can construct the duals 
$$
(\bL^{\fr}_k)^{\dagger} = (\bR^{\fr}_k)^{\dagger}+(\bT^{\fr}_k)^{\dagger}: e^{\alpha}\otimes \TT_{\Lambda^{\fr}} \longrightarrow e^{\alpha}\otimes \TT_{\Lambda^{\fr}}
$$
which I collect into operators on $V^{\fr}_{\bullet} = \CC[\Lambda^{\fr}]\otimes \TT_{\Lambda^{\fr}}$. Together with $L^{\fr}_k$, there are now two potentially different sets of operators satisfying Virasoro commutation relations. The point of this subsection is to show that they are related and that their restrictions to $Q$ are identical. 

One of the main results of our previous work -- \cite[Theorem 4.12]{BML} -- stated the analog of the next theorem for sheaves. For this reason, I only discuss the adaptations needed for quivers and refer the reader to \cite{BML} for the part of the proof that remains the same.
\begin{theorem}
\label{thm:twovirasoro}
For each $k\geq -1$, the operators $(\bL^{\fr}_k)^{\dagger}$ and $L^{\fr}_k$ on $\CC[\Lambda^{\fr}]\otimes\TT_{\Lambda^{\fr}}$ are related by
$$
(\bL^{\fr}_k)^{\dagger} = L^{\fr}_k+\delta_k\,.
$$ 
Denoting by $(\bL_k)^{\dagger}$ the restriction of $(\bL^{\fr}_k)^{\dagger}$ to $\CC[\Lambda]\otimes \TT_{\Lambda}$ induced by the inclusion $\Lambda \hookrightarrow \Lambda^{\fr}$, the above result implies
$$
(\bL_k)^{\dagger} = L_k\,.
$$
\end{theorem}
\begin{proof}
    I mostly sketch the proof while referring to the proof of \cite[Theorem 4.12]{BML} for a more detailed account. Applying \eqref{eq:Yaz} together with $\eqref{Eq: reconstruction}$ to \eqref{eq:omega} in the case of the quiver $Q^{\fr}$ with vertices $v^{\fr}\in V^{\fr}$, I obtain
    \begin{align*}
    Y(\omega^{\fr}, z)=\frac{1}{2}\sum_{v^{\fr}\in V^{\fr}}\bigg( \sum_{\begin{subarray}{c} i,j\geq 1\end{subarray}}\hat{v}^{\fr}_{(-i)}v^{\fr}_{(-j)}z^{i+j-2} +&\sum_{\begin{subarray}{c} i\geq 1,\\j\geq 0\end{subarray}}\big(\hat{v}^{\fr}_{(-i)}v^{\fr}_j + v^{\fr}_{(-i)}\hat{v}^{\fr}_{(j)}z^{i-j-2}\big)\\+&\sum_{i,j\geq 0 }\hat{v}^{\fr}_{(i)}v^{\fr}_{(j)}z^{-(i+j)-2}\bigg)\,.
    \end{align*}
The terms containing powers of $z$ less than or equal to $-1$ can be summed up to 
$$
\sum_{k\geq -1}(R^{\fr}_k + T^{\fr}_k)z^{-k-2}
$$
where 
$$R_k = \frac{1}{2}\sum_{v^{\fr}\in V^{\fr}}\sum_{\begin{subarray}{c} j-i=k,\\ i\geq 1,j\geq 0\end{subarray}}\big(\hat{v}^{\fr}_{(-i)}v^{\fr}_{(j)}+v^{\fr}_{(-i)}\hat{v}^{\fr}_{(j)}\big)\,,\qquad  T_k = \frac{1}{2}\sum_{v^{\fr}\in V^{\fr}}\sum_{\begin{subarray}{c} i+j=k,\\i,j\geq 0\end{subarray}}\hat{v}^{\fr}_{(i)}v^{\fr}_{(j)}\,.$$
The proof that $$(\bR^{\fr}_k)^{\dagger} = R^{\fr}_k$$ for $k\geq 0$ is identical to the one found in the second half of the proof of \cite[Theorem 4.12]{BML} after replacing $\ch^H_k(\gamma^{\pa})$ with $\tau_k(v^{\fr})$. In the case $k=-1$, one could also use the geometric argument given in \cite[Lemma 4.9]{BML}. In the end, I only need to prove that 
$$
(\bT^{\fr}_k)^{\dagger} = T^{\fr}_k\quad \textnormal{for}\quad k\geq -1\,.
$$
From \eqref{eq:fieldastau}, it follows that
$$
T^{\fr}_k = \frac{1}{2}\sum_{v^{\fr},w^{\fr}\in V^{\fr}}\sum_{\begin{subarray}{c} i+j=k,\\i,j\geq 0\end{subarray}}i!j!\chi^{\fr}_{\sym}(v^{\fr},w^{\fr})\tau_i(v^{\fr})\tau_j(w^{\fr})\,.
$$
Note that 
\begin{align*}
\sum_{v^{\fr},w^{\fr}\in V^{\fr}}\chi^{\fr}_{\sym}(v^{\fr},&w^{\fr}) \ v^{\fr}\boxtimes w^{\fr} \\
&= \sum_{v^{\fr},w^{\fr}\in V^{\fr}} \ \big\langle v^{\fr}, \big(\td(Q^{\fr}) +\td(Q^{\fr})^{\intercal}\big)\cdot w^{\fr}\big\rangle \ v^{\fr}\boxtimes w^{\fr}\\
&=\sum_{v^{\fr}\in V^{\fr}}\td(Q^{\fr})\cdot v^{\fr}\boxtimes v^{\fr} + \sum_{v^{\fr}\in V^{\fr}} v^{\fr}\boxtimes \td(Q^{\fr})\cdot v^{\fr} \,,
\end{align*}
so comparing with the definition of $\tau_i\tau_j\big(\Delta_*\td(Q^{\fr})\big)$ from Example \ref{ex:Tvir} and swapping the order of the $\tau$ factors once, I conclude that 
$$
   T^{\fr}_k +\delta_k =  \sum_{\begin{subarray}{c} i+j=k \\ i,j \geq 0 \end{subarray}}i!j!\tau_i\tau_j\big(\Delta_*\td(Q^{\fr})\big) +\delta_k =   (\bT^{\fr}_k)^{\dagger}\,.
$$
When restricting to $Q$ without frozen vertices, the $\delta_k$ term does not appear in $\bT_k$ so the second claim follows immediately from the above result. 
\end{proof}
This observation allowed us to view Virasoro constraints as a vanishing condition on the homology classes $[M^{\mu}_{Q,\ov{d}}]$ rather than the integrals with respect to it. The consequent connection to $\widecheck{P}_0$ will be recalled in the next subsection.
\subsection{Geometric interpretation of the Lie algebra of physical states}
\label{sec:geominterpLie}

I will denote a generic element of $\Lambda^{\fr}$ by $\alpha$, specifying that $\alpha=(\ov{c},\ov{d})$ only when necessary.
Consider the $[*/\GG_m]$-action $\rho$ on $\mM_{Q^{\fr}}$ and construct the rigidification $$\mM^{\rig}_{Q^{\fr}} =\mM_{Q^{\fr}}\mkern-7mu\fatslash [*/\GG_m]\,.$$
I also denote by $\mM^{\rig}_{\alpha}$ the rigidification of each of the connected components $\mM_{\alpha}$.

As a consequence of \cite[Proposition 3.24 a), b)]{Jo17} and the non-degeneracy of $\chi^{\fr}$, one can conclude that 
$$
H_{\bullet}(\mM^{\rig}_{\alpha})\cong H_{\bullet}(\mM_{\alpha})/T \big(H_{\bullet -2}(\mM_{\alpha})\big)
$$
whenever $\alpha\neq 0$. After adding the shift $$\widecheck{H}_{\bullet}(\mM^{\rig}_{\alpha}) = H_{\bullet+2-2\chi^{\fr}(\alpha,\alpha)}(\mM^{\rig}_{\alpha})\,,$$
one can collect the above isomorphisms into 
\begin{equation}
\label{eq:VcheckHcheck}
\bigoplus_{\alpha\in\Lambda^{\fr}\backslash \{0\}}\widecheck{H}_{\bullet}(\mM^{\rig}_{\alpha}) \cong \widecheck{V}_{\bullet} \backslash e^{0}\otimes\widecheck{\TT}_{\Lambda^{\fr}}\,.
\end{equation}
I implicitly used that the translation operator $T$ from \eqref{eq:Translation} is identified with the abstract one on $\CC[\Lambda^{\fr}]\otimes \TT_{\Lambda^{\fr}}$ from \eqref{eq:Tabstract}.

Before I state the main observations for the quiver $Q^{\fr}$, I compare the Lie brackets of $\ovninfty \widecheck{V}_{\bullet}$ and $\ovninfty \widecheck{V}'_{\bullet}$ as was already promised in Remark \ref{rem:rigidvsnonrigid}.
\begin{lemma}
    \label{lem:sameLie}
    Suppose that $a\in \ovninfty V_{\bullet}$ and $\ov{b}\in\widecheck{V}_{\bullet}\subset \ovninfty \widecheck{V}_{\bullet}$, and denote by $\widehat{\ad},\widehat{\ad}'$ the partial lifts of the Lie brackets on $\ovninfty \widecheck{V}_{\bullet}$ and $\ovninfty \widecheck{V}'_{\bullet}$ respectively, then 
    $$
    \widehat{\ad}'(a)(\ov{b}) =  \widehat{\ad}(a)(\ov{b})  \,.
    $$
\end{lemma}
\begin{proof}
    From the assumptions, it follows that $\ov{b}\in e^{\ov{d}}\otimes \TT_{\Lambda}\subset e^{(0,\ov{d})}\otimes \TT_{\ovninfty\Lambda}$ which implies that there is a lift $b$ such that
    $$
    b\in H_{\bullet}(\mM_{\ov{d}}) \subset H_{\bullet}(\mM_{(0,\ov{d})})
    $$
where I used the inclusion $\mM_{\ov{d}}\hookrightarrow \mM_{(0,\ov{d})}$.
     From \eqref{eq:extcomplex} and \eqref{eq:ext}, one sees that $$\Ext_{\ovninfty Q}|_{\mM_{\Qfr}\times \mM_{\ov{d}}} = \Ext'_{\ovninfty Q}|_{\mM_{\Qfr}\times \mM_{\ov{d}}}\,.$$
    Together with \eqref{eq: joycefields} and the comparison of pairings from \eqref{eq:comparingforms}, this implies that 
    $$
    Y(a,z)b = Y'(a,z)b\,.
    $$
\end{proof}
This is all that I require to be able to work directly with $V^{\fr}_{\bullet}$ just like we did in \cite{BML}. This time, however, I want to accommodate both points of view discussed in Remark \ref{rem:Lierigid} which explains the previously mysteriously distinguished pair Virasoro constraints.

 By the arguments in the second half of the proof of \cite[Theorem 4.12]{BML}, there is the identification
    \begin{equation}
    \label{eq:T=Rdual}
    (\bR^{\fr}_{-1})^{\dagger} = T: e^{\alpha}\otimes \TT_{\Lambda^{ \fr}}\longrightarrow e^{\alpha}\otimes \TT_{\Lambda^{ \fr}}\,,
    \end{equation}
    for any $\alpha\in \Lambda^{\fr}$. In particular, the paring descends to
    $$
    \langle -,-\rangle : \TT^{Q^{\fr}}_{\inv,\alpha}\times e^{\alpha}\otimes \widecheck{\TT}_{\Lambda^{\fr}}\to \CC
    $$
    which makes $\TT^{Q^{\fr}}_{\inv, \alpha}$ into the dual of  $ \widecheck{\TT}_{\Lambda^{\fr}}$. Lastly, the  weight zero operator from Definition \ref{def:weight0} constructed out of $\bL^{\fr}_k$ will be denoted by $\bL^{\fr}_{\inv}$.
\begin{definition}
\label{def:abstractvir}
  When working in the geometric setting using the isomorphisms \eqref{eq:xi} and \eqref{eq:xidagger}, I will write 
    $$
    \int_a \tau = \langle \tau, a \rangle \,,\quad \textnormal{whenever} \ a\in e^{\alpha}\otimes \TT_{\Lambda^{\fr}}\,, \tau\in \TT^{Q^{\fr}}_{\alpha}\,.  
    $$
    If $\ov{a}\in e^{\alpha}\otimes \widecheck{\TT}_{\Lambda^{\fr}}$ instead, then the following is still well-defined:
    $$
    \int_{\ov{a}}\tau = \langle \tau,\ov{a}\rangle \,, \quad \textnormal{whenever} \ \tau \in \TT^{Q^{\fr}}_{\inv, \alpha}\,.
    $$
 This way, one can define Virasoro constraints for any $\ov{a}\in e^{\alpha}\otimes \widecheck{\TT}_{\Lambda^{\fr}}$ as the condition
    $$
    \int_{\ov{a}}  \bL^{\fr}_{\inv}(\tau) = 0\,, \quad \textnormal{for any} \ \tau\in \TT^{Q^{\fr}}_{\alpha}\,.
    $$
\end{definition}
Next, I complete the alternative point of view presented in remarks \ref{rem:rigidvsnonrigid} and \ref{rem:Lierigid}. If I continued working with the corrected vertex algebra $\ovninfty V_{\bullet}$ without introducing $Q^{\fr}$, I could instead use $(\Lambda^{\nd},\chi^{\nd}_{\sym})$. Denote by $V^{\nd}_{\bullet}$ the corresponding lattice vertex operator algebra with the conformal element $\omega^{\nd}$ and the resulting Virasoro operators $L^{\nd}_k$.  There are then two Virasoro algebras acting on the vector space $\ovninfty V_{\bullet}$:
\begin{enumerate}
    \item one generated by the operators $\ovninfty L'_k$ which are the restrictions of $L^{\fr}_k$ under the inclusion $\ovninfty V'_{\bullet}\hookrightarrow V^{\fr}_{\bullet}$,
    \item another one generated by $\ovninfty L_k$ obtained by restricting $L^{\nd}_k$ under the inclusion $\ovninfty V_{\bullet}\hookrightarrow V^{\nd}_{\bullet}$.
\end{enumerate}
 To distinguish between the two cases, I will write $\ovninfty P'_i$ and $\ovninfty P_i$ to denote the two subspaces of physical states of conformal weight $i$ defined in terms of $\ovninfty L'_k$ and $\ovninfty L_k$ respectively.
 
Applying Definition \ref{def:virop} to $\Qfr$, one constructs the operators ${}^{\ov{n}}_{\infty}\bL_k$, $\ovninfty \bL_{\inv}$ acting on $\TT^{\Qfr}_{(d_{\infty},\ov{d})}$. They are compatible with $\bL^{\fr}_k, \bL^{\fr}_{\inv}$ under the pullback along $\iota^{\fr}: \mM_{\Qfr}\to \mM_{Q^{\fr}}$.

The next lemma elucidates where the additional $\delta_k$ term in $\ovninfty \bT_k$ comes from and relates all the operators present on $\ovninfty V^{\bullet}$.

\begin{lemma}
\label{lem:dashnodash}
    Suppose that $a\in e^{(1,\ov{d})}\otimes \TT_{\ovninfty \Lambda}$ satisfies $\tau_i(\infty)\cap a = 0$ for $i> 0$, then 
\begin{align*}
\label{eq:comparingLk}
\ovninfty L_k(a) &= \big(\ovninfty L'_k+\delta_k\big)(a)\,,\\
(\ovninfty \bL_k)^{\dagger}(a) &= \ovninfty L_k(a)\quad \textnormal{for} \quad k\geq -1\,.
\numberthis
\end{align*}
The above implies for any $i\in \ZZ$ that 
\begin{equation}
\label{eq:comparingPi}
a\in P_{i+1} \quad \iff \quad a\in P'_i\,.
\end{equation}
\end{lemma}
\begin{proof}
Applying the arguments in the second half of the proof of \cite[Theorem 4.12]{BML}, one can show that the actions of  $\ovninfty R_k$ and $\ovninfty R'_k$ on polynomials in $v_{k}$ for $v\in \ovninfty V$ are equal. Thus I only need to compare the two different $T_k$ operators. The same computation as in the proof of Theorem \ref{thm:twovirasoro} leads to 
$$
T^{\nd}_k = \frac{1}{2}\sum_{v^{\nd},w^{\nd}\in V^{\nd}}\sum_{\begin{subarray}{c} i+j=k,\\i,j\geq 0\end{subarray}}i!j!\chi^{\nd}_{\sym}(v^{\nd},w^{\nd})\tau_i(v^{\nd})\tau_j(w^{\nd})
$$
where $V^{\nd}\subset \Lambda^{\nd}\otimes_{\ZZ}\CC$ is a complex basis containing $\ovninfty V$. The first equality of \eqref{eq:comparingLk} then follows from setting $\tau_l(v^{\nd}) = 0$ for all $v^{\nd}\notin \ovninfty V$ and $l\geq 0$ while also using the assumption on $a$. Because of \eqref{eq:comparingLk} and Theorem \ref{thm:twovirasoro}, one sees that the second equality also holds.   
\end{proof}
In particular, it is the vertex algebra $\ovninfty V_{\bullet}$ that produces the correct Virasoro algebra dual to the one from Definition \ref{def:virop}. Note also that the restrictions of $\ovninfty L_k$ and $\ovninfty L'_k$ to $V_\bullet$ are equal.

I will henceforth work mainly with $\ovninfty L_k, \ovninfty \bL_k$ and the associated weight zero operator 
$\ovninfty\bL_{\inv}
$. Simultaneously, I keep the alternative approach using $\omega^{\fr}$  in mind to give a more geometric point of view. To do so, I use the dashed notation as explained in \eqref{eq:dashednot}.

Note that part (ii) has already been proved in \cite[Lemma 5.11 b)]{BML}, while part (iii) generalizes the observation from \cite[Lemma 5.11 a)]{BML}.
\begin{proposition}
\label{cor:PMimpliesadPM}
\begin{enumerate}[label=\roman*)]
Suppose that $a\in e^{\ov{c}_{\infty}}\otimes \TT_{\ovninfty \Lambda}$ and $\ov{b}\in e^{\ov{d}_{\infty}}\otimes \widecheck{\TT}_{\ovninfty \Lambda}$.
    \item If they satisfy 
   $$
    \int_{a} \ovninfty\bL_{\inv}(\tau_1) = 0 =  \int_{\ov{b}} \ovninfty\bL_{\inv}(\tau_2) \quad \textnormal{for any}\quad \tau_1\in \TT^{\Qfr}_{\ov{c}_{\infty}},\tau_2\in \TT^{\Qfr}_{\ov{d}_{\infty}}\,,
   $$
   then 
  $$
     \int_{\widehat{\ad}(a)(\ov{b})}  \ovninfty\bL_{\inv}(\tau) =0\quad \textnormal{for any} \quad \tau\in \TT^{\Qfr}_{\ov{c}_{\infty}+\ov{d}_{\infty}}\,.
$$
\item If $\ov{a} = \ov{b}$ for a different lift $b\in \ovninfty V_{\bullet}$ and both satisfy $$\tau_1(\infty)\cap a = 0 = \tau_1(\infty)\cap b\,,$$
    then $a = b$.
    \item If $\tau_k(\infty)\cap a= 0$ holds for a fixed $k\geq 1$ then 
    $$
    \tau_k(\infty)\cap \widehat{\ad}(a)(\ov{b}) = 0 \quad \textnormal{whenever} \quad \ov{b}\in \widecheck{V}_{\bullet}\subset {}^{\ov{n}}_{\infty}\widecheck{V}_{\bullet}\,.
    $$
    \item If additionally  $a\in e^{(1,\ov{c})}\otimes \TT_{\ovninfty \Lambda}$ satisfies $\tau_i(\infty)\cap a = 0$ for all $i>0$, and $\ov{b}$ lies in $e^{\ov{d}}\otimes \widecheck{\TT}_{\Lambda}\subset e^{(0,\ov{d})}\otimes \widecheck{\TT}_{\ovninfty \Lambda}$, and
     \begin{align*}
    \int_{a} \ovninfty\bL'_k(\tau_1) &= 0 \quad \textnormal{for any}\quad \tau_1\in \TT^{\Qfr}_{(1,\ov{c})},  k\geq 0\,,\\  \int_{\ov{b}} \bL_{\inv}(\tau_2)&=0\quad \textnormal{for any} \quad \tau_2\in \TT^{Q}_{\ov{d}}\,,
   \end{align*}
   then 
   $$
     \int_{\widehat{\ad}(a)(\ov{b})} \ovninfty\bL'_k(\tau)  =0\quad \textnormal{for any} \quad \tau\in \TT^{\Qfr}_{(1,\ov{c}+\ov{d})}, k\geq 0\,.
$$
\end{enumerate}
  
\end{proposition}
\begin{proof}
   (i) I will illustrate how the previously discussed results can be combined to prove this. 
  By \eqref{eq:comparingLk} together with dualizing the formula for the weight zero operator from \ref{def:weight0}, one obtains
    $$
    \sum_{j\geq -1}\frac{(-1)^j}{(j+1)!}T^{j+1}\circ L^{\nd}_j (c) = 0 \quad  \textnormal{for} \quad c= a,b \,.
    $$
By Proposition \ref{prop:combinedBLM}, this is equivalent to $\ov{a},\ov{b}\in \widecheck{K}_0$ with respect to $\omega^{\nd}$ so it implies
    $$
    \big[\ov{a},\ov{b}\big]\in \widecheck{K}_0\,.
    $$
   Therefore, the lift $\widehat{\ad}(a)(\ov{b})$ of this Lie bracket also satisfies Virasoro constraints. 
   
   (ii) The proof is identical to the one in \cite[Lemma 5.11 b)]{BML}. 
   
    (iii) Here, one need to work with $V^{\fr}_{\bullet}$ or $V^{\nd}_{\bullet}$ instead of ${}^{\ov{n}}_{\infty}V_{\bullet}$. To be more explicit, I choose to phrase everything in terms of $Q^{\fr}$. Consider the image of $\infty$ in $\Lambda^{\fr}_{\CC}$ under the map ${}^{\ov{n}}_{\infty}\Lambda_{\CC}\to \Lambda^{\fr}_{\CC}$ induced by Lemma \ref{lem:Vfrnondegenerate} and denote it also by $\infty$. 
    Then because $\chi^{\fr}$ is non-degenerate, there is a $u\in \Lambda^{\fr}_{\CC}$ such that $\chi^{\fr}(u,x) = \langle \infty, x\rangle$ for any other $x\in \Lambda^{\fr}_{\CC}$. This implies $u_{(k)}=k!\tau_k(\infty)\cap$ by \eqref{eq:fieldastau} so I just need to show that 
    $$
   u_{(k)}(b_{(0)}a) = b_{(0)}(u_{(k)}a) - (b_{(0)}u)_{(k)}a = 0
    $$
    where the first equality uses \eqref{eq:jacobi}. The last step uses that $u_{(k)}a =0$ by assumption and $b_{(0)}u =0$ as a consequence of \eqref{eq:skew} together with $b\in V_{\bullet}$. 
    
   (iv) Note that the proof of \cite[Proposition 2.16]{BML} comparing Virasoro constraints from Definition \ref{def:Skvirasoro} and Definition \ref{def:weight0} is formal and only requires the vanishing of $\tau_1(\infty)\cap$  acting on the (virtual) fundamental class. By the assumption on $a$ and the definition of $\bS^{v}_k$, I know that 
$$\int_{a}\ovninfty\bS^{\infty}_k(\tau_1) =-\delta_k\int_{a}\tau_1$$
   which implies that the vanishing of the integrals in the assumption for $a$ is equivalent to
$$
 \int_{a} \big(\ovninfty\bL_k+\ovninfty S^{\infty}_k)(\tau_1) = 0 \quad \text{for} \quad \tau_1\in \TT^{\Qfr}_{(1,\ov{c})}, k\geq 0\,.
$$
This, in turn, is equivalent to
   $$ \int_{a} \ovninfty\bL_{\inv}(\tau_1) = 0\quad \textnormal{for} \quad \tau_1\in \TT^{\Qfr}_{(1,\ov{c})}\,.$$
Applying (i), I conclude that $\widehat{\ad}(a)(\ov{b})$ satisfies weight zero Virasoro constraints so by (iii) and reversing the argument I just presented for $a$, the claim is proved.
   
   Alternatively, the vanishing of the integrals for $a$ is equivalent to $a\in P'_0$. Since $\ov{b}$ is an element of both $\widecheck{P}_0$ and $\widecheck{P}'_0$, one concludes that
$$
\widehat{\ad}(a)(\ov{b})\in P'_0
$$
which again shows that the claim holds. \end{proof}
The final statement of the proposition mimics the original formulation of the compatibility of pair Virasoro constraints under wall-crossing shown in \cite{BML}.
\subsection{Reducing Virasoro constraints to point moduli spaces}
\label{sec:reductiontop}
I now give a brief review of wall-crossing for quiver representations in the form it was proved in \cite[§6]{Jo17} and \cite[§5, §6]{GJT}. I will continue working with $Q$ that has no frozen vertices and whose path algebra $\mA_Q$ is finite-dimensional.  I will cover the cases with $F\neq \emptyset$ by using Definition \ref{def:onefrozen} to reduce them to working with $\ovninfty Q$.
To make formulae less cluttered with symbols, I will not specify $Q$ in the subscript of $M^\mu_{Q,\ov{d}}$.

Fix a vector of slope parameters $\ov{\mu}\in \RR^{Q}$ and assume that there are no strictly $\mu$-semistables with dimension vector $\ov{d}$, then there are open embeddings
\begin{equation}
\label{eq:openembed}
\begin{tikzcd}
&\MQ^{\rig}\\
M^{\mu}_{\ov{d}}\arrow[r,hookrightarrow]\arrow[ru,hookrightarrow]&\Mf^{\rig}_Q\arrow[u,hookrightarrow]
\end{tikzcd}\,.
\end{equation}
 Starting from the (virtual) fundamental classes $[M^{\mu}_{\ov{d}}]\in H_{-\chi(\ov{d},\ov{d})}(M^{\mu}_{\ov{d}})$, I denote both of their pushforwards along the above open embeddings by 
\begin{equation}
\label{eq:barMclass}
\ov{[M^{\mu}_{\ov{d}}]}\in \widecheck{H}_{0}(\Mf^{\rig}_{Q}), \widecheck{H}_{0}(\MQ^{\rig})\,.\footnote{The definition of $\widecheck{H}_{\bullet}(\Mf^{\rig}_{Q})$ is in terms of the same shift on connected components as for $\widecheck{H}_{\bullet}(\MQ^{\rig})$.}
\end{equation}
The isomorphism \eqref{eq:VcheckHcheck} leads then to a class $\ov{[M^{\mu}_{\ov{d}}]}\in e^{\ov{d}}\otimes \widecheck{\TT}_{Q}$. 

\begin{remark}
\label{rem:relatingtoGW}
  Instead of using $v_{-k}$ defined by \eqref{eq:tauder}, I will change for the moment to the variables $t_{k,v}$ which satisfy 
 $$
 t_{k,v} = (k-1)!v_{-k}\,.
 $$
If I choose a lift $[M^{\mu}_{\ov{d}}]\in V_{\bullet}$ representing $\ov{[M^{\mu}_{\ov{d}}]}$, as can be done for example by a choice of a universal representation $(\ov{\UU},\ov{\Ff})$, then I can phrase the data of $[M^{\mu}_{\ov{d}}]$ in a form reminiscent of the Gromov-Witten superpotential. Using the notation
$$
  \big\langle \tau_{k_1}(v_1)\cdots \tau_{k_n}(v_n)\big\rangle^{\mu}_{\ov{d}} = \int_{[M^{\mu}_{\ov{d}}]} \tau_{k_1}(v_1)\cdots \tau_{k_n}(v_n)\,,
$$
one may expand the class as 
$$
 [M^{\mu}_{\ov{d}}] = \Big\langle \exp \Big[\sum_{k,v} \tau_k(v)t_{k,v}\Big]\Big\rangle^{\mu}_{\ov{d}}\,.
$$
The Virasoro operators from Theorem \ref{thm:twovirasoro} once restricted from $V^{\fr}_{\bullet}$ to $V_{\bullet}$ become differential operators $L_k$ quadratic in the derivatives with respect to the variables $t_{k,v}$. The condition for Virasoro constraints from \ref{def:weight0} or \ref{def:Skvirasoro}  to hold is equivalent to the existence of a lift $[M^{\mu}_{\ov{d}}]$ such that
$$
L_k \ [M^{\mu}_{\ov{d}}] = \delta_{k} \ [M^{\mu}_{\ov{d}}]\quad\textnormal{for}\quad k\geq 0\,.
$$
I mentioned this formulation to draw a parallel with the standard Virasoro constraints for Gromov-Witten potentials as conjectured by Eguchi--Hori--Xiong  \cite{ehx} and Katz (noted down in \cite{EJX}, See also the review article of Getzler \cite{getzler}). It is a recommendable exercise to restate \cite{BML} in this language as I have done in \cite{mytalk}.
\end{remark}

When working with $\Qfr$, one additionally needs to specify $\mu_{\infty}$ to get a stability condition out of $$\ovninfty\mu=(\mu_{\infty},\ov{\mu})\,.$$ 
I then work with the moduli schemes $M^{^{\ov{n}}_\infty \mu}_{(1,\ov{d})}$ when there are no strictly semistables. In this case, the diagram of open embeddings \eqref{eq:openembed} can be lifted to 
\begin{equation}
\label{eq:liftinc}
\begin{tikzcd}
&\mM_{\Qfr}\\
M^{\mufr}_{(1,\ov{d})}\arrow[r,hookrightarrow]\arrow[ru,hookrightarrow]&\Mf_{\Qfr}\arrow[u,hookrightarrow]
\end{tikzcd}\,,
\end{equation}
so the (virtual) fundamental class $\Big[M^{\mufr}_{(1,\ov{d})}\Big]\in H_{\bullet}\Big(M^{\mufr}_{(1,\ov{d})}\Big)$ can be pushed forward to (note that the absence of the dash is important for the specified degree to be correct)
\begin{equation}
\label{eq:Mfrclass}
\Big[M^{\mufr}_{(1,\ov{d})}\Big]\in\widehat{H}_2(\Mf_{\Qfr}),\widehat{H}_2(\mM_{\Qfr})= {}^{\ov{n}}_{\infty}V_2\,.
\end{equation}
From the sequence of inclusions of vertex algebras in \eqref{eq:inclusionsofVA} compatible with translation operators, one can conclude the inclusions of the associated Lie algebras
$$
\begin{tikzcd}
    \widecheck{V}_{\bullet} \arrow[r,hookrightarrow]& ^{\ov{n}}_{\infty}\widecheck{V}'_{\bullet}\arrow[r,hookrightarrow]&\widecheck{V}^{\fr}_{\bullet}
\end{tikzcd}\,,\qquad \begin{tikzcd}
    \widecheck{V}_{\bullet} \arrow[r,hookrightarrow]& ^{\ov{n}}_{\infty}\widecheck{V}_{\bullet}
    \end{tikzcd} \,.
$$
Therefore, I may also work with the classes $\ov{[M^{\mu}_{\ov{d}}]}$ and $\ov{\Big[M^{\mufr}_{(1,\ov{d})}\Big]}$ as elements of $^{\ov{n}}_{\infty}\widecheck{V}_{\bullet}$ or $\widecheck{V}^{\fr}_{\bullet}$. 

Wall-crossing relates different classes $\ov{[M^{\mu}_{\ov{d}}]}$ for different choices of $\mu$-stability but can be made more general if need be as will be done in the case of Bridgeland stability conditions in §\ref{sec:P2}. For $\mu$-stability, acyclic quivers without relations and with no frozen vertices appeared in \cite[§5, §6]{GJT}, while \cite[§6]{Jo17} included non-trivial relations. To cover acyclic quivers with frozen vertices, I use Definition \ref{def:onefrozen} to reduce to a fixed quiver of the form $\Qfr$ and its dimension vectors $(d_{\infty},\ov{d})$. As it is enough to consider $d_{\infty}\leq 1$, I can use \cite[§6]{Jo17} directly because
\begin{enumerate}
    \item if $d_{\infty}=0$, then everything including wall-crossing reduces to working with $Q$ without frozen vertices,
    \item if $d_{\infty}=1$, then only classes of the form $(c_{\infty}, \ov{c})$ with $c_{\infty}\leq 1$ contribute to wall-crossing. The long list of assumptions in \cite[§5.1,§5.2]{Jo17} is satisfied by \cite[§6.4]{Jo17}. The case $c_{\infty}=0$ is treated using the previous point, and when $c_\infty=1$, there is no difference between freezing the vertex $\infty$ and rigidifying the moduli problem. 
\end{enumerate}
 To include $Q$ with finite dimensional $\mA_Q$ but containing cycles, I will only note that by the discussion at the beginning of §\ref{sec:semrepmod} and the arguments in \cite[§6.4.3]{Jo21}, one still retains projectivity of all moduli spaces that are necessary for the proof of wall-crossing. 

In the theorem below, the second statement of the theorem is phrased in terms of the operation $\widehat{\ad}$ from \eqref{eq:hatad} providing a partial lift of Joyce's wall-crossing formula. Note that unlike \cite{BML}, this is not necessary, as I could use Proposition \ref{cor:PMimpliesadPM} i) instead of iv). However, it does give strictly more information about the wall-crossing behavior of the classes.
\begin{theorem}[{\cite[Theorem 5.9, Theorem 6.16]{Jo17}}]
\label{thm:WCformulae}
Let $\ov{\mu},\ov{\nu}\in \RR^{V}$ be two different vectors of parameters used to define stabilities as in \eqref{eq:mustability} for the quiver $Q$, and $\mu_{\infty},\nu_{\infty}\in \RR$, then 
\begin{enumerate}
    \item  for each  $\ov{d}\in =\Lambda_+$, there exist uniquely defined classes $$\ov{[M^{\mu}_{\ov{d}}]}, \ov{[M^{\nu}_{\ov{d}}]}\in e^{\ov{d}}\otimes \widecheck{\TT}_{\Lambda}$$
    identified with the natural ones from \eqref{eq:barMclass} when there are no strictly semistables.
    They are related by the wall-crossing formula
    \begin{align*}
         \label{eq:quiverWC}
    \ov{[M^{\mu}_{\ov{d}}]} =\sum_{\begin{subarray}{c} (\ov{d}_i\in \Lambda_+)_{i=1}^l:\\ \sum_{i=1}^l\ov{d}_i=\ov{d}\end{subarray}}\tilde{U}(&\ov{d}_1,\ldots,\ov{d}_l;\nu,\mu) \\ &\qquad\Big[\cdots \Big[\Big[\ov{[M^{\nu}_{\ov{d}_1}]},\ov{[M^{\nu}_{\ov{d}_2}]}\Big],\ov{[M^{\nu}_{\ov{d}_3}]}\Big]\cdots,\ov{[M^{\nu}_{\ov{d}_l}]}\Big]
    \numberthis
    \end{align*}
  in $\widecheck{V}_{\bullet}$. Here $\tilde{U}(\ov{d}_1,\ldots,\ov{d}_l;\nu,\mu)$ are the coefficients defined in \cite[§3.2]{Jo21}.
    \item for each $\ov{d}\in (\NN_0)^V$, there exist uniquely defined classes $$\Big[M^{\mufr}_{(1,\ov{d})}\Big], \Big[M^{\nufr}_{(1,\ov{d})}\Big]\in e^{(1,\ov{d})}\otimes \TT_{\ovninfty \Lambda}\,,$$ that additionally satisfy
    $$
    \tau_1(\infty)\cap \Big[M^{\mufr}_{(1,\ov{d})}\Big] = 0 =  \tau_1(\infty)\cap \Big[M^{\nufr}_{(1,\ov{d})}\Big]\,.
    $$
      They are identified with the natural ones from \eqref{eq:Mfrclass} when there are no strictly semistables and fit into the following refined wall-crossing formula
    \begin{align*}
    \label{eq:pairWC}
    \Big[M^{\mufr}_{(1,\ov{d})}\Big]= 
    &\sum_{\begin{subarray}{c} \ov{d}_0\in (\NN_0)^V,(\ov{d}_i\in \Lambda_+)_{i=1}^l:\\ \sum_{i=0}^l\ov{d}_i=\ov{d}\end{subarray}} \widehat{U}((1,\ov{d}_0),\ov{d}_1,\ldots,\ov{d}_l;{}\nufr,{}\mufr)\\
    &\qquad \widehat{\ad}\Big(\cdots \widehat{\ad}\Big(\widehat{\ad}\Big(\Big[M^{\nufr}_{(1,\ov{d}_0)}\Big]\Big)\big(\ov{[M^{\nu}_{\ov{d}_1}]}\big)\Big)\big(\ov{[M^{\nu}_{\ov{d}_2}]}\big)\cdots\Big)\big(\ov{[M^{\nu}_{\ov{d}_l}]}\big)\,.
    \numberthis
\end{align*}
   in $^{\ov{n}}_{\infty}V_{\bullet}$. The coefficients $\widehat{U}((1,\ov{d}_0),\ov{d}_1,\ldots,\ov{d}_l;{}\nufr,{}\mufr)$ are obtained from \cite[§3.2]{Jo21} by reordering the entries in the (partially lifted) iterated Lie bracket so that $(1,\ov{d}_0)$ always appears first. 
\end{enumerate}
\end{theorem}
\begin{proof}
The first statement is the content of \cite[Theorem 5.9, Theorem 6.16]{Jo21}. Consider the situation in (2), then these theorems imply that
\begin{align*}
\label{eq:nonliftedWC}
 \ov{\Big[M^{\mufr}_{(1,\ov{d})}\Big]} =\sum_{\begin{subarray}{c} \ov{d}_0\in (\NN_0)^V,(\ov{d}_i\in \Lambda_+)_{i=1}^l:\\ \sum_{i=0}^l\ov{d}_i=\ov{d}\end{subarray}}& \widehat{U}((1,\ov{d}_0),\ov{d}_1,\ldots,\ov{d}_l;{}\nufr,{}\mufr)\\ & \Big[\cdots \Big[\Big[\ov{\Big[M^{\nufr}_{(1,\ov{d}_0)}\Big]},\ov{[M^{\nu}_{\ov{d}_1}]}\Big],\ov{[M^{\nu}_{\ov{d}_2}]}\Big]\cdots,\ov{[M^{\nu}_{\ov{d}_l}]}\Big]
 \numberthis
\end{align*}
 in $^{\ov{n}}_{\infty}\widecheck{V}_{\bullet}$.
Proposition \ref{cor:PMimpliesadPM} ii) shows that the lifts $\Big[M^{\mufr}_{(1,\ov{d})}\Big]\in \ovninfty V_{\bullet}$ are uniquely determined by the vanishing of $\tau_1(\infty)\cap$. Starting from such a lift on the right-hand side and applying Proposition \ref{cor:PMimpliesadPM} iii) repeatedly shows that the term on the left of the equality is $\Big[M^{\mufr}_{(1,\ov{d})}\Big]$.
The natural inclusions in \eqref{eq:liftinc} lead to classes in $^{\ov{n}}_{\infty}V_{\bullet}$ satisfying the vanishing condition so applying \ref{cor:PMimpliesadPM} ii) yet again shows that the lifts coincide with those obtained by pushforwards along the inclusions. 
\end{proof}

The wall-crossing formulae together with the compatibility of the (partially lifted) Lie bracket with Virasoro constraints established in Proposition \ref{cor:PMimpliesadPM} i), iv) imply the main result which was stated in Theorem \ref{thm:first} (I) for Bridgeland stability conditions instead of just slope stability.
\begin{theorem}
\label{thm:main}
Let $\ov{\mu}\in \RR^{V}, \mu_{\infty}\in \RR$, then $\Big[M^{\mufr}_{(1,\ov{d})}\Big]$, $\ov{[M^{\mu}_{\ov{d}}]}$ satisfy Virasoro constraints for any $\ov{d}$ if there exist $\ov{\nu}\in \RR^V, \nu_{\infty}\in \RR$ for which this already holds. If $Q$ is additionally acyclic, then such $(\mu_{\infty},\ov{\mu})$ always exists.

By Definition \ref{def:abstractvir}, this says that 
\begin{equation}
\label{eq:PMvirasoro}
\int_{\Big[M^{\mufr}_{(1,\ov{d})}\Big]} {}^{\ov{n}}_{\infty}\bL_{\inv}(\tau_1) =0=\int_{\ov{[M^{\mu}_{\ov{d}}}]} \bL_{\inv}(\tau_2) \quad \textnormal{for any} \quad \tau_1\in \TT^{^{\ov{n}}_{\infty}Q}, \tau_2\in \TT^{Q}\,.
\end{equation}
  Alternatively, one could write the first condition as
  \begin{equation} 
\label{eq:Pvirasoro}\int_{\Big[M^{\mufr}_{(1,\ov{d})}\Big]} {}^{\ov{n}}_{\infty}\bL'_{k}(\tau) =0 \quad \textnormal{for any} \quad \tau\in \TT^{^{\ov{n}}_{\infty}Q}, k\geq 0\,.\end{equation}

\end{theorem}
\begin{proof}
    I will only focus on $\Qfr$ addressing the statement for an acyclic $Q$ because the case without frozen vertices is a simpler version of the same proof, and the general statement is immediate. Because there are no cycles, one can label the vertices of $Q$ using the set $\{1,\ldots,|V|\}$ in such a way that every edge $e\in E$ satisfies $h(e)>t(e)$. Following \cite[Definition 5.5]{GJT}, I will say that $(\nu_{\infty},\ov{\nu})$ is \textit{increasing} if $\nu_{v}>\nu_{w}$ whenever $v>w$ and $\nu_{\infty}<\nu_v$ for all $v,w\in V$. Then \cite[Theorem 5.8 (iii)]{GJT} shows that the only non-zero classes $\Big[M^{\nufr}_{(1,\ov{d})}\Big]$, $\ov{[M^{\nu}_{\ov{d}}]}$ are
    $$
    \Big[M^{\nufr}_{(1,0)}\Big] = e^{(1,0)}\otimes 1\quad\text{and} \quad \ov{[M^{\nu}_{v}]} =e^{v}\otimes 1 \ \textnormal{for} \ v\in V\,. 
    $$
     By Theorem \ref{thm:WCformulae} (2), one can write any class $\Big[M^{\mufr}_{(1,\ov{d})}\Big]$ as a sum of partially lifted iterated Lie brackets of the form 
    $$
    \widehat{\ad}(P)\big(\ov{[M^{\nu}_v]}\big)\,.
    $$
    In the case that $P = e^{(1,0)}\otimes 1$, it is immediate that both \eqref{eq:PMvirasoro} and \eqref{eq:Pvirasoro} are satisfied, while $\tau_k(\infty)\cap e^{(1,0)}\otimes 1=0$ whenever $k>0$. Because weight zero Virasoro constraints hold for $e^{v}\otimes 1$, all three vanishing conditions also hold for $\widehat{\ad}(e^{(1,0)}\otimes 1)\big(\ov{[M^{\nu}_v]}\big)$ by Proposition \ref{cor:PMimpliesadPM} i), iii) and iv). The general case is shown by induction on the number of iterated operations $\widehat{\ad}$.
\end{proof}
\subsection{Application to Virasoro constraints for surfaces}
\label{sec:P2}
As already explained in the introduction, the derived equivalences between surfaces $S$ with strong full exceptional collections and quivers with relations can be used to prove Virasoro constraints for some surfaces. I address here the three cases $S=\PP^2, \PP^1\times \PP^1$ and $S=\Bl_{\pt}\PP^2$ where $\pt$ is a point in $\PP^2$. The first two cases are sufficient for proving Virasoro constraints on $\Hilb^n(S)$ for a general $S$ due to their universality. 

So let $S=\PP^2,\PP^1\times \PP ^1$, or $\Bl_{\pt}\PP^2$, and fix an ample integral divisor class $H$ for it. The Chern character of a perfect complex $F$ on $S$ will be expressed in the form
$$
\ch(F) = (r,d,n)\in H^\bullet(S)
$$
where $r$ is the rank of $F$, $d = c_1(F)$, and $n=\ch_2(F)$. Letting $p_F$ denote the reduced Hilbert polynomial of $F$, I will also label the resulting Gieseker stability as $p$-stability. This section proves Virasoro constraints for any moduli scheme
$
M^{p}_{r,d,n}
$
of $p$-semistable sheaves of type $(r,d,n)$ in the sense that it shows that
$$
\big[M^{p}_{r,d,n}\big]^{\inva} \in \widecheck{P}_0
$$
for the invariant classes of Joyce \cite[Theorem 7.64]{Jo21}. Here $\widecheck{P}_0$ was used in \cite[§4.4]{BML} to reformulate Virasoro constraints that originally had the form akin to the one presented for quiver moduli in Definition \ref{def:abstractvir}, but the proof is independent of the previous results in \cite{BML} beyond this comparison. Still, I will use multiple times notation that appeared in \cite{BML} and was adapted to quivers here by replacing a variety $X$ with $Q$. Such instances will not be explained in full detail unless necessary. I will also neglect writing $\ov{(-)}$ which was used to specify classes in $\widecheck{H}\big(\mM^{\rig}_Q\big)$.

 The Bridgeland stability I will be working with is not equivalent to $p$-stability but to a version of it twisted by a line bundle. For any line bundle $L$, define $p_E^L = p_{E(-L)}$, then $p^L$-stability of $E$ is defined the same way as $p$-stability but using the polynomials $p^L_{(-)}$. The moduli spaces of $p^L$-semistable sheaves in class $(r,d,n)$ will be labelled by $M^{p^L}_{r,d,n}$ including the case $r=0$.
 It is a well-known fact that Bridgeland stability called  \textit{large volume limit} stability can be identified with the twisted Gieseker stability for some $L$. Its name is motivated by its origin in string theory, and it appeared for the first time in the form presented here in \cite[§14]{bridgelandK3}.
 
One relies on Mumford stability for sheaves to define the heart of the large volume limit stability. Because I already used the terminology $\mu$-stability to describe the slope stability for a quiver, I will instead write 
$$
\lambda(F) = \frac{d\cdot H}{r}
$$
to denote the slope of a sheaf on $S$. Mumford stability will then also be called $\lambda$-stability. When $F$ is torsion, I will use the convention 
$\lambda(F) = \infty$ with $\infty > q$ for any $q\in \RR$ which I will interpret as any torsion sheaf being $\lambda$-semistable of slope $\infty$. 

The main tool I rely on are the derived equivalences in the following example.
\begin{example}
    \label{ex:excP2}
    \begin{enumerate}
        \item By \cite{beilinson}, it is known that $D^b(\PP^2)$ is equivalent to the bounded derived category of the quiver 
    \[P_2=\begin{tikzcd}[column sep=large]\overset{1}{\bullet}\arrow[r, shift left = 0.55em, "a_1"]\arrow{r}[name=B, description]{a_2}\arrow[r, shift right = 0.55em, "a_3"']&\overset{2}{\bullet}\arrow[r, shift left = 0.55em,"b_1"]\arrow{r}[name=B, description]{b_2}\arrow[r, shift right = 0.55em, "b_3"']&\overset{3}{\bullet}\end{tikzcd}\]
    with the relations $b_j\circ a_i = b_i \circ a_j$ for $i,j=1,2,3$. To see this, one can take the strong full exceptional collection $$\mE(k)=\big(\mO_{\PP^2}(k-2)[2], \mO_{\PP^2}(k-1)[1], \mO_{\PP^2}(k)\big)$$ for any $k\in \ZZ$ and assign to each object a vertex in $(1,2,3)$ in the prescribed order. The dimension of $\Ext^1\big(\mO_{\PP^2}(k-1), \mO_{\PP^2}(k)\big) = H^0\big(\mO(1)\big)$, hence the three arrows. The relations can be understood from the map $H^0\big(\mO(1)\big)\times H^0\big(\mO(1)\big)\to H^0\big(\mO(2)\big)$.
        \item For $\PP^1\times \PP^1$, I choose to follow \cite[§4.1]{EMiles} and use the strong full exceptional collection 
        $$
        \mE(i,j) = \big(\mO(i-2,j-1)[2], \mO(j-1,j-2)[2],\mO(i-1,j-1)[1], \mO(i,j)\big)
        $$
        where $i,j\in \ZZ$ and $\mO(i,j)$ denotes the line bundle $\mO_{\PP^1}(i)\boxtimes \mO_{\PP^1}(j)$.
        By the same argument as above, one can see that the resulting quiver is given by 
        $$
        P_1\times P_1 = \begin{tikzcd}
            \overset{1}{\bullet}\arrow[dr, shift right = 0.25em, "x_1"']\arrow[dr, shift left = 0.25em, "x_2"]&&[3em]\\
            &\overset{3}{\bullet}\arrow[r,shift left = 0.75em,"(x_iy_j)_{i,j\in\{1,2\}}"]\arrow[r,shift left = 0.25em]\arrow[r, shift right = 0.25em]\arrow[r, shift right = 0.75em]&\overset{4}{\bullet}\\
            \overset{2}{\bullet}\arrow[ur, shift right = 0.25em, "y_1"']\arrow[ur, shift left = 0.25em, "y_2"]&&
        \end{tikzcd}
        $$
        with the relations 
        $$
        (x_iy_k)\circ x_j = (x_jy_k)\circ x_i\,,\qquad  (x_ky_i)\circ y_j = (x_ky_j)\circ y_i
        $$
        for $i,j,k=1,2$.
        \item For $\Bl_{\pt}\PP^2$, set $\tilde{H}$ to be the strict transform of the hyperplane on $\PP^2$ and $E$ the exceptional divisor. Then using the divisor $F=\tilde{H}-E$, the authors of \cite[§5.1]{EMiles} describe two strong full exceptional collections:
        \begin{align*}
        \mathcal{E}_1 &= \big(\mO_S(-E-2F)[2],\mO_S(E)|_E[1],\mO_S(-E-F)[1],\mO_S\big)\,,\\
        \mathcal{E}_2 &= \big(\mO_S(-E-2F)[2], \mO_S(-F)[1], \mO_S(E)|_E,\mO_S\big)
        \end{align*}
        which lead to two quivers with relations $Q_1$ and $Q_2$.
    \end{enumerate}
   
    The resulting derived equivalences constructed using $\mE(k), \mE(i,j),\mE_1$ and $\mE_2$ are denoted in the respective order by
    \begin{align*}
    \label{eq:derequiv}
    \Gamma(k):& \ D^b(\PP^2)\stackrel{\sim}{\longrightarrow}D^b(P_2)\,,\\
     \Gamma(i,j):& \  D^b(\PP^1\times \PP^1)\stackrel{\sim}{\longrightarrow}D^b(P_1\times P_1)\,,\\
     \Gamma_{a}:& \ D^b(\Bl_{\pt}\PP^2) \stackrel{\sim}{\longrightarrow}D^b(Q_a)\quad \text{for}\quad a=1,2\,.
    \numberthis
    \end{align*}
    Starting in either of the above cases from $\rep(Q)\subset D^b(Q)$, I will write
    \begin{align*}
   \mA(k)= \Gamma(k)^{-1}&\big(\rep(P_2)
\big)\,,\qquad  \mA(i,j)= \Gamma(i,j)^{-1}\big(\rep(P_1\times P_1)
\big)\,,\\
\mA_{a} &= \Gamma_a^{-1}\big(\rep(Q_a)\big)\quad \text{for}\quad a=1,2
    \end{align*}
 for the resulting hearts of a bounded t-structure on $D^b(S)$.
\end{example}
\begin{remark}
\label{rem:derequiv}
    To be precise, one needs a bit more than just the derived equivalence of the categories. For a smooth projective surface $S$, the moduli stack $\mM_S$ appearing in \cite[Definition 4.1]{BML} parametrizes perfect complexes on $S$ as objects in the dg-category $L_{\text{pe}}(S)$ from \cite[§3.5]{TV07}. Similarly, the moduli stack $\mM_Q$ that I used in Definition \ref{def: VOAconstruction} parametrizes the objects in the dg-category $\CC Q\! -\! \Mod $ of left dg-modules of $\CC Q$ whose underlying complex of vector spaces is perfect. Denoting this latter dg-category by $(\CC Q\! -\! \Mod)_{\text{ppe}}$, it is necessary for the present purpose to lift the derived equivalences to quasi-equivalences of dg categories
    $$
    L_{\pe}(S)\stackrel{\sim}{\longrightarrow} (\CC Q\! -\! \Mod)_{\text{ppe}}\,.
    $$
    But this quasi-equivalence is constructed in the same way as the equivalence between derived categories in \cite{bondaltilt}, just on the level of dg-categories. As a consequence, we obtain an isomorphisms $\mM_X\xrightarrow{\sim}\mM_Q$ in the homotopy category of higher stacks. This isomorphism identifies the $\Ext$ complex on $\mM_S\times \mM_S$ (see \cite[Definition 4.1]{BML}) with the $\Ext_Q$ complex on $\mM_Q\times \mM_Q$ from Definition \ref{def: VOAconstruction}\footnote{The simplest way to see this is to use the natural direct sum map $\Sigma:\mM_Q\times \mM_Q$ and to note that $\Ext_Q[1]$ is the virtual normal complex of this map. Because the virtual structure is determined uniquely on $\mM_Q$ and $\mM_Q\times \mM_Q$ by their derived refinements from which they are 
 constructed in \cite{TV07}, and this also holds for $\mM_X,\mM_X\times \mM_X$, we see that the two $\Ext$-complexes must be the same. Here we used that the equivalence of higher stacks lifts to derived stacks.}. In particular, it induces an isomorphism of vertex algebras 
    $$
    \widehat{H}_{\bullet}(\mM_S)\stackrel{\sim}{\longrightarrow} \widehat{H}_{\bullet}(\mM_Q)\,.
    $$
    The same argument is applicable to any derived equivalence that lifts to a Morita-equivalence (not just quasi-equivalence) of dg-categories.
 \end{remark}
For the proof of the theorem below, I will assume some basic knowledge about Bridgeland stability conditions on a triangulated category $\mT$ defined in \cite{bridgeland} as pairs $\sigma = (\mP,Z)$ where $\mP$ is a \textit{slicing} of $\mT$ and $Z: K^0(\mT)\to \CC$ is a compatible \textit{central charge}. For a triangulated category $\mT$, I will denote by $\Stab(\mT)$ its stability manifold. A common way to construct $\sigma\in \Stab(\mT)$ is to find a heart $\mA$ of a bounded t-structure on $\mT$ and to introduce a \textit{stability function }
$$
Z: K^0(\mT)=K^0(\mA)\longrightarrow \CC \quad \textnormal{satisfying}\quad Z(E)\in \HH \quad \text{for} \quad E\in \mA\,.
$$
Here $\HH$ denotes the strict upper half-plane $\HH = \big\{me^{i\pi\phi}\colon m>0,\phi\in (0,1]\big\}$. If $Z$ has the \textit{Harder-Narasimhan property} from \cite[Definition 2.3]{bridgeland}, Bridgeland \cite[Proposition 5.3]{bridgeland} extends this data to a stability condition $\sigma$ on $\mT$. In this scenario, I will say that $\sigma$ is a stability condition of $\mT$ constructed on the heart $\mA$. I will denote the subset of $\Stab(\mT)$ of such stability conditions by $\Omega_{\mA}$. One has the action of $\RR$ on $\Stab(\mT)$ given for each $a\in \RR$ by shifting the phases of semistable objects by $+a$, and $\ov{\Omega}_{\mA}$ will denote the set of stability conditions obtained from applying this action to $\Omega_{\mA}$.

For the surface $S$, one first needs a suitable heart $\mA_{s}\subset D^b(S)$ such that Gieseker stability is captured by a stability condition in $\Omega_{\mA_{s}}$. It is obtained in \cite[Lemma 6.1]{bridgelandK3}, \cite[Definition 5.10]{ABCH}, and \cite[Definition 6.6]{Macrilecture} by \textit{tilting} which is a general construction introduced in \cite[Chapter 1]{HRS}.
\begin{definition}
Starting from the torsion pair $(\mF_{s},\mT_{s})$ defined for each $s\in \RR$ by
  \[\mF_s = \left\{F\in \Coh(S) \colon \begin{array}{c}\textnormal{every Mumford semistable}\\ \textnormal{factor }F_i\textnormal{ in the Harder--Narasimhan}\\ \textnormal{filtration of }F\textnormal{ satisfies } \lambda(F_i)\leq \frac{K_S\cdot H}{2H^2} + s\end{array}\right\}\,,\]  
    \[\mT_s = \left\{F\in \Coh(S) \colon \begin{array}{c}\textnormal{every Mumford semistable}\\ \textnormal{factor }F_i\textnormal{ in the Harder--Narasimhan}\\ \textnormal{filtration of } F\textnormal{ satisfies } \lambda(F_i)> \frac{K_S\cdot H}{2H^2} + s\end{array}\right\}\,,\]
one defines
$$
\mA_s = \big\langle \mT_s,\mF_s[1] \big\rangle\,.
$$
Note that I use a different convention compared to Bridgeland which makes the comparison with Gieseker stability more direct when $K_S\neq 0$. The stability function $Z_{s,t}: \mA_s\to \HH$ depends on an additional parameter $t>0$, and it is given in \cite[Lemma 6.2]{bridgelandK3}, \cite[Theorem 5.11]{ABCH}, and \cite[Theorem 6.10]{Macrilecture} by
$$
Z_{s,t}(F)  = -\int_{\PP^2}\ch(F) e^{-\frac{K_S}{2}-(s+it)H}\,.
$$
Denote by $\sigma_{s,t}$ the resulting stability condition in $\Omega_{\mA_s}$.
\end{definition}
I shall now prove the following statement. 
\begin{theorem}
\label{thm:P2indep}
    The classes $\big[M^{p}_{r,d,n}\big]^\inva$ for $S=\PP^2,\PP^1\times \PP^1$ and $\Bl_{\pt}(\PP^2)$ and any Chern character of a sheaf $(r,d,n)$ on $S$ satisfy Virasoro constraints in the sense of \cite[Conjecture 1.10]{BML}.
\end{theorem}
\begin{proof}
For the purpose of simplicity, I focus on the $r>0$ case for now and explain the necessary adapations for $r=0$ later. In the first step, I will recall that some twist by a line bundle of every $p$-semistable sheaf for the above surfaces can be identified by using the derived equivalences from \eqref{eq:derequiv}  with a Bridgeland semistable representations of the quivers from Example \ref{ex:excP2}. Gieseker stability for any surface is related to the above Bridgeland stability in \cite[Exercise 6.27]{Macrilecture} which is based on the previous similar statements in \cite[§14]{bridgelandK3} and \cite[§6]{ABCH} given in special cases. One then uses \textit{Bertram's nested wall theorem} proved in \cite[Theorem 3.1]{Macioca13} to move around in $\Stab\big(D^b(S)\big)$ without changing the classes of semistable objects and to reach the regions $\ov{\Omega}_{A(k)}$ or $\ov{\Omega}_{A(i,j)}$. In the case of $\Bl_{\pt}\PP^2$, one needs to consider the twists by line bundles of both $\ov{\Omega}_{\mA_1}$ and $\ov{\Omega}_{\mA_2}$ to cover all sheaves. I have to of course translate all of this into a statement about invariants. 

The rest of the proof reduces to Proposition \ref{prop:bridgeland} which generalizes wall-crossing to include Bridgeland stability conditions constructed on the standard heart $\rep(Q)\subset D^b(Q)$. I can then repeat the argument given in the proof of Theorem \ref{thm:main} after noting that derived equivalences preserve Virasoro constraints.

Let us formulate the above arguments in a precise way now. For a fixed $(r,d,n)$ such that
\begin{equation}
\label{eq:conditionons}
s\Big(1-\frac{H^2}{r}\Big)< \frac{d\cdot H}{r}-\frac{K_S\cdot H}{2H^2}\,,\qquad r>0
\end{equation}
it is stated in \cite[Exercise 6.27]{Macrilecture} (based on \cite[Proposition 14.2]{bridgelandK3} and \cite[Proposition 6.2, Corollary 6.5]{ABCH}) that when $t\gg 0$ and $s\in \ZZ$, the $\sigma_{s,t}$-semistable objects in $\mA_s$ of class $(r,d+sH,n')$ coincide with the $p^{sH}$-semistable torsion-free sheaves of the same class. Thus there is an identification of substacks 
\begin{equation}
\label{eq:openembed2}
\mM^{p}_{r,d,n}=\mM^{p^{sH}}_{r,d+sH,n'}= \mM^{(s,t)}_{r,d+sH,n'}\subset \mM^{\rig}_S\end{equation}
where $n'=n+sd\cdot H+r \frac{(sH)^2}{2}$. The symbol $\mM^{(-)}_{(-)}$ is reserved for the stack as opposed to the scheme $M^{(-)}_{(-)}$, and $\mM^{(s,t)}_{(-)}$ is the stack of $\sigma_{s,t}$-semistable objects in $\mA_s$ with a similar notation for the corresponding scheme $$M^{(s,t)}_{(-)}\,.$$ As before, $\mM^{\rig}_X$ is the rigidification of the stack of all perfect complexes on $X$ (used in \cite[Definition 4.1]{BML}). Finally, as one can rescale $H$ without changing $p$-stability, one can assume that $\frac{H^2}{r}\neq 1$. Therefore, one can always find a suitable $s_0\in \ZZ$ satisfying \eqref{eq:conditionons} together with a sufficiently large $t_0$ which I choose as our starting point $(s_0,t_0)$. In Lemma \ref{lem:filing}, I will show that the above equalities translate word for word to invariants, but this requires some further arguments. 

In \cite[Corollary 7.6]{ABCH} and \cite[§4.2]{EMiles}, \cite[§6.2]{EMilesthesis}, the authors start at $(s_0,t_0)$, and then vary the stability condition $\sigma_{s,t}$ along a (numerical) wall $W$ in the $(s,t)$-plane while decreasing $t$. These walls are described by Bertram's nested wall theorem proved in \cite[Proposition 2.6, Theorem 3.1]{Macioca13}, which shows that they are vertical lines or semi-circles centred at $t=0$ and that they never intersect each other because the semi-circles are nested, while the lines are disjoint from them. In particular, any semistable object remains semistable along this trajectory so that the moduli stacks $\mM^{(s,t)}_{r,d', n'}$ (here $d'=d+s_0H$) are kept unchanged.  This is mirrored by the associated invariant classes as will be explained in Lemma \ref{lem:filing}.

The next argument was proved in the literature on a case-by-case basis, so we address each $S$ separately:
\begin{enumerate}
    \item For $\PP^2$, \cite[Proposition 7.5 and Corollary 7.6]{ABCH} show that $W$ eventually enters the region $\ov{\Omega}_{\mA(k)}$ for some $k$. I will label a choice of the resulting stability condition by $\sigma(k)\in \ov{\Omega}_{\mA(k)}$. Applying the derived equivalence $\Gamma(k)$ and shifting the phases of all semistables, one can transform $\sigma(k)$ into
    $$
\sigma^{P_2}_k \in \Omega_{\rep(P_2)}\,.
$$
Then the moduli stack $\mM^{\sigma^{P_2}_k}_{\ov{d}}$ of $\sigma^{P_2}_k$-semistable representations of $P_2$ with dimension vector $\ov{d}$ corresponding to $(r,d',n')$ is naturally isomorphic to $\mM^{\sigma(k)}_{r,d',n'}$.
    \item The claim that $W$ enters the quiver region $\ov{\Omega}_{\mA(i,j)}$ for $\PP^1\times \PP^1$ is proved in \cite[§4.2]{EMiles}, \cite[Proposition 6.1]{EMilesthesis}. In the same way, one obtains a stability condition in $\Omega_{\Rep(P_1\times P_1)}$ with the associated moduli stack isomorphic to $\mM^{\sigma(k)}_{r,d',n'}$.
    \item To conclude the same for $\Bl_{\pt}\PP^2$, one needs two different quiver regions obtained from $\ov{\Omega}_{\mA_1}$ and $\ov{\Omega}_{\mA_2}$ by twisting with line bundles. The wall $W$ enters the union of these two regions by \cite[§5.3, §5.4]{EMiles}.
\end{enumerate}
While I compared the stacks for different stabilities, I have not addressed the corresponding equalities for invariants yet. The next lemma fills in the necessary details. 
\begin{lemma}
   \label{lem:filing}
For each $(s,t)$ contained in the numerical wall $W$, there are well-defined classes
$
\Big[M^{(s,t)}_{r,d',n'}\Big]^{\inva}\in \widecheck{H}_0(\mM^{\rig}_S)
$ constructed by the same procedure as in \cite[§7.2.2, §9.1]{Jo21}. They satisfy
$$
\Big[M^{(s,t)}_{r,d',n'}\Big]^{\inva} = \Big[M^{p^{s_0H}}_{r,d',n'}\Big]^{\inva} =\Big[M^{p}_{r,d,n}\Big]^{\inva} \,.
$$
The last step uses independence of the $[-]^{\inva}$ classes under twisting by line bundles which is a consequence of their definition.
Simultaneously, when $S=\PP^2$ this defines classes
$$
\Big[M^{\sigma(k)}_{r,d',n'}\Big]^{\inva}\in \widecheck{H}_0(\mM^{\rig}_{\PP^2})
$$
which satisfy (see Remark \ref{rem:derequiv})
$$
\Big[M^{\sigma(k)}_{r,d',n'}\Big]^{\inva} = \Big[M^{\sigma^{P_2}_k}_{\ov{d}}\Big]\in 
\widecheck{H}_0\big(\mM^{\rig}_{\PP^2}\big)=\widecheck{H}_0\big(\mM^{\rig}_{P_2}\big)
$$
The classes on the right-hand side of the equalities are defined by applying \cite[Example 5.6, §6.4.3, §9.1]{Jo21} to the stability condition $\sigma^{P_2}_k$. 

The corresponding statements for $S=\PP^1\times \PP^1$ and $\Bl_{\pt}\PP^2$ also hold.  
\end{lemma}
\begin{proof}
    For the sake of brevity, I will write $\alpha = (r,d,n),\alpha'=(r,d',n')$ here, and one can think of them as either rational K-theory classes or their Chern characters. I will recall the definition of $\big[M^p_{\alpha}\big]^{\inva}$ here before I explain how to adapt it for our moduli spaces at any $(s,t)$. Instead of studying sheaves, Joyce \cite[§7.2.2, §9.1]{Jo21} based on \cite{mochizuki} considers the moduli schemes $P^l_{\alpha}$ of stable pairs 
    $
    \mO_S(-l)\xrightarrow{s} F
    $ for a sufficiently large $l$ and $\ch(F)=\alpha$ where the stability is defined in \cite[Example 5.6]{Jo21} as follows:
    \begin{enumerate}
        \item $F$ is $p$-semistable,
        \item $s\neq 0$, and if it factors through a strict subsheaf $F'\subsetneq F$, then $p_F'<p_F$.
    \end{enumerate}
    These moduli schemes carry a virtual fundamental class $\big[P^l_{\alpha}\big]^{\vir}$ and admit the natural projection $\Pi^{l}:P^l_{\alpha}\to \mM^{\rig}_S$. Using $T_{\Pi^l}$ to denote the relative tangent bundle of this projection and $c_{\rk}\big(T_{\Pi^l}\big)$ its top Chern class, the classes
    $$
\bar{\Upsilon}^l_{\alpha} = \Pi^l_*\Big(\big[P^l_{\alpha}\big]^{\vir}\cap c_{\rk}\big(T_{\Pi^l}\big)\Big)\,,
    $$
    are defined in \cite[Definition 9.2]{Jo21}.  To obtain the invariants $\big[M^p_{\alpha}\big]^{\inva}$, the above classes need to be corrected by subtracting contributions corresponding to strictly semistable sheaves. This is expressed in \cite[Proposition 9.3]{Jo21} in terms of the Lie bracket coming from the vertex algebra on $\widehat{H}_{\bullet}(\mM_S)$. The rest of \cite[§9]{Jo21} is then dedicated to showing that $\big[M^p_{\alpha}\big]^{\inva}$ are independent of the choice of $l$. This hinges on \cite[Corollary 9.10]{Jo21} which states that
$$
\chi\big(\alpha(l_1)\big)\bar{\Upsilon}^{l_2}_{\alpha} - \chi\big(\alpha(l_2)\big)\bar{\Upsilon}^{l_1}_{\alpha}  = \sum_{\beta\in \text{Dec}^p(\alpha)}\Big[\bar{\Upsilon}^{l_1}_{\beta},\bar{\Upsilon}^{l_2}_{\alpha-\beta}\Big]
$$
where $\text{Dec}^p(\alpha)\subset K^0(S,\QQ)$ is the set of K-theory classes satisfying
\begin{enumerate}
\item $p_{\beta} = p_{\alpha}$ and thus $p_{(\alpha-\beta)} = p_{\alpha}$,
    \item there exists an inclusion $F\hookrightarrow E$ of $p$-semistable sheaves with $\ch(F) = \beta$ and $\ch(E)=\alpha$.
\end{enumerate}
The second statement could be equivalently stated as $\mM^p_{\beta},\mM^p_{\alpha-\beta}\neq 0$ which is closer to the original formulation in \cite{Jo21}.
We see that to prove that the definition is independent of $l$,  one only needs $\ov{\Upsilon}^{l}_{\beta}$ to be defined whenever $\beta \in \text{Dec}^p(\alpha)$ and $l\gg 0$. We want the same to be true when one works with any one of the above stability conditions $\sigma$ and defines $\text{Dec}^{\sigma}(\alpha)$ in the same way. 

Replacing $p$ by $p^{s_0H}$, I define moduli spaces $P^{l,s_0H}_{\alpha'}$ of stable pairs of the same form $\mO_S(-l)\to F$, but this time the definition of stability uses $p^{s_0H}$. The same argument as above implies that the resulting $\Big[M^{p^{s_0H}}_{\alpha'}\Big]^{\inva}$ is independent of the choice of $l$ which in turn shows that 
$$
\Big[M^{p^{s_0H}}_{\alpha'}\Big]^{\inva} = \Big[M^{p}_{\alpha}\Big]^{\inva}\,.
$$
I will show now that the same holds when $\alpha'$ is replaced by $\beta\in\text{Dec}^{\sigma}(\alpha')$ where $\sigma$ can be any of the above three stability conditions because this set is the same for all three cases. I use the following statement proved in \cite[Proposition 6.22 (7)]{Macrilecture} based on \cite[Lemma 6.3]{ABCH}: for any $(s,t)\in W$, $F\to E$ is an inclusion in $\mA_s$ of $\sigma_{s,t}$-semistable sheaves with equal phases with respect to $\sigma_{s,t}$ and with $\ch(E)=\alpha'$, if and only if this is true for any other $(s_0,t_0)\in W$. As a consequence, I immediately conclude that 
$$
\text{Dec}^{\sigma_{s,t}}(\alpha')= \text{Dec}^{\sigma_{s_0,t_0}}(\alpha')\,,\qquad \mM^{(s,t)}_{\beta}= \mM^{(s_0,t_0)}_{\beta} \quad \text{for}\quad \beta\in \text{Dec}^{\sigma_{s,t}}(\alpha')\,.
$$
If $[F]\in \mM^{(s_0,t_0)}_{\beta}$ for $\beta\in \Dec^{\sigma_{s_0,t_0}}(\alpha')$, then there is a $[G]\in \mM^{(s_0,t_0)}_{\alpha'-\beta}$ such that $F\oplus G$ is a $p^{s_0H}$-semistable torsion-free sheaf, so $[F]\in \mM^{p^{s_0H}}_{\beta}$ and $\beta\in \Dec^{p^{s_0H}}(\alpha')$. After showing the converse by the same argument, I deduce that $$\Dec^{\sigma_{s_0,t_0}}(\alpha')= \Dec^{p^{s_0H}}(\alpha')\,,\qquad \mM^{(s_0,t_0)}_{\beta}= \mM^{p^{s_0H}}_{\beta} \quad \text{for}\quad \beta\in \text{Dec}^{p^{s_0H}}(\alpha')\,.$$

Therefore, $P^{l,\sigma}_{\alpha'}$ is defined for all three of the above stability conditions $\sigma$ by replacing $p$-stability in the first definition. Simultaneously, this applies to $P^{l,\sigma}_{\beta}$ for all $\beta\in \Dec^{\sigma}(\alpha')$. Using the above results, it is immediate that all three types of pair moduli spaces are equal and therefore
$$
\Big[M^{(s,t)}_{\alpha'}\Big]^{\inva} = \Big[M^{p^{s_0H}}_{\alpha'}\Big]^{\inva} =\Big[M^{p}_{\alpha}\Big]^{\inva} 
$$
are all independent of $l$. 

For the rest of the argument, I only focus on $\PP^2$ because it is identical for $\PP^1\times \PP^1$. The equivalence $\Gamma(k)$ gives us another way to construct stable pairs for $\sigma(k)$ stability. The semistable sheaves in $\mA_s$ are mapped to semistable representations of $P_2$ (or their shifts) from which one constructs pairs by attaching a vertex $\infty$ with dimension $1$ and edges ending at every one of the original vertices. This construction is explained in greater generality in the proof of Proposition \ref{prop:bridgeland} and was used in \cite[Definition 5.12]{GJT}. The resulting pair moduli space $P^{\infty, \sigma^{P_2}_k}_{\ov{d}}$ parametrizes stable representations with dimension vector $(1,\ov{d})$ of the quiver $^{\infty}_{\ov{1}}P^2$ where $\ov{1}$ is the vector with 1 at each vertex. The stability condition of these representations is not important for the current proof, so I will not write it out as it could be guessed from its analog for sheaves. Using $\Big[P^{\infty, \sigma^{P_2}_k}_{\ov{d}}\Big]$, I define
$$\Big[M^{\sigma^{P_2}_k}_{\ov{d}}\Big]\in \widecheck{H}_0\big(\mM^{\rig}_{P_2}\big)$$
in the same way as was done before for sheaves. Because both constructions of pair moduli spaces $P^{l,\sigma(k)}_{\alpha'}$ and $P^{\infty, \sigma^{P_2}_k}_{\ov{d}}$ are special cases of \cite[Example 5.6]{Jo21}. The two different definitions of classes counting $\sigma(k)$-semistable objects they lead to are identical by the same argument in \cite[§9.2, §9.3]{Jo21} that was used to show $l$-independence. This concludes the proof that 
$$
\Big[M^{\sigma(k)}_{\alpha'}\Big]^{\inva} = \Big[M^{\sigma^{P_2}_k}_{\ov{d}}\Big]\in \widecheck{H}_0\big(\mM^{\rig}_{\PP^2}\big)=\widecheck{H}_0\big(\mM^{\rig}_{P_2}\big)\,.
$$
\end{proof}
\begin{remark}
    Here, I address the necessary changes for the case $r=0$. In this case, the proof of the second equality in \eqref{eq:openembed2} does not follow directly from \cite[Exercise 6.27]{Macrilecture}, \cite[Proposition 14.2]{bridgelandK3}, and \cite[Proposition 6.2]{ABCH}. These references can only be used to show that every $\sigma_{s,t}$-semistable object of class $(0,d,n')$ is a $p^{sH}$-semistable 1-dimensional sheaf for $t\gg 0$. That there is a sufficiently large $t$ such that the two moduli stacks can be identified in $\mM^{\rig}_S$ follows from \cite[Lemmma 7.3]{Wolf} -- I was made aware of this result by Pierrick Bousseau. Note that we never need any restrictions on $s$ like in \eqref{eq:conditionons} because 1-dimensional sheaves are always contained in $\mT_s$. 

    The walls for $r=0$ are always semi-circles by \cite[§6, Case 2]{ABCH} so this part of the argument does not change. To retain the control over the invariant classes along the wall in the sense of Lemma \ref{lem:filing}, we need to replace \cite[Proposition 6.22 (7)]{Macrilecture} by an appropriate result for 1-dimensional sheaves. This can be found in \cite[Proposition 6.10]{Wolf}, so we still recover the equalities 
    $$
\Big[M^{p}_{0,d,n}\Big]^{\inva} = \Big[M^{(s,t)}_{0,d,n'}\Big]^{\inva}=\Big[M^{\sigma^{P_2}_k}_{\ov{d}}\Big]\in 
\widecheck{H}_0\big(\mM^{\rig}_{\PP^2}\big)=\widecheck{H}_0\big(\mM^{\rig}_{P_2}\big)
$$
and their analogs for $S=\PP^1\times \PP^1$ and $\Bl_{\pt}\PP^2$.
\end{remark}

By Proposition \ref{prop:bridgeland} below, one knows that $\Big[M^{\sigma^{P_2}_{k}}_{\ov{d}}\Big]\in \widecheck{P}_0$ or in other words -- they satisfy Virasoro constraints in the sense of Definition \ref{def:abstractvir}.
Finally, I use the isomorphisms of vertex algebras $\check{H}_\bullet(\mM_{P_2})\cong \check{H}_\bullet(\mM_{\PP^2})$ constructed from the derived equivalences in Remark \ref{rem:derequiv} to conclude from Lemma \ref{lem:filing} that the classes $\Big[M^{p^{s_0H}}_{\alpha'}\Big]^{\inva}$ also lie in $\widecheck{P}_0$. Twisting by a line bundle is also a derived equivalence, so the same argument implies that 
$$
[M^{p}_{\alpha}]^{\inva}\in \widecheck{P}_0\,.
$$
By \cite[Theorem 1.9, Conjecture 1.10]{BML}, this means that $[M^{p}_{\alpha}]^{\inva}$ satisfy Virasoro constraints for sheaves.
The same is true for $\PP^1\times \PP^1$ and $\Bl_{\pt}(\PP^2)$.   
\end{proof}
\begin{proposition}
\label{prop:bridgeland}
\sloppy    The wall-crossing formula \eqref{eq:quiverWC} still holds when $\mu$- and $\nu$-stabilities are replaced by $\sigma_0, \sigma_1\in \Omega_{\rep(Q)}$ .
    
Consequently, the (virtual) fundamental classes $[M^\sigma_{\ov{d}}]$ satisfy Virasoro constraints for any $\sigma\in \Omega_{\rep(Q)}$ and $\ov{d}\in K^0(Q)$ if there exists $\sigma'\in \ov{\Omega}_{\Rep(Q)}$ such that this holds for $[M^{\sigma'}_{\ov{c}}]$ for all $\ov{c}\in K^0(Q)$. If $Q$ is acyclic, then any increasing slope stability from the proof of Theorem \ref{thm:main} is an example of such $\sigma'$.
\end{proposition}\begin{proof}
The proof of the first statement follows from checking that \cite[Assumptions 5.1--5.3]{Jo21} are satisfied. This might seem as a daunting task at first, but for quivers, it boils down to \cite[Assumption 5.2 (g), (h)]{Jo21} which follows for $\mu$-stability by \cite[Proposition 4.3]{King} together with \cite[§6.4.3]{Jo21}. The latter rephrases \cite[Assumption 5.2 (h)]{Jo21} which applies to Joyce's quiver pair moduli schemes from \cite[Definition 5.5]{Jo21} in such a way that King's projectivity result can be used again. Since I will follow the same logic here, I will only mention parts where the present proof deviates from the original one. Hence, it is recommended to get acquainted with the said quiver pairs and their simplification in \cite[§6.4.3]{Jo21}.

\begin{proof}[Checking {\cite[Assumption 5.2 (g)]{Jo21}}] In the present case, the stability condition is $\sigma\in \Omega_{\rep(Q)}$ determined by a stability function $Z$. I will just recall a standard argument for proving that $M^{\sigma}_{\ov{d}}$ is projective whenever there are no strictly $\sigma$-semistables of class $\ov{d}$ because I will need to modify it later on. To do so, I may assume that $$\Im\big(Z(E)\big)>0$$ for any $E\in \rep(Q)$. If this is not the case, one can change the values of $Z$ on $v\in V$ slightly without reordering them so that $\Im\big(Z(v)\big)>0$ for all $v\in V$. To make the following paragraphs more intuitive, I will use the notation 
$$
\rk_{\sigma}(\ov{d})= \Im\big(Z(\ov{d})\big)\,,\qquad \deg_{\sigma}(\ov{d})=-\Re\big(Z(\ov{d})\big)\,.
$$
Fixing $\ov{c}\in K^0(Q)$, I define the character $\theta^{\ov{c}}_{\sigma}$ on $\rep(Q)$ in the sense of \cite[Definition 1.1]{King} by
$$
\theta^{\ov{c}}_{\sigma}(\ov{d}) = \deg_{\sigma}(\ov{d}) - \rk_{\sigma}(\ov{d})\frac{\deg_{\sigma}(\ov{c})}{\rk_{\sigma}(\ov{c})}\,.
$$
It is clear that a representation $\ov{U}$ of $Q$ with the dimension vector $\ov{c}$ is (semi)stable if and only if all of its subobjects $\ov{U}'$ satisfy
$$
\theta^{\ov{c}}_{\sigma}\big(\ov{U}'\big)(\leq)<0\,.
$$
Simoultaneously $\theta^{\ov{c}}_{\sigma}(\ov{c})=0$ so one may apply \cite[Proposition 4.3]{King} to establish projectivity of the moduli scheme $M^{\sigma}_{\ov{c}}$.
\end{proof}
\begin{proof}[Checking {\cite[Assumption 5.2 (h)]{Jo21}}]
Quiver pairs are generally constructed after specifying additional data that includes a further acyclic quiver $\mathring{Q}$ whose sets of vertices and edges will be denoted by $(\mathring{V}, \mathring{E})$. The quiver $\mathring{Q}$ contains a special set of vertices labelled here by $\circ$. In fact, the general proof of wall-crossing in \cite[§10]{Jo21} only requires the quivers $\mathring{Q}=Q_{\text{Flag}}, Q_{\text{MS}}$  from \cite[(10.1), (10.57)]{Jo21}. They are given as follows:
\begin{center}
\includegraphics[scale=1]{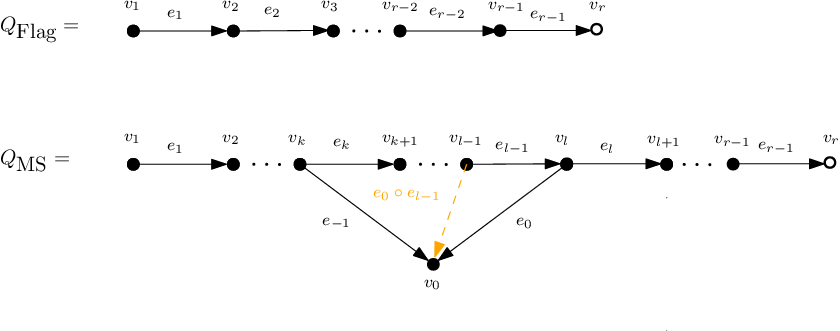}
\end{center}
The {\color{BurntOrange}dashed arrow} denotes a relation $r=e_0\circ e_{l-1}$ that generates $\mR$. Notice that they have a unique special vertex $\circ$. It will become useful to introduce the quiver
$$
\mathring{Q}_{r-1} = \big(\mathring{V}_{r-1},\mathring{E}_{r-1}\big)
$$
obtained from $\mathring{Q}$ by removing the vertex $\circ$ and the edge $e_{r-1}$.
  
Once $\mathring{Q}$ is chosen, the objects from \cite[Definition 5.5]{Jo21} for wall-crossing in the category $\rep(Q)$ are precisely the representations of 
$$
\mathring{Q}\sqcup_{\circ} Q = \big(\mathring{V}_{r-1}\sqcup V, \mathring{E}_{r-1}\sqcup V\sqcup E \big)\,.
$$
The edges labelled by vertices in $V$ start at $v_{r-1}$ and end at their respective vertex. Because this quiver is obtained roughly by gluing $Q$ to $\mathring{Q}$ along the vertex $\circ$, its dimension vectors also consist of two parts $\big(\mathring{\ov{d}}_{r-1}, \ov{d}\big)$ where $\mathring{\ov{d}}_{r-1}\in K^0(\mathring{Q}_{r-1}),\ov{d}\in K^0(Q)$, and the coefficients of $\mathring{\ov{d}}_{r-1}$ are denoted by $\mathring{d}_v$ for $v\in \mathring{V}_{r-1}$. 

The equation \cite[(5.21)]{Jo21} defines a stability condition on the category $\mA = \rep\big(\mathring{Q}\sqcup_{\circ} Q\big)$ after specifying
\begin{enumerate}
    \item a stability condition on $\rep(Q)$,
    \item a vector of stability parameters $\mathring{\ov{\mu}}_{r-1}\in \RR^{\mathring{V}_{r-1}}$ with the component corresponding to the vertex $v\in \mathring{V}_{r-1}$ denoted by $\mathring{\mu}_v$,
    \item a group homomorphism $\epsilon: K^0(Q)\to \RR$,
    \item and a \textit{rank function} $\rk: \Lambda_+\to \NN$ that satisfies $\rk(\ov{c}+\ov{d}) = \rk(\ov{c})+\rk(\ov{d})$.
\end{enumerate}
For the last piece of data, I will use $\rk_{\sigma}$ defined above. After giving the set $\RR\times \RR$ the alphabetical order and appending $\{\pm\infty\}$ with $-\infty<(a,b)<\infty$ for all $(a,b)\in \RR\times \RR$, one can define the following slope function:
$$
\sigma^{\epsilon}_{\mathring{\mu}}(\mathring{\ov{d}}_{r-1},\ov{d}) =\begin{cases}\renewcommand\arraystretch{1.7}
\begin{pmatrix}\deg_{\sigma}(\ov{d})/\rk_{\sigma}(\ov{d}) \\(\epsilon(\ov{d})  + \langle \mathring{\ov{\mu}}_{r-1},\mathring{\ov{d}}_{r-1}\rangle)/\rk_{\sigma}(\ov{d}) \end{pmatrix}&\text{ if }\ov{d}\neq 0\\
\\
 \infty&{\begin{array}{l}\textnormal{if }\ov{d}=0,\\\langle \mathring{\ov{\mu}}_{r-1}, \mathring{\ov{d}}_{r-1}\rangle>0\end{array}}\\
\\
 -\infty&\begin{array}{l}\textnormal{if }\ov{d}=0,\\\langle \mathring{\ov{\mu}}_{r-1}, \mathring{\ov{d}}_{r-1}\rangle\leq 0\end{array} \end{cases}\,.
$$
An object $\big(\mathring{\ov{U}}_{r-1}, \ov{U}\big)$ in $\mA$ is (semi)stable if every one of its subobjects $\big(\mathring{\ov{U}}'_{r-1}, \ov{U}'\big)$ satisfies $$\sigma^{\epsilon}_{\mathring{\mu}}\big(\mathring{\ov{U}}'_{r-1}, \ov{U}'\big)(\leq )<\sigma^{\epsilon}_{\mathring{\mu}}\big(\mathring{\ov{U}}_{r-1}, \ov{U}\big)\,.$$
For each $\big(\mathring{\ov{c}}_{r-1}, \ov{c}\big)$ satisfying $\ov{c}\neq 0$, the stability of a representation $\big(\mathring{\ov{U}}_{r-1}, \ov{U}\big)$ with this dimension vector can be alternatively expressed in terms of the character 
\begin{align*}
\vartheta^{\epsilon,\delta}_{\sigma,\mathring{\mu}}\big(\mathring{\ov{d}}_{r-1}, \ov{d}\big) = \deg_{\sigma}(\ov{d})+\delta \big(\epsilon(\ov{d})+&\langle \mathring{\ov{\mu}}_{r-1}, \mathring{\ov{d}}_{r-1}\rangle\big)\\-& \rk_{\sigma}(\ov{d})\frac{\deg_{\sigma}(\ov{c})+\delta \big(\epsilon(\ov{c})+\langle \mathring{\ov{\mu}}_{r-1}, \mathring{\ov{c}}_{r-1}\rangle\big)}{\rk_{\sigma}(\ov{c})}
\end{align*}
where $\delta>0$ is sufficiently small. Applying King's result shows that the moduli scheme of $\sigma^{\epsilon}_{\mathring{\mu}}$-stable objects with class $\big(\mathring{\ov{c}}_{r-1}, \ov{c}\big)$ in $\mA$ is projective whenever $\ov{c}\neq 0$, and there are no strictly semistables.
\end{proof}

 Note that $\mu$-stability on $\Rep(Q)$ can be expressed as an element $\sigma$ of $\Omega_{\Rep(Q)}$ with the stability function
$$
Z(\ov{d}) = -\sum_{v\in V}\mu_{v}d_v + i\sum_{v\in V}d_v\,.
$$
By repeating the arguments in the proof of Theorem \ref{thm:main} while working with an increasing $\mu$-stability, I conclude the second statement of the proposition.    
\end{proof}
\bibliography{refs.bib} 
\bibliographystyle{mybstfile.bst}
\end{document}